\documentclass[reqno,a4paper]{amsart}
\usepackage{graphicx}
\usepackage{a4wide} 
\usepackage{amsmath}
\usepackage{amssymb} 
\usepackage{nicefrac} 
\usepackage{aliascnt} 
\usepackage[pagebackref,dvips]{hyperref}
\usepackage[capitalise]{cleveref}
\usepackage[active]{srcltx} 
\usepackage{verbatim} 

\makeatletter
\@ifpackageloaded{mathptmx}{
    \DeclareMathAlphabet{\vect}{OT1}{ptm}{bx}{it}}{
    \newcommand{\vect}[1]{\boldsymbol{#1}}} 
\makeatother
\newcommand{\diff}[2]{\frac{\partial#1}{\partial#2}} 
\newcommand{\grad}[1][]{\nabla_{#1}} 
\renewcommand{\div}[1][]{\mathrm{div}_{#1}}
\newcommand{\lapl}[1][]{\Delta_{#1}} 
\renewcommand{\natural}{\mathbb{N}} 
\newcommand{\Real}{\mathbb{R}} 
\newcommand{\dd}{\,{\mathrm d}} 
\newcommand{\abs}[2][y]{\if#1y\left\fi\lvert#2\if#1y\right\fi\rvert} 
\newcommand{\norm}[2][y]{\if#1y\left\fi\lVert#2\if#1y\right\fi\rVert} 
\DeclareMathOperator*{\argmin}{arg\,min} 
\DeclareMathOperator*{\tp}{\textstyle\bigotimes} 
\newcommand{\BIGOP}[1]{\mathop{\mathchoice%
{\raise-0.22em\hbox{\huge $#1$}}%
{\raise-0.05em\hbox{\Large $#1$}}{\hbox{\large $#1$}}{#1}}}
\newcommand{\bigtimes}{\BIGOP{\times}} 
\newcommand{\compEmb}{\Subset} 
\newcommand{\supp}{\mathrm{supp}}
\newcommand{\ii}[1]{[#1]} 
\newcommand{\seq}[2]{\left(#1\right)_{#2}} 
\newcommand{\D}{\mathsf{D}}
\newcommand{\trans}{\mathrm{T}} 
\newcommand{\Tau}{\mathrm{T}} 

\newcommand{\vel}{\vect{u}} 
\newcommand{\conf}{\vect{q}} 
\newcommand{\ponf}{\vect{p}} 
\newcommand{\nb}{N\negmedspace+\negmedspace1} 
\newcommand{\wi}{{\rm Wi}} 
\newcommand{\maxw}{\mathsf{M}} 

\newcommand{\HH}{{\mathrm H}}
\newcommand{\LL}{{\mathrm L}}
\newcommand{\CC}{{\mathrm C}}
\newcommand{\CIC}{\CC^\infty_0} 
\newcommand{\HDM}{\HH(\D;\maxw)}
\newcommand{\LTOOM}{\LL^2_{1/\maxw}(\D)} 
\newcommand{\HOM}{\HH^1_\maxw(\D)}
\newcommand{\LTM}{\LL^2_\maxw(\D)}
\newcommand{\LOL}{\LL^1_{\loc}(\D)}
\newcommand{\HDiMi}{\HH(D_i;M_i)}
\newcommand{\HOMi}{\HH^1_{M_i}(D_i)}
\newcommand{\LTOOMi}{\LL^2_{1/M_i}(D_i)}
\newcommand{\LTMi}{\LL^2_{M_i}(D_i)}
\newcommand{\LTMj}{\LL^2_{M_j}(D_j)}

\newcommand{\tcdot}{\vect{\cdot}} 
\newcommand{\sumi}{\sum_{i=1}^N} 
\newcommand{\sumj}{\sum_{j=1}^N} 
\newcommand{\allibutj}{\substack{i=1\\i\neq j}}
\newcommand{\psp}[1]{\sp{(#1)}} 
\newcommand{\sspsp}[1]{\sp{\smash{(#1)}}} 
\newcommand{\onehalf}{{\textstyle\frac{1}{2}}}
\newcommand{\loc}{\mathrm{loc}}
\newcommand{\smartxi}{\bigg(\frac{s}{p_d}\bigg)}
\newcommand{\lsp}{\thinspace} 

\newtheorem{theorem}{Theorem}[section]

\newaliascnt{lemma}{theorem}
\newtheorem{lemma}[lemma]{Lemma}
\aliascntresetthe{lemma}

\newaliascnt{corollary}{theorem}
\newtheorem{corollary}[corollary]{Corollary}
\aliascntresetthe{corollary}

\theoremstyle{remark}
\newtheorem{remark}{Remark}

\theoremstyle{definition}

\newtheorem{hypothesis}{Hypothesis}

\Crefname{hypothesis}{Hypothesis}{Hypotheses}
\newtheoremstyle{Algorithmicky}{.5\baselineskip plus .2\baselineskip minus .2 \baselineskip}{.5\baselineskip plus .2\baselineskip minus .2 \baselineskip}{\rmfamily\itshape}{}{\bfseries}{}{ }{\thmname{#1}\thmnumber{ #2}.\thmnote{ (\mdseries{#3})}}
\theoremstyle{Algorithmicky}
\newtheorem{alg}{Algorithm}

\Crefname{alg}{Algorithm}{Algorithms}

\numberwithin{equation}{section}
\numberwithin{figure}{section}

\newcounter{constant}
\makeatletter
\newcommand{\const}{\@ifstar
    {\constStar}%
    {\constNoStar}%
} 
\makeatother
\newcommand{\constNoStar}{\addtocounter{constant}{1}C_{\arabic{constant}}}
\newcommand{\constStar}{C_{\arabic{constant}}}
\newcommand{\resetconst}{\setcounter{constant}{0}} 
\newcounter{shiftedConstant}
\newcommand{\shiftedconst}[1]{\setcounter{shiftedConstant}{\value{constant}}\addtocounter{shiftedConstant}{#1}C_{\arabic{shiftedConstant}}} 

\begin{document}
\title{Greedy approximation of high-dimensional Ornstein--Uhlenbeck operators}
\thanks{The first author acknowledges a doctoral scholarship from the Chilean government's \emph{Comisi\'on Nacional de Investigaci\'on Cient\'{\i}fica y Tecnol\'ogica}. The second author was supported by the EPSRC Science and Innovation award to the Oxford Centre for Nonlinear PDE (EP/E035027/1).}
\author{Leonardo E. Figueroa \and Endre S\"uli}
\address{L. E. Figueroa (corresponding author)\\
    Centro de Investigaci\'on en Ingenier\'{\i}a Matem\'atica (CI\textsuperscript{2}MA), Universidad de Concepci\'on, Casilla 160-C, Concepci\'on, Chile.}
\email{leonardo@leonardofigueroa.org}
\address{E. S\"uli\\
    Mathematical Institute, University of Oxford, 24--29 St Giles', Oxford OX1 3LB, United Kingdom.}
\email{Endre.Suli@maths.ox.ac.uk}


\begin{abstract}
We investigate the convergence of a nonlinear approximation method introduced
by Ammar et al.\ (J.\ Non-Newtonian Fluid Mech.\ 139:153--176, 2006) for the
numerical solution of high-dimensional Fokker--Planck equations featuring in
Navier--Stokes--Fokker--Planck systems that arise in kinetic models of dilute
polymers. In the case of Poisson's equation on a rectangular domain in
$\Real^2$, subject to a homogeneous Dirichlet boundary condition, the
mathematical analysis of the algorithm was carried out recently by Le Bris,
Leli\`evre and Maday (Const.\ Approx.\ 30:621--651, 2009), by exploiting its
connection to greedy algorithms from nonlinear approximation theory, explored,
for example, by DeVore and Temlyakov (Adv.\ Comput.\ Math.\ 5:173--187, 1996);
hence, the variational version of the algorithm, based on the minimization of a
sequence of Dirichlet energies, was shown to converge. Here, we extend
the convergence analysis of the pure greedy and orthogonal greedy algorithms
considered by Le Bris et al.\ to a technically more complicated
situation, where the Laplace operator is replaced by an
Ornstein--Uhlenbeck operator of the kind that appears in
Fokker--Planck equations that arise in bead-spring chain type kinetic polymer
models with finitely extensible nonlinear elastic potentials, posed on a
high-dimensional Cartesian product configuration space $\D = D_1 \times \dotsm
\times D_N$ contained in $\Real^{N d}$, where each set $D_i$, $i=1, \dotsc, N$,
is a bounded open ball in $\Real^d$, $d = 2, 3$.
\end{abstract}

\maketitle

\section{Introduction}
High-dimensional partial differential equations are ubiquitous in mathematical
models in science, engineering and finance. They arise in a number of areas,
including, for example, kinetic theory, molecular dynamics, quantum mechanics,
and uncertainty quantification based on polynomial chaos expansions, to name
only a few.

The purpose of the present paper is to explore the convergence of a numerical
algorithm that was recently proposed in the engineering literature in a
succession of papers by Ammar, Mokdad, Chinesta, Keunings and collaborators
\cite{AMChK06,AMChK07,ANDGCCh,GAChC,ChALK}, for the numerical solution of
high-dimensional Fokker--Planck equations in kinetic models of polymeric
fluids under the names \emph{Separated Representation} and \emph{Proper
Generalized Decomposition}. A variant with a discretization based on spectral
methods instead of the finite element methods preferred by Ammar et al.\ was
presented by Leonenko and Phillips \cite{LP:2009}. A similar method was
considered independently by Nouy
\cite{Nouy:2007,Nouy:2008} and Nouy \& Le Ma{\^{\i}}tre \cite{NL} under the
name \emph{Power type Generalized Spectral Decomposition}, for the numerical
solution of stochastic partial differential equations, although the historical
roots of the technique can be traced back to the work of Schmidt
\cite{Schmidt}. Ammar et al.\ and Nouy report that the algorithm performs well
in numerical experiments and comment that it extends to a large variety of
partial differential equations.

In the simplified mathematical setting of Poisson's equation $-\lapl u = f$
posed on the rectangular domain $\Omega = \Omega_x \times \Omega_y$, where
$\Omega_x$ and $\Omega_y$ are bounded open subintervals of $\Real$, subject to
a homogeneous Dirichlet boundary condition on $\partial \Omega$, the
convergence of the algorithm was shown in a recent paper by Le Bris, Leli\`evre
and Maday \cite{LLM}, by drawing on connections with greedy algorithms from
nonlinear approximation theory (cf.\ DeVore and Temlyakov \cite{DvT}). In
\cite{LLM}, the solution was represented as a sum
\begin{equation}\label{tensor1}
u(x, y) =\sum_{n\geq 1} r_n(x)\, s_n(y)
\end{equation}
by iteratively determining functions $x\in\Omega_x \mapsto r_n(x)$ and $y \in \Omega_y \mapsto s_n(y)$,
$n \geq 1$, such that
for all $n$, the product $(x,y)\in \Omega \mapsto r_n(x)\,s_n(y)$ is the best approximation in the norm
of the Sobolev space $\HH^1_0(\Omega)$ to the solution $(x,y) \in \Omega \mapsto v(x, y)$ of the Poisson
equation
\[-\Delta v(x, y) = f (x, y)+\lapl\left(\sum_{k\leq n-1}
r_k(x) s_k(y)\right),\]
subject to a homogeneous Dirichlet boundary condition, in terms of a single
function of the factorized form $r(x)\,s(y)$; Le Bris et al.\ thus show that it is
possible to give a sound mathematical basis to the algorithm proposed by Ammar
et al., provided that one considers a variational form of the approach that
manipulates minimizers of Dirichlet energies instead of stationary points to
the associated Euler--Lagrange equations (in the follow-up paper
\cite{CEL:2011} by Canc\`es, Ehrlacher and Leli\`evre it was further shown that
one can also work with local---yet still energy-decreasing---minimizers
provided that one stays within the two-fold tensor product setting of
\eqref{tensor1}). In order to reformulate the approach
in such a variational setting, the arguments in \cite{LLM} crucially rely on
the fact that the Laplace operator is self-adjoint, and as noted by the authors
of \cite{LLM}, the analysis does not apply exactly to the actual implementation
of the method as described in the papers by Ammar et al., where stationary
points of the Euler--Lagrange equations associated with the Dirichlet energies
are computed instead. Indeed, since minimizers of Dirichlet energies in the
approach of Le Bris et al.\ on the one hand and stationary points
of the associated Euler--Lagrange equations in the approach of Ammar et al.\ on
the other are each sought in \emph{nonlinear} manifolds embedded in a Sobolev
space, rather than over the entire Sobolev space (which is a normed
\emph{linear} space), the two approaches are not equivalent. The authors of
\cite{LLM} also comment that: ``Likewise, it is unclear to us how to provide a
mathematical foundation of the approach for nonvariational situations, such as
an equation involving a differential operator that is not self-adjoint.'' This
latter remark is particularly pertinent in the context of Fokker--Planck
equations for kinetic bead-spring chain models for dilute polymers, of the kind
considered by Ammar et al., where the differential operator in configuration
space featuring in the Fokker--Planck equation, a generalized
Ornstein--Uhlenbeck operator, is a non-self-adjoint elliptic operator with a
drift term that involves an unbounded potential.

It is this last point that the present paper is aimed at addressing: we shall
be concerned with the numerical approximation of high-dimensional
Fokker--Planck equations that arise in bead-spring chain type kinetic models of
dilute polymers on the Cartesian product domain $\Omega \times \D$, where
$\Omega \subset \Real^d$, $d = 2, 3$, is the physical (flow) domain, and the
\emph{configuration space} $\D$ is the $N$-fold Cartesian product
$\bigtimes_{i=1}^N D_i$ of sets $D_i \subset \Real^d$, $i = 1, \dotsc, N$, $N
\geq 2$, each of which is a bounded open ball in $\Real^d$. Here, $N$ denotes
the number of springs connecting, in a linear fashion, the $\nb$ beads in the
bead-spring chain model (cf.\ \cref{modelPoly}).
Proceeding as in \cite{BS:2007,BS:2008,BS:2009,BS:2011a,BS:2011b}, we
rewrite the Ornstein--Uhlenbeck operator, a
non-self-adjoint elliptic operator with respect to the configuration space
variable $\conf$ featuring in the Fokker--Planck equation whose drift term
contains an unbounded potential, as a degenerate, but now self-adjoint,
elliptic operator on a Maxwellian-weighted Sobolev space. We then perform a
nonlinear approximation of the analytical solution
$\psi\colon(\conf_1,\dotsc,\conf_N) \in \D \mapsto \psi(\conf_1, \dotsc,
\conf_N)$ to this high-dimensional degenerate elliptic boundary-value problem on
the appropriate Maxwellian-weighted Sobolev space by separated representations of the form
\[\sum_{k = 1}^K \prod_{i=1}^N \psi_k\psp{i}(\conf_i),\]
where the factors
$\psi_k\psp{i}$, $k = 1, \dotsc ,K$, are defined on the $d$-dimensional domain
$D_i$, $i = 1, \dotsc, N$. Instead of being selected from an \emph{a
priori} fixed set, the factors $\psi_k^{(i)}$, $i=1, \dotsc, N$, are obtained,
$N$ at a time, for each $k \in \{1, \dotsc, K\}$, as the best approximation (in
a sense to be made precise in \cref{sec:SR}) among all possible such
factors. The (potentially large) number of terms $K$ is likewise not fixed in advance, but
depends on a termination criterion.

The paper is structured as follows. After introducing our notational
conventions and formulating briefly an alternating direction scheme that
separates, by a fractional step method, the full Fokker--Planck equation into a
low-dimensional physical space part and a high-dimensional configuration space
part, we will concentrate on the latter problem. The central difficulty in the
numerical solution of the configuration space problem is the presence of the
high-dimensional Ornstein--Uhlenbeck operator, a non-self-adjoint elliptic
operator whose drift term contains an unbounded potential. In \cref{sec:operator} we show that
the configuration space problem can be restated, in a Maxwellian-weighted
Sobolev space, as the weak formulation of a symmetric degenerate elliptic
boundary-value problem on the high-dimensional configuration space $\D$.
\Cref{sec:SR} is devoted to the description of a separated representation
strategy for the problem, in the spirit of Le Bris et al.\
\cite{LLM}. Following \cite{LLM}, we consider a pure greedy algorithm and an
orthogonal greedy algorithm. \Cref{sec:convergence} concentrates on the
convergence of the two algorithms. We shall characterize the convergence rates
of the two greedy algorithms by invoking abstract convergence results due to
DeVore and Temlyakov \cite{DvT}. In \cref{sec:characterization}, we give
explicit necessary and sufficient conditions, in terms of Maxwellian-weighted
Sobolev spaces, for membership of the space of DeVore and Temlyakov in the case
of our degenerate elliptic problem. In \cref{sec:conclusions}, we provide
some conclusions and possible directions for further work.

At an abstract level, our convergence proof follows the arguments in
\cite{LLM}; however, the verification of certain key properties of the function
spaces involved, on the one hand, and the characterization of verifiable
sufficient conditions under which the predicted convergence rates of the two
greedy algorithms considered are observed, on the other, for the
high-dimensional degenerate elliptic problem studied herein are considerably
more complicated than in the case of Poisson's equation studied in \cite{LLM}.
The former is mostly based on tensorizing the corresponding results for the
function spaces associated with the single-spring case (i.e., the
\emph{dumbbell}) and the latter relies on shift-theorems for degenerate
elliptic operators in Maxwellian-weighted Sobolev spaces and delicate results
from the spectral theory of self-adjoint degenerate elliptic operators, which
we were unable to find in the literature; these are described in
\cref{sec:characterization} and \cref{sec:ev-asymptotics}, respectively.
Appendices \ref{sec:CPAIL-results} and \ref{sec:distrResults} collect a number
of technical results that are used throughout the paper.

\subsection{Notation} We denote by $\ii{k}$ the integer interval $\{i \in
\natural \colon 1 \leq i \leq k\}$. We shall denote sequences and
arrangements of elements $a_i$ indexed by indices $i$ in an index set
$\mathcal{I}$ by $\seq{a_i}{i \in \mathcal{I}}$.

We shall write $\conf = (\conf_1, \dotsc, \conf_N) \in D_1 \times \dotsm \times
D_N = \bigtimes_{i\in\ii{N}} D_i =: \D$. Given $N$ real-valued functions $f_i$,
each defined on the corresponding set $D_i$, we denote by $\tp_{i\in\ii{N}} f_i$
their \emph{tensor product}; i.e., the function
\[ \conf \in \D \mapsto \prod_{i\in\ii{N}} f_i(\conf_i). \]
We extend this notation in three ways.
Firstly, as the tensor-product operation is order-dependent, we will use
subscripts on the $\otimes$ and the $\tp$ signs to denote where on $\conf \in
\D$ the function, or functions, following them act; e.g., $\tp_{i \in \ii{N}
\setminus \{ j \}} f_i \otimes_j f_j$ evaluated on $\conf \in \D$ is
$f_j(\conf_j)\prod_{i \in \ii{N} \setminus \{ j \}} f_i(\conf_i)$.
Secondly, we will
use the same notation for the sets resulting from the tensor products of
members of function spaces: suppose that $F_i$ is a nonempty set of real-valued
functions defined on $D_i$, $i \in \ii{N}$; we then write $\tp_{i\in\ii{N}} F_i
:= \{ \tp_{i\in\ii{N}} f_i \colon f_i \in F_i,\ i \in \ii{N} \}$. Thirdly, if
exactly one of the factors is vector-valued, the products involving it at the
time of evaluation must be interpreted as scalar-vector products implying that
the resulting tensor product will be vector-valued too.

The symbol $\compEmb$ will stand for the compact embedding relation. The
support of a real-valued function $f$ will be denoted by $\supp(f)$.


Given a measurable and almost everywhere positive real-valued function $w$
defined on an open set $E \subset \Real^n$; i.e., a \emph{weight}, we denote by
$\LL^2_w(E)$ the Lebesgue space of square-integrable functions with respect to
the weight $w$, equipped with its usual norm,
\[ \norm{\varphi}_{\LL^2_w(E)} := \bigg(\int_E \abs{\varphi}^2 w\bigg)^{1/2}. \]
We also define the $w$-weighted Sobolev space $\HH^m_w(E)$ and its norm $\norm{\tcdot}_{\HH^m_w(E)}$ by
\begin{gather*}
\HH^m_w(E) := \left\{ \varphi \in \LL^2_w(E) \cap \LL^1_{\loc}(E) \colon \partial_\alpha \varphi \in \LL^2_w(E),\ \abs{\alpha}\leq m \right\},\\
\norm{\varphi}_{\HH^m_w(E)} := \bigg(\sum_{\abs{\alpha} \leq m} \norm{\partial_\alpha \varphi}_{\LL^2_w(E)}^2\bigg)^{1/2} \quad \forall\,\varphi \in \HH^m_w(E).
\end{gather*}

We shall suppose henceforth that $\Omega$ is a bounded open set in $\Real^d$
with a sufficiently regular (say, Lipschitz continuous) boundary, and denote by
$\mathfrak{n}_{\vect{x}}$ and $\mathfrak{n}_{\conf_i}$ the unit outward normal
vector defined (a.e.\ with respect to the surface measure) on $\partial\Omega$
and $\partial{D_i}$, $i \in \ii{N}$, respectively.

\subsection{Fokker--Planck equation}

\begin{figure}[t]
\centering\includegraphics[scale=0.7]{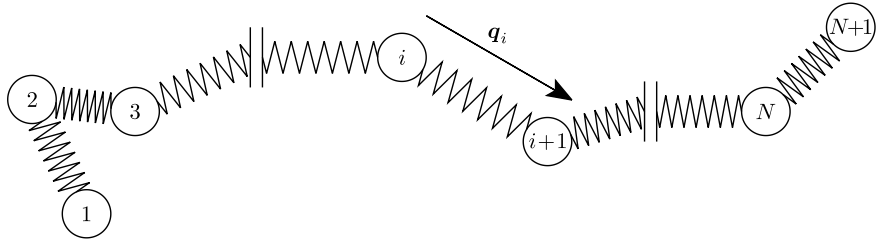}
\caption{Bead-spring chain with $N$ springs and $\nb$ beads. Adapted from
Figure 11.4-1 of \cite{Bird2}}
\label{modelPoly}
\end{figure}

The \emph{spring forces} in the model are given by functions $\vect{F}_i \colon
D_i \rightarrow \Real^d$, which have the form $\vect{F}_i(\ponf) = U_i'(\onehalf
\abs{\ponf}^2) \ponf$, $\ponf \in D_i := B(0,\sqrt{b_i}) \subset \Real^d$, $b_i
> 0$, $i \in \ii{N}$, and the $U_i\colon[0,b_i/2) \rightarrow \Real$, the
\emph{spring potentials}, are such that $U_i(s) \to +\infty$ as $s \to b_i/2_-$.
It follows that $\vect{F}_i(\ponf) = -\vect{F}_i(-\ponf)$ for all $\ponf \in
D_i$. Typical examples include the \emph{FENE (Finitely Extensible Nonlinear
Elastic) model} \cite{Warner:1972} with
\begin{equation}\label{FENE-model}
U_i(s) = - \frac{b_i}{2}\ln \left(1 - \frac{2s}{b_i}\right) \qquad\text{and}\qquad \vect{F}_i(\conf_i) = \frac{1}{1 - \abs{\conf_i}^2\!/b_i} \conf_i,
\end{equation}
where $b_i>0$ is a parameter, and \emph{Cohen's Pad\'e approximant to
the Inverse Langevin (CPAIL) model} \cite{Cohen} with
\begin{equation}\label{CPAIL-model}
U_i(s) = \frac{s}{3} - \frac{b_i}{3}\ln\left(1 - \frac{2 s}{b_i}\right) \qquad\text{and}\qquad \vect{F}_i(\conf_i) = \frac{ 1 - \abs{\conf_i}^2\!/(3 b_i) }{ 1 - \abs{\conf_i}^2\!/b_i }\conf_i,
\end{equation}
where $b_i>0$ is again a parameter. We note in passing that both of these force laws are approximations to the \emph{Inverse Langevin force law} \cite{KG}
\[ \vect{F}_i(\conf_i) = \frac{\sqrt{b_i}}{3} L^{-1}\left( \frac{\abs{\conf_i}}{\sqrt{b_i}} \right) \frac{\conf_i}{\abs{\conf_i}},\]
where the $\emph{Langevin function}$ $L$ is defined by $L(t) := \coth(t)-1/t$
on $[0,\infty)$. As $L$ is strictly monotonic increasing on $[0,\infty)$ and tends to
$1$ as its argument tends to $\infty$, it follows that the function
$\abs{\conf_i} \in [0,\sqrt{b_i}) \mapsto L^{-1}(|\conf_i|/\sqrt{b_i})
 \in [0,\infty)$ is strictly monotonic increasing, with a vertical asymptote at $\abs{\conf_i} = \sqrt{b_i}$.

\begin{remark} An important spring force model, which is excluded from our
considerations, is the simple \emph{Hookean model} described by
\[ D_i = \Real^d, \qquad U_i(s) = s \qquad\text{and}\qquad \vect{F}_i(\conf_i) = \conf_i. \]
However, in many practically relevant flow regimes the physically unrealistic allowance
of the Hookean model for indefinitely extended springs outweighs its mathematical
convenience.
\end{remark}

The Fokker--Planck equation under consideration for the probability density
function $\psi$ has the following form (cf. \cite{BS:2007,BS:2008,BS:2009,BS:2011b,BS:2011a}):
\begin{subequations}\label{FP-full}
\begin{multline}\label{FP-full-body}
    \diff{\psi}{t} + \div[\vect{x}](\vel\psi)
    + \sumi \div[\conf_i]\bigg[ (\grad[\vect{x}]\vel)\conf_i\psi -
    \frac{1}{4\wi} \sumj A_{ij} (\vect{F}_j(\conf_j)\psi + \grad[\conf_j]\psi) \bigg]\\
    = \frac{(l_0/L_0)^2}{4\wi(\nb)}\lapl[\vect{x}]\psi,
    \quad (\vect{x},\conf,t) \in \Omega \times \D \times (0,T],
\end{multline}
with initial and no-flux boundary conditions
\begin{align}
\psi(\tcdot,\tcdot,0) & = \psi_0, && (\vect{x},\conf) \in \Omega \times \D,\\
\frac{(l_0/L_0)^2}{4\wi(\nb)}\grad[\vect{x}]\psi \cdot \mathfrak{n}_{\vect{x}} & = 0, && (\vect{x},\conf,t) \in \partial\Omega \times \D \times (0,T],
\end{align}
and
\begin{equation}\label{FP-full-bc-q}
\bigg[(\grad[\vect{x}] \vel) \conf_i\psi - \frac{1}{4\wi}\sumj A_{ij}(\vect{F}_j(\conf_j)\psi + \grad[\conf_j]\psi)\bigg] \cdot \mathfrak{n}_{\conf_i} = 0,\quad
i \in \ii{N},\ (\vect{x},\conf,t) \in \Omega \times \partial \D \times (0,T].
\end{equation}
\end{subequations}
Here, $\vel \colon \overline\Omega \times [0,T] \rightarrow \Real^d$ is the
flow velocity, $\wi := \lambda U_0/L_0$ is the (nondimensional) Weissenberg
number, $l_0$ is the characteristic length-scale of a spring, $\lambda$ is the
characteristic relaxation time of a spring and $L_0$ and $U_0$ are the
characteristic macroscopic length and velocity, respectively (thus, $\wi$ is
the ratio of the microscopic to macroscopic time scales). The matrix $A =
\seq{A_{ij}}{i,j\in\ii{N}}$ is symmetric and positive definite; we denote the
smallest eigenvalue of $A$ by $\lambda_{\rm min}$.

We remark that the boundary condition \eqref{FP-full-bc-q} is an ensemble of
$N$ boundary conditions, which collectively account for the full $(N d -
1)$-dimensional measure of $\partial{\D}$.

We define the partial Maxwellians $M_i$ and the (full) Maxwellian
$\maxw$ by
\begin{equation}\label{partial-maxw}
M_i(\ponf) := Z_i^{-1} \exp\left( -U_i\left( \onehalf \abs[n]{\ponf}^2 \right) \right), \quad \ponf \in D_i, \quad i \in \ii{N};
\end{equation}
\begin{equation}\label{full-maxw}
\maxw(\conf) := \prod_{i=1}^N M_i(\conf_i), \quad \vect{q} \in \D;
\end{equation}
that is, $\maxw = \tp_{i\in\ii{N}} M_i$. Here, each $Z_i$ is a positive
constant chosen so that $\int_{D_i} M_i = 1$ (we can do so
because of \cref{hyp:potential}, below). Thereby, $\int_{\D} \maxw
= 1$. We note that since $U_i$ is assumed to tend to $+\infty$
as $\conf_i$ approaches $\partial D_i$, the corresponding partial Maxwellian $M_i$
tends to $0$ as $\conf_i$ approaches $\partial D_i$, $i \in [N]$; consequently,
$\maxw$ tends to $0$ as $\conf$ approaches $\partial \D$.
The fact that the Maxwellian factorizes---which comes from the
fact that the energy stored in the chain is the sum of the potential energies
stored in each spring---will be crucial throughout the rest of this paper. For
a start, this fact allows us to write
\begin{equation}\label{Kolmogorov-regularization}
\vect{F}_j(\conf_j) \psi + \grad[\conf_j]\psi
= \psi \grad[\conf_j] U_j(\onehalf\abs[n]{\conf_j}^2) + \grad[\conf_j]\psi
= \maxw \grad[\conf_j]\left(\frac{\psi}{\maxw}\right).
\end{equation}
Multiplying \eqref{FP-full-body} by $\varphi/\maxw$, using
\eqref{Kolmogorov-regularization} and (formally) integrating by parts, the
corresponding weak form of \eqref{FP-full} is: Find $\psi =
\psi(\vect{x},\conf,t)$ such that
\begin{multline}\label{FP-weak}
\int_{\Omega \times \D} \Bigg\{ \diff{\psi}{t} \frac{\varphi}{\maxw}
+ \div[\vect{x}](\vel\psi)\frac{\varphi}{\maxw}
- \sumi \bigg[ (\grad[\vect{x}]\vel)\conf_i\psi - \sumj \frac{A_{ij}}{4\wi} \maxw \grad[\conf_j]\left(\frac{\psi}{\maxw}\right) \bigg] \cdot \grad[\conf_i] \left(\frac{\varphi}{\maxw}\right)\\
+ \frac{(l_0/L_0)^2}{4\wi(\nb)} \grad[\vect{x}]\psi \cdot \grad[\vect{x}] \varphi \frac{1}{\maxw} \Bigg\} = 0
\end{multline}
for all $\varphi = \varphi(\vect{x},\conf)$ in a suitable function space.

For the sake of convenience we define the following bilinear forms:
\begin{gather}\label{bilinear-forms}
\tilde{\mathcal{T}}(\vel;\sigma,\tau) := \int_{\Omega \times \D} \div[\vect{x}](\vel\sigma)\frac{\tau}{\maxw}, \qquad
\tilde{\mathcal{K}}(\sigma,\tau) := \frac{(l_0/L_0)^2}{4\wi(\nb)} \int_{\Omega \times \D}\grad[\vect{x}]\sigma \cdot \grad[\vect{x}] \tau \frac{1}{\maxw},\\
\mathcal{T}(\vel;\sigma,\tau) := -\int_{\Omega \times \D} \sumi (\grad[\vect{x}]\vel)\conf_i\sigma \cdot \grad[\conf_i]\left(\frac{\tau}{\maxw}\right),\\
\mathcal{K}(\sigma,\tau) := \int_{\Omega \times \D} \sumi\sumj \frac{A_{ij}}{4\wi} \maxw \grad[\conf_j]\left(\frac{\sigma}{\maxw}\right) \cdot \grad[\conf_i]\left(\frac{\tau}{\maxw}\right).
\end{gather}
Then, \eqref{FP-weak} can be written concisely as
\begin{equation}\label{FP-weak-2}
\bigg\langle \diff{\psi}{t}, \varphi/\maxw \bigg\rangle + \tilde{\mathcal{T}}(\vel;\psi,\varphi) + \tilde{\mathcal{K}}(\psi,\varphi) + \mathcal{T}(\vel;\psi,\varphi) + \mathcal{K}(\psi,\varphi) = 0
\end{equation}
for all $\varphi = \varphi(\vect{x},\conf)$ in a suitable function space. We
note that $\tilde{\mathcal{T}}$ and $\tilde{\mathcal{K}}$ involve partial
derivatives of their arguments with respect to the \emph{spatial} variable
$\vect{x}$ only. Analogously, $\mathcal{T}$ and $\mathcal{K}$ involve partial
derivatives of their arguments with respect to the \emph{configuration} space
variable $\conf$ only. This 
motivates the use of the alternating
direction scheme based on operator splitting whose informal description is
given in the next subsection.

\subsection{Alternating direction scheme} Let $\Delta t$ be such that $M :=
T/\Delta t \in \natural$ and define $t^n := n \Delta t$ for $n \in \{0,
\dotsc, M\}$. We will consider the following \emph{alternating-direction}
semidiscretization of \eqref{FP-weak}: We initialize the scheme by defining
$\psi^0 := \psi_0$; for $n \in \{ 0, \dotsc, M - 1 \}$ and then define the
`intermediate' function $\psi^{n+1/2}$ and the approximation $\psi^{n+1}$
to $\psi(t^{n+1}, \tcdot, \tcdot)$, respectively, by
\begin{subequations}\label{FP-weak-SD}
\begin{multline}\label{FP-weak-SD-x}
\bigg\langle \frac{\psi^{n+1/2} - \psi^n}{\Delta t/2}, \frac{\varphi}{\maxw} \bigg\rangle
+ \tilde{\mathcal{T}}(\vel(\tcdot,t^{n+1});\psi^{n+1/2},\varphi) + \tilde{\mathcal{K}}(\psi^{n+1/2},\varphi)\\
= -\mathcal{T}(\vel(\tcdot,t^n);\psi^n,\varphi) - \mathcal{K}(\psi^n,\varphi)
\end{multline}
and
\begin{multline}\label{FP-weak-SD-q}
\bigg\langle \frac{\psi^{n+1} - \psi^{n+1/2}}{\Delta t/2}, \frac{\varphi}{\maxw} \bigg\rangle
+ \mathcal{K}(\psi^{n+1},\varphi) = -\mathcal{T}(\vel(\tcdot,t^n);\psi^n,\varphi)\\
- \tilde{\mathcal{T}}(\vel(\tcdot,t^{n+1});\psi^{n+1/2},\varphi) - \tilde{\mathcal{K}}(\psi^{n+1/2},\varphi),
\end{multline}
\end{subequations}
for all $\varphi=\varphi(\vect{x},\conf)$ in a suitable function space. In
\eqref{FP-weak-SD-x} the spatial bilinear forms $\tilde{\mathcal T}$ and
$\tilde{\mathcal K}$ are treated implicitly while the configuration space
bilinear forms $\mathcal{T}$ and $\mathcal{K}$ are treated explicitly. In
\eqref{FP-weak-SD-q} the spatial bilinear forms $\tilde{\mathcal T}$ and
$\tilde{\mathcal K}$ and the configuration space bilinear form $\mathcal{T}$
associated with the drag term are treated explicitly, while the bilinear
form $\mathcal{K}$ is treated implicitly.

Let $\seq{ (\conf\psp{k}, w_{\D}\psp{k}) }{k \in \ii{Q_{\D}} }$ and $\seq{
(\vect{x}\psp{k}, w_\Omega\psp{k}) }{ k \in \ii{Q_\Omega} }$ be $\frac{1}{\maxw}$- and
$1$-weighted quadrature rules on $\D$ and $\Omega$, respectively. We then
approximate \eqref{FP-weak-SD-x} by performing numerical integration over the
configuration space, which results in
\begin{multline*}
\sum_{k=1}^{Q_{\D}}\! w_{\D}\psp{k}\!\!\! \int_{\Omega}\!\! \frac{\psi^{n+1/2}(\tcdot,\conf\psp{k}) - \psi^n(\tcdot,\conf\psp{k}) }{\Delta t/2} \varphi(\tcdot,\conf\psp{k})\displaybreak[0]\\
+ \sum_{k=1}^{Q_{\D}}\! w_{\D}\psp{k}\!\! \int_{\Omega}\! \div[\vect{x}]\!\left( \vel(\tcdot,t^{n+1}) \psi^{n+1/2}(\tcdot,\conf\psp{k}) \right)\! \varphi(\tcdot,\conf\psp{k})\displaybreak[0]\\
+ \sum_{k=1}^{Q_{\D}} w_{\D}\psp{k} \frac{(l_0/L_0)^2}{4\wi(\nb)} \int_{\Omega} \grad[\vect{x}] \psi^{n+1/2}(\tcdot,\conf\psp{k}) \cdot \grad[\vect{x}] \varphi(\tcdot,\conf\psp{k})\displaybreak[0]\\
\shoveleft{\approx \sum_{k=1}^{Q_{\D}} w_{\D}\psp{k} \int_{\Omega} \sumi \maxw(\conf\psp{k}) (\grad[\vect{x}]\vel(\tcdot,t^n))\conf\psp{k}_i\psi^n(\tcdot,\conf\psp{k}) \cdot \grad[\conf_i] \left.\left( \frac{\varphi}{\maxw} \right)\right|_{(\tcdot,\conf\psp{k})}}\displaybreak[0]\\
- \sum_{k=1}^{Q_{\D}} w_{\D}\psp{k} \int_{\Omega} \sumi\sumj \frac{A_{ij}}{4\wi}\maxw(\conf\psp{k}) \left.\grad[\conf_j] \left( \frac{\psi^n}{\maxw} \right)\right|_{(\tcdot,\conf\psp{k})} \cdot \left.\grad[\conf_i] \left( \frac{\varphi}{\maxw} \right)\right|_{(\tcdot,\conf\psp{k})},
\end{multline*}
for all $\varphi=\varphi(\vect{x},\conf)$ in a suitable function space. Here, the symbol
$\approx$ denotes equality, up to quadrature errors. By
selecting $Q_{\D}$ linearly independent functions $\zeta_{(m)}$, $m \in
\ii{Q_{\D}}$, of $\conf \in \D$ such that $\zeta_{(m)}(\conf^{(k)}) =
\delta_{km}$, $k,m \in \ii{Q_{\D}}$, and taking successively $\varphi =
\varphi_{(m)}$, where $\varphi_{(m)}(\vect{x},\conf) := \chi(\vect{x})
\zeta_{(m)}(\conf)$, in the equality above, we obtain a total of $Q_{\D}$
independent variational problems, each posed over the $d$-dimensional domain
$\Omega$, of the form:
\begin{multline}\label{FP-weak-SD-x-2}
\frac{1}{\Delta t/2} \int_{\Omega} \psi^{n+1/2}(\tcdot,\conf\psp{m}) \chi
+ \int_{\Omega} \div[\vect{x}]\left( \vel(\tcdot,t^{n+1}) \psi^{n+1/2}(\tcdot,\conf^{(m)}) \right) \chi\\
+ \frac{(l_0/L_0)^2}{4\wi(\nb)} \int_{\Omega} \grad[\vect{x}] \psi^{n+1/2}(\tcdot,\conf\psp{m}) \cdot \grad[\vect{x}] \chi
\approx \frac{1}{\Delta t/2}\int_{\Omega} \psi^n(\tcdot,\conf\psp{m}) \chi 
\\
+ \frac{1}{w_{\D}^{(m)}}\sum_{k=1}^{Q_{\D}}w_{\D}^{(k)}\!\Bigg[\int_{\Omega} \sumi \maxw(\conf\psp{k}) (\grad[\vect{x}]\vel(\tcdot,t^n))\conf\psp{k}_i\psi^n(\tcdot,\conf\psp{k}) \cdot\! \left.\grad[\conf_i]\!\left(\frac{\zeta_{(m)}}{\maxw} \right)\right|_{(\tcdot,\conf\psp{k})}\!\chi\\
- \int_{\Omega} \sumi\sumj \frac{A_{ij}}{4\wi}\maxw(\conf\psp{k}) \grad[\conf_j] \left.\left( \frac{\psi^n}{\maxw} \right)\right|_{(\tcdot,\conf\psp{k})} \cdot \grad[\conf_i] \left.\left( \frac{\zeta_{(m)}}{\maxw} \right)\right|_{(\tcdot,\conf\psp{k})}\chi\Bigg]\\
=: \mathfrak{M}_{(m)}(\psi^{n}; \chi) \qquad \forall\, m \in \ii{Q_{\D}},
\end{multline}
for all $\chi = \chi(\vect{x})$ in a suitable function space, where each
$\mathfrak{M}_{(m)}(\psi^{n}; \tcdot)$, $m \in \ii{Q_{\D}}$, is a linear
functional. Thus, \eqref{FP-weak-SD-x-2} amounts to solving $Q_{\D}$
mutually independent linear convection-diffusion problems over $\Omega$.

In turn, we can approximate \eqref{FP-weak-SD-q} by performing numerical quadrature over
$\Omega$, resulting in
\begin{multline*}
\sum_{k=1}^{Q_\Omega} w_\Omega\psp{k} \int_{\D} \frac{\psi^{n+1}(\vect{x}\psp{k},\tcdot) - \psi^{n+1/2}(\vect{x}\psp{k},\tcdot) }{\Delta t/2} \frac{\varphi(\vect{x}\psp{k},\tcdot)}{\maxw}
\\
- \sum_{k=1}^{Q_\Omega} w_\Omega\psp{k} \int_{\D} \sumi (\grad[\vect{x}]\vel(\vect{x}\psp{k},t^n))\conf_i\psi^n (\vect{x}\psp{k},\tcdot) \cdot \grad[\conf_i] \left( \frac{\varphi(\vect{x}\psp{k},\tcdot)}{\maxw} \right)\displaybreak[0]\\
+ \sum_{k=1}^{Q_\Omega} w_\Omega\psp{k} \int_{\D} \sumi\sumj \frac{A_{ij}}{4\wi} \maxw \grad[\conf_j] \left( \frac{\psi^{n+1}(\vect{x}\psp{k},\tcdot)}{\maxw} \right) \cdot \grad[\conf_i] \left( \frac{\varphi(\vect{x}\psp{k},\tcdot)}{\maxw} \right)\displaybreak[0]\\
\approx -\sum_{k=1}^{Q_\Omega} w_\Omega\psp{k} \int_{\D} \left.\div[\vect{x}]\left( \vel(\tcdot,t^{n+1}) \psi^{n+1/2} \right) \right|_{(\vect{x}\psp{k},\tcdot)} \varphi(\vect{x}^{(k)},\tcdot) \frac{1}{\maxw}\displaybreak[0]\\
- \sum_{k=1}^{Q_\Omega} w_\Omega\psp{k} \frac{(l_0/L_0)^2}{4\wi(\nb)} \int_{\D} \left.\grad[\vect{x}] \psi^{n+1/2}\right|_{(\vect{x}\psp{k},\tcdot)} \cdot \left.\grad[\vect{x}] \varphi\right|_{(\vect{x}\psp{k},\tcdot)} \frac{1}{\maxw},
\end{multline*}
for all $\varphi=\varphi(\vect{x},\conf)$ in a suitable function space. By
selecting $Q_{\Omega}$ linearly independent functions $\chi_{(m)}$, $m \in
\ii{Q_{\Omega}}$, of $\vect{x} \in \Omega$ such that
$\chi_{(m)}(\vect{x}^{(k)}) = \delta_{km}$, $k,m \in \ii{Q_{\Omega}}$, and
taking successively $\varphi = \varphi_{(m)}$, where
$\varphi_{(m)}(\vect{x},\conf) := \chi_{(m)} \zeta(\conf)$, in the equality
above, we obtain a total of $Q_{\Omega}$ independent variational problems over
the $N d$-dimensional domain $\D$ of the form:
\begin{multline}\label{FP-weak-SD-q-2}
\frac{1}{\Delta t/2} \int_{\D}\psi^{n+1}(\vect{x}\psp{m},\tcdot) \frac{\zeta}{\maxw}
+ \int_{\D} \sumi\sumj \frac{A_{ij}}{4\wi} \maxw \grad[\conf_j] \left( \frac{\psi^{n+1}(\vect{x}\psp{m},\tcdot)}{\maxw} \right) \cdot \grad[\conf_i] \left( \frac{\zeta}{\maxw} \right)\displaybreak[0]\\
\approx \left[\frac{1}{\Delta t/2}\int_{\D}\psi^{n+1/2}(\vect{x}\psp{m},\tcdot)\frac{\zeta}{\maxw} +
\int_{\D} \sumi (\grad[\vect{x}]\vel(\vect{x}^{(m)},t^n))\conf_i\psi^n (\vect{x}\psp{m},\tcdot) \cdot \grad[\conf_i] \left( \frac{\zeta}{\maxw} \right)\right.\displaybreak[0]\\
- \int_{\D} \left.\div[\vect{x}]\left( \vel(\tcdot,t^{n+1}) \psi^{n+1/2}\right)\right|_{(\vect{x}\psp{m},\tcdot)} \frac{\zeta}{\maxw}\displaybreak[0]\\
\left.- \frac{1}{w_\Omega^{(m)}}\sum_{k=1}^{Q_\Omega}w_\Omega^{(k)}\frac{(l_0/L_0)^2}{4\wi(\nb)} \int_{\D} \left.\grad[\vect{x}] \psi^{n+1/2}\right|_{(\vect{x}\psp{k},\tcdot)} \cdot \left.\grad[\vect{x}] \chi_{(m)}\right|_{(\vect{x}\psp{k},\tcdot)} \frac{\zeta}{\maxw}\right]\displaybreak[0]\\
=: \mathfrak{N}_{(m)}(\psi^{n+1/2}; \zeta) \qquad \forall\, m \in \ii{Q_{\Omega}},
\end{multline}
for all $\zeta = \zeta(\conf)$ in a suitable function space, where each
$\mathfrak{N}_{(m)}(\psi^{n+1/2}; \tcdot)$, $m \in \ii{Q_{\Omega}}$, is a linear
functional. Thus, \eqref{FP-weak-SD-q-2} amounts to solving $\ii{Q_{\Omega}}$
mutually independent linear elliptic variational problems, each posed on the
high-dimensional configurational domain $\D = D_1 \times \dotsm \times D_N
\subset \Real^{N d}$. It is the approximate solution of \eqref{FP-weak-SD-q-2}
by greedy algorithms that this paper is concerned with.

\section{The configuration space operator}\label{sec:operator}

\subsection{Variational formulation and function spaces} The form of the problem
\eqref{FP-weak-SD-q-2} motivates us to consider the linear elliptic variational problem
\begin{equation}\label{FP-elliptic-3}
    a(\psi, \varphi) = f(\varphi),
\end{equation}
posed on the high-dimensional configurational domain $\D = D_1 \times \dotsm \times D_N
\subset \Real^{N d}$, where
\begin{equation}\label{definition-of-a}
a(\psi,\varphi) := \int_{\D} \sumi\sumj \frac{A_{ij}}{4\wi} \maxw \grad[\conf_j]\left(\frac{\psi}{\maxw}\right) \cdot \grad[\conf_i]\left(\frac{\varphi}{\maxw}\right) + c \int_{\D}\frac{\psi\varphi}{\maxw},
\end{equation}
the parameter $c$ is positive and $f$ is a linear functional. The natural function space associated with
problem \eqref{FP-elliptic-3} is
\[ \HDM := \left\{ \varphi\in\LTOOM \cap \maxw \, \LOL \colon \grad[\conf_i](\varphi/\maxw) \in [\LTM]^d \quad\forall\, i \in \ii{N} \right\}\!, \]
equipped with the norm
\[ \norm{\varphi}_{\HDM} := \bigg(\norm{\varphi}_{\LTOOM}^2 + \sumi \norm{\grad[\conf_i](\varphi/\maxw)}_{[\LTM]^d}^2\bigg)^{1/2}. \]

The spaces $\LTOOM$ and $\HDM$ are isometrically isomorphic to, respectively,
$\LTM$ and $\HOM$ via the relations
\begin{subequations}\label{iso-iso}
\begin{align}
    \LTOOM = \maxw\,\LTM, & \quad \norm{\tcdot}_{\LTOOM} = \norm{\maxw^{-1}\tcdot}_{\LTM}\!,\\
    \HDM = \maxw\,\HOM, & \quad \norm{\tcdot}_{\HDM} = \norm{\maxw^{-1}\tcdot}_{\HOM}\!.
\end{align}
\end{subequations}

Later, we will make use of the spaces $\HDiMi$, $i \in \ii{N}$, each of which
is the $i$-th partial Maxwellian analogue of $\HDM$. That is,
\[ \HDiMi := \left\{ \varphi\in\LTOOMi \cap M_i \, \LL^1_{\loc}(D_i) \colon \grad(\varphi/M_i) \in [\LTMi]^d \right\}\!, \]
equipped with the norm $\norm{\varphi}_{\HDiMi} := \Big( \norm{\varphi}_{\LTOOMi}^2 + \norm{\grad(\varphi/M_i)}_{[\LTMi]^d}^2 \Big)^{1/2}$.
\begin{remark}\noindent
\begin{enumerate}
\item For $i \in \ii{N}$, $\HDiMi$ is exactly $\HDM$ if $N = 1$ and $\maxw =
M_i$. None of the results involving $\HDM$ appearing below depend on
restrictions on $N$ and thereby remain valid for $\HDiMi$. Just like
\eqref{iso-iso}, $\varphi \mapsto M_i \varphi$ is an isometric isomorphism
between $\LTMi$ and $\LTOOMi$ and between $\HOMi$ and $\HDiMi$.
\item The definitions above can be extended to open subsets of $\D$ and of the
$D_i$, $i \in \ii{N}$, in the usual way.
\end{enumerate}
\end{remark}

Before listing our structural hypotheses and proving the properties we need of
$\HDM$ we fully state the weak formulation of our model problem:

\medskip

Given $f \in \HDM'$, find $\psi \in \HDM$ such that
\begin{align}\label{FP-elliptic-4}
a(\psi, \varphi) = f(\varphi)\qquad \forall\, \varphi \in \HDM.
\end{align}

\medskip

We adopt the following structural hypotheses.

\begin{hypothesis}\label{hyp:potential} For each $i \in \ii{N}$, the spring
potential $U_i$ belongs to $\CC^1([0,\frac{b_i}{2}))$, where $b_i > 0$, and
satisfies $\lim_{s \to b_i/2_-}\allowbreak U(s) = +\infty$.
\end{hypothesis}

Immediate consequences of \cref{hyp:potential} are that $\maxw \in
\CC(\overline{\D}) \cap \CC^1(\D)$ and that, for any $K \compEmb \D$, there
exist positive constants $c_K$ and $C_K$ such that $c_K \leq \maxw(\conf) \leq
C_K$, for all $\conf \in K$.

\begin{hypothesis}\label{hyp:partialCompEmb} For each $i \in \ii{N}$, $\HOMi$
is compactly embedded in $\LTMi$.
\end{hypothesis}

\begin{remark}\label[remark]{rem:hypotheses}
It is easy to check that springs obeying any of the example force models
\eqref{FENE-model} and \eqref{CPAIL-model} comply with
\cref{hyp:potential}.

In Step 1 of section A.1 of \cite{BS:2008} it is proved that springs obeying
the FENE model \eqref{FENE-model} satisfy \cref{hyp:partialCompEmb}, under
the condition $b_i \geq 2$. The compliance with \cref{hyp:partialCompEmb}
of springs obeying the CPAIL model \eqref{CPAIL-model} is shown in
\cref{lem:CPAIL-results} in \cref{sec:CPAIL-results}
under the condition $b_i \geq 3$.
\end{remark}

\begin{lemma}\label[lemma]{lem:Hilbert}
$\LTM$, $\HH^m_{\maxw}(\D)$ for $m \in \natural$, $\LTOOM$ and $\HDM$ are
separable Hilbert spaces.
\end{lemma}
\begin{proof} The operation $\varphi \in \LTM \mapsto \varphi/\sqrt{\maxw}$
defines an isometric isomorphism between $\LTM$ and $\LL^2(\D)$. Therefore
the first space inherits its separability from the latter. On noting that
$\maxw^{-1} \in \LOL$, Theorem 1.11 of \cite{KO} guarantees the completeness of
$\HH^m_{\maxw}(\D)$ (this source actually states the result for the case $m = 1$ only; however, the proof carries over to higher $m$ in this single-weight case) and thus, $\HH^m_{\maxw}(\D)$ is separable by an argument
along the lines of \cite[\P 3.5]{AF:2003}. The spaces $\LTOOM$ and $\HDM$
inherit these properties via the isometric isomorphism \eqref{iso-iso}.
Finally, as their respective norms obey the parallelogram law, these spaces are
Hilbert spaces.
\end{proof}

\begin{lemma}\label[lemma]{lem:compEmb}
$\HOM$ is compactly embedded in $\LTM$, and $\HDM$ is compactly
embedded in $\LTOOM$.
\end{lemma}
\begin{proof}
Throughout this proof we will assume, for ease of exposition, that $N = 2$; the
argument carries over to higher $N$ without difficulties. Let $u \in \HOM$. As,
by \eqref{full-maxw}, $\maxw = M_1 \otimes M_2$, it follows from Fubini's
theorem that, for almost all $\conf_1 \in D_1$,
\[ u(\conf_1,\tcdot) \in \LL^2_{M_2}(D_2) \cap \LL^1_{\loc}(D_2) \quad\text{and}\quad \partial_\alpha u(\conf_1,\tcdot) \in \LL^2_{M_2}(D_2), \]
where $\alpha$ is any multi-index in $[\natural_0]^d$ with $0 \leq
\abs{\alpha} \leq 1$. Fubini's theorem, again, ensures that, given $\varphi_2
\in \CIC(D_2)$ and $\alpha_2 \in [\natural_0]^d$, $0 \leq \alpha_2 \leq 1$,
\begin{equation*}
\int_{D_1} \left[ (-1) \int_{D_2} u(\conf_1,\tcdot) \partial_{\alpha_2}\varphi_2 \right] \varphi_1 \dd \conf_1
= \int_{D_1} \left[ \int_{D_2} \partial_{(0,\alpha_2)}u(\conf_1,\tcdot) \varphi_2 \right] \varphi_1 \dd \conf_1,
\end{equation*}
for all $\varphi_1 \in \CIC(D_1)$. Therefore,
$\partial_{\alpha_2}[u(\conf_1,\tcdot)] = \partial_{(0,\alpha_2)}
u(\conf_1,\tcdot)$ in the weak sense on $D_2$ for almost all $\conf_1 \in D_1$.
As $\partial_{(0,\alpha_2)} u(\conf_1,\tcdot)$ lies in $\LL^2_{M_2}(D_2)$ for
almost all $\conf_1 \in D_1$, we have that
\begin{equation}\label{regularity-a.e.}
u(\conf_1,\tcdot) \in \HH^1_{M_2}(D_2) \quad\text{for almost all }\conf_1 \in D_1.
\end{equation}
In the same way it can be proved that
$u(\tcdot,\conf_2) \in \HH^1_{M_1}(D_1)$ for almost all $\conf_2 \in D_2$.

Let us define, for $i \in \{1, 2\}$, the sequence $\seq{D_{i,(n)}}{n \geq 1}$
of bounded and proper subsets of $D_i$ by $D_{i,(n)} := B\big( 0,
\frac{\sqrt{b_i}n}{n+1} \big)$. Then,
\[ D_{i,(n)} \subset D_{i,(n+1)},\ n \in \natural,\qquad \bigcup_{n=1}^\infty D_{i,(n)} = D_i\qquad\text{and}\qquad \HH^1_{M_i}(D_{i,(n)}) \compEmb \LL^2_{M_i}(D_{i,(n)}). \]
This last relation is a consequence of the corresponding relation for the
unweighted case, $\HH^1(D_{i,(n)})\allowbreak\compEmb \LL^2(D_{i,(n)})$---in turn a
consequence of the boundedness and Lipschitz continuity of $D_{i,(n)}$---on
account of the existence of positive lower and upper bounds for $M_i$ on
$D_{i,(n)}$, whereupon there is algebraic and topological equivalence between
$\HH^1_{M_i}(D_{i,(n)})$ and $\HH^1(D_{i,(n)})$ and between
$\LL^2_{M_i}(D_{i,(n)})$ and $\LL^2(D_{i,(n)})$.

Letting, for $n \in \natural$,
$\D_{(n)} := \bigtimes_{i=1}^2 D_{i,(n)} \subsetneq \D$, the above properties get
inherited:
\[ \D_{(n)} \subset \D_{(n+1)},\ n \in \natural,\qquad \bigcup_{n=1}^\infty \D_{(n)} = \D\qquad\text{and}\quad \HH^1_{\maxw}(\D_{(n)}) \compEmb \LL^2_{\maxw}(\D_{(n)}). \]
The third statement follows from the fact that the $\D_{(n)}$, being Cartesian
products of bounded Lipschitz domains, are also bounded Lipschitz
domains\footnote{This follows by combining Theorem 3.1~in the Ph.D. Thesis of
Reinhard Hochmuth: {\em Randwertproblem einer nicht hypoelliptischen linearen
partiellen Differentialgleichung. Dissertation, Freie Universitat Berlin,
1989}, which implies that the Cartesian product of a finite number of bounded
domains, each satisfying the uniform cone property, is a bounded domain
satisfying the uniform cone property, and Theorem 1.2.2.2~in the book of
Grisvard \cite{Grisvard}, which states that a bounded open set in $\Real^n$ has
the uniform cone property if, and only if, its boundary is Lipschitz. In the
special case of the domain $\D_{(n)}$ an alternative proof is to note that, as
a Cartesian product of bounded open convex sets, $\D_{(n)}$ is a bounded open
convex set in $\Real^n$ (cf.\ \cite{HUL}, p.~23), and then apply Corollary
1.2.2.3 in Grisvard \cite{Grisvard}, which states that a bounded open convex set in $\Real^n$ has 
Lipschitz boundary.}. Let us define $D_i\psp{n} := D_i \setminus D_{i,(n)}$ and
$\D\psp{n} := \D \setminus \D_{(n)}$. Thanks to \cite[Theorem 17.6]{OK},
the above compact embeddings on members of a nested covering imply the
following characterizations (the first, for $i \in \{1, 2\}$):
\begin{gather}
\label{equiv-i}
\HOMi \compEmb \LTMi \iff
\lim_{n \to \infty} \sup_{u \in \HOMi \setminus \{0\}}
\int_{D_i\psp{n}} u^2 M_i \,/\, \norm{u}_{\HOMi}^2 = 0,\\
\label{equiv-full}
\HOM \compEmb \LTM \iff
\lim_{n \to \infty} \sup_{u \in \HOM \setminus \{0\}}
\int_{\D\psp{n}} u^2 \maxw \,/\, \norm{u}_{\HOM}^2 = 0.
\end{gather}
From \cref{hyp:partialCompEmb}, the left-hand side of \eqref{equiv-i} holds;
hence, its right-hand side also holds. Using \eqref{regularity-a.e.} and
\eqref{equiv-i} with $i = 2$, we deduce that for each $\varepsilon > 0$ there
exists some $\tilde n = \tilde n(\varepsilon) \in \natural$ such that $n \geq
\tilde n$ implies
\begin{equation*}
\begin{split}
\int_{D_1 \times D_2\psp{n}} u^2 \maxw & = \int_{D_1} \left[ \int_{D_2\psp{n}} u^2(\conf_1,\tcdot) M_2 \right] M_1(\conf_1) \dd \conf_1\\
& \leq \varepsilon \int_{D_1} \norm{u(\conf_1,\tcdot)}_{\HH^1_{M_2}(D_2)}^2 M_1(\conf_1) \dd\conf_1\\
& = \varepsilon \int_{D_1} \! \left[ \int_{D_2} \! \left( u^2(\conf_1,\tcdot) + \abs{\grad[\conf_2] u(\conf_1,\tcdot)}^2 \right) M_2 \right] \! M_1(\conf_1) \dd \conf_1
\leq \varepsilon \norm{u}_{\HOM}^2.
\end{split}
\end{equation*}
An analogous result can be proved for the $\maxw$-weighted integral of
$u^2$ on $D_1\psp{n} \times D_2$. Then, since $\D\psp{n} = (D_1 \times D_2\psp{n})
\cup (D_1\psp{n} \times D_2)$, the right-hand side of \eqref{equiv-full} holds;
hence, so does its left-hand side.

Finally, the embedding $\HDM \compEmb \LTOOM$ follows directly from the
embedding $\HOM \compEmb \LTM$ on account of the isometric isomorphism
\eqref{iso-iso}.
\end{proof}

\begin{lemma}\label[lemma]{lem:inclusions} The following inclusion holds:
$\CC^1_0(\D) \subset \HDM$.
\end{lemma}

\begin{proof}
Let $\varphi \in \CC^1_0(\D)$ and $K := \supp(\varphi) \compEmb \D$. Then,
trivially, $\varphi \in \LTOOM$, since
\begin{equation*}
\int_{\D}\varphi^2\frac{1}{\maxw} = \int_{K} \varphi^2 \frac{1}{\maxw}
\leq |K| \sup_{\conf \in K} \frac{\varphi(\conf)^2}{\maxw(\conf)} < \infty,
\end{equation*}
which, in turn, stems from the fact that $\maxw$ is positively bounded from
below on each compact subset of $\D$. Similarly, for all $K' \compEmb \D$,
\[ \int_{K'} \abs{\frac{\varphi}{\maxw}} \leq \abs{K' \cap K} \sup_{\conf \in K'\cap K} \frac{\abs{\varphi(\conf)}}{\maxw(\conf)} < \infty \]
on account of which $\varphi \in \maxw \, \LOL$. The latter implies that
$\varphi/\maxw$ defines a regular distribution in the usual way. Then, for each
$i \in \ii{N}$, $\grad[\conf_i](\varphi/\maxw)$ exists as a distribution and
coincides with the classical $i$-th component gradient of $\varphi / \maxw$,
which belongs to $[\CC(\D)]^d$ because of \cref{hyp:potential}. Then,
\begin{equation*}
\int_{\D}\Big|\grad[\conf_i]\Big(\frac{\varphi}{\maxw}\Big)\Big|^2 \maxw
\leq |K|\sup_{\conf \in K} \abs{ \grad[\conf_i] \negmedspace \left(\frac{\varphi(\conf)}{\maxw(\conf)}\right) }^2 \maxw(\conf) < \infty
\end{equation*}
and that proves the lemma.
\end{proof}

\subsection{Properties of tensor products}
\begin{lemma}\label[lemma]{lem:distrTP} Suppose that $T \in \mathcal{D}'(\D)$ is a
distribution such that
\[ T\left(\tp_{i=1}^N \varphi\psp{i}\right) = 0
\qquad \forall\, (\varphi\psp{1}, \dotsc, \varphi\psp{N}) \in \bigtimes_{i\in\ii{N}} \CIC(D_i). \]
Then, $T=0$ in $\mathcal{D}'(\D)$.

Further, for any ensemble of sequences of distributions $\seq{R_n\psp{i}}{n
\geq 1}$, $i\in\ii{N}$, with $R_n\psp{i} \in \mathcal{D}'(D_i)$ and such that $\lim_{n
\rightarrow \infty} R\psp{i}_n = R\psp{i}$ in $\mathcal{D}'(D_i)$ for $i \in
\ii{N}$, we have that
\[ \lim_{n \rightarrow \infty} \tp_{i\in\ii{N}} R\psp{i}_n = \tp_{i\in\ii{N}} R\psp{i} \quad\text{in}\quad \mathcal{D}'(\D). \]
\end{lemma}
\begin{proof}
These are standard results from the theory of distributions, so we omit the
proofs and refer the reader to Section 1.3.2 of the book of Vladimirov
\cite{Vladimirov}, for example.
\end{proof}

\begin{lemma}\label[lemma]{lem:TP} The following statements hold:
\begin{enumerate}
\item For any ensemble $r\psp{i} \in \HDiMi$, $i\in\ii{N}$, $\tp_{i\in\ii{N}}
r\psp{i} \in \HDM$.
\item Suppose that $r\psp{i}\colon D_i \rightarrow \Real$, $i\in\ii{N}$, are
measurable functions. Then, the next two statements are equivalent:
\begin{itemize}
\item[(a)] $r\psp{i} \in \HDiMi \setminus\{0\}\quad \text{for all }
i\in\ii{N}$;
\item[(b)] $\tp_{i\in\ii{N}} r\psp{i}\in \HDM \setminus\{0\}$.
\end{itemize}
\end{enumerate}
\end{lemma}

\begin{proof} (1) It is immediate from the factorization of $\maxw$ that
$\tp_{i\in\ii{N}} r\psp{i}$ belongs to $\LTOOM$. Thanks to \cref{lem:distrTP}, the
identity
\begin{equation}\label{gradWTP}
\grad[\conf_j]\left(\frac{\tp_{i=1}^N r\psp{i}}{\maxw}\right) = \tp_{\allibutj}^N \left(\frac{r\psp{i}}{M_i}\right) \otimes_j \grad\negthickspace\left(\frac{r\psp{j}}{M_j}\right)
\end{equation}
holds in the distributional sense. Then, as $r\psp{i}/M_i \in \LTMi$ for $i \in
\ii{N} \setminus \{ j \}$, and $\grad(r\psp{j}/M_j) \in [\LL^2_{M_j}(D_j)]^d$,
the factorization of the Maxwellian $\maxw$ implies that, for $j \in
\ii{N}$,
$\grad[\conf_j](\tp_{i=1}^N r\psp{i}/\maxw) \in [\LTM]^d$.
That completes the proof of Part (1).

\medskip

(2) We shall prove the second part by showing that (b) is both necessary and
sufficient for (a).

\smallskip

(a) $\implies$ (b): This is immediate from the first part and the fact that the
tensor product of the $r\psp{i}$, $i \in \ii{N}$, cannot be null if none of its factors is.

\smallskip

(b) $\implies$ (a): Suppose that $\tp_{i=1}^N r\psp{i}\in \HDM \setminus
\{0\}$; then, because of the tensor-product structure of $\maxw$, the
positivity of $M_i$ on compact subsets of $D_i$ for $i \in \ii{N}$ and Fubini's
theorem, $r\psp{i} \in M_i\, \LL^1_{\loc}(D_i) \cap \LTOOMi$, $i \in \ii{N}$.
Hence, each $r\psp{i} / M_i$ defines a regular distribution in
$\mathcal{D}'(D_i)$. Again, \cref{lem:distrTP} makes \eqref{gradWTP} valid
and thus,
\begin{equation}\label{factor2}
\norm{\tp_{i=1}^N r\psp{i}}_{\HDM}^2 = \prod_{i=1}^N \norm{r\psp{i}}_{\LTOOMi}^2 + \sum_{j=1}^N \Bigg[ \prod_{\allibutj}^N \norm{r\psp{i}}_{\LTOOMi}^2 \Bigg] \norm{\grad\negthickspace\left(\frac{r\psp{j}}{M_j}\right)}_{[\LL^2_{M_j}(D_i)]^d}^2.
\end{equation}
Now, none of the $r\psp{i}$ can be null (otherwise their tensor product would
be null). On combining this with their $1/M_i$-weighted square integrability,
the identity \eqref{factor2} yields
$ \|\grad(r\psp{i}/M_i)\|_{[\LTMi]^d} < \infty$ for all $i \in \ii{N}$.
Hence $r\psp{i} \in \HDiMi \setminus \{ 0 \}$ for $i \in \ii{N}$.
\end{proof}

\section{Separated representation}\label{sec:SR}

\subsection{Two algorithms} The existence of a unique weak solution to
\eqref{FP-elliptic-3} is an immediate consequence of the Lax--Milgram theorem
via the facts that $\HDM$ is a Hilbert space (cf.\ \cref{lem:Hilbert}) and
$a$ is a bounded and coercive bilinear form on $\HDM$. By virtue of
the Riesz representation theorem, there exists a bounded linear operator
$\mathcal{A}\colon \HDM \rightarrow \HDM'$, defined by
$(\mathcal{A} \psi)(\varphi) = a(\psi, \varphi)$ for all  $\varphi \in \HDM$.
Thanks to the symmetry of $a$, the weak formulation \eqref{FP-elliptic-3} can
be restated as the following, equivalent, energy minimization problem:
\begin{equation}\label{FP-elliptic-3a}
\psi := \argmin_{\varphi \in \HDM} J_{f}(\varphi)\qquad \text{where}\qquad J_{f}(\varphi) := \frac{1}{2}\, a(\varphi,\varphi) - f(\varphi).
\end{equation}

We observe that, with $\psi \in \HDM$ as in \eqref{FP-elliptic-3a},
\begin{equation}\label{identity}
J_{f}(\varphi) = \frac{1}{2}\,a(\varphi - \psi, \varphi - \psi) - \frac{1}{2}\,a(\psi, \psi)\qquad
\forall\, \varphi \in \HDM.
\end{equation}
Following the work of Le~Bris, Leli\`evre and Maday \cite{LLM} concerning the
numerical solution of high-dimensional Poisson equations, we consider two
numerical methods.

\begin{alg}[Pure Greedy Algorithm]\label[alg]{PGA}\noindent
\begin{itemize}
\item[0.] Define: $f_0:=f \in \HDM'$.
\item[1.] For $n\geq 1$ do:
\begin{itemize}
\item[1.1] Find $r_n\psp{i} \in \HDiMi$, $i \in \ii{N}$, such that
\begin{equation}\label{greedy1}
(r_n\psp{1}, \dotsc, r_n\psp{N}) \in \argmin_{(s\psp{1}, \dotsc, s\psp{N})\in \bigtimes_{i=1}^N\HDiMi}
J_{f_{n-1}} \! \left( \tp_{i=1}^N s\psp{i} \! \right).
\end{equation}
\item[1.2] Define:
$f_n:= f_{n-1} - \mathcal{A} \left(\tp_{i=1}^N r_n\psp{i}\right) \in \HDM'$.
\item[1.3] If $\norm{f_n}_{\HDM'} \geq \mathtt{TOL}$, then proceed to iteration
$n+1$; else, stop.
\end{itemize}
\end{itemize}
\end{alg}

\begin{alg}[Orthogonal Greedy Algorithm]\label[alg]{OGA}\noindent
\begin{itemize}
\item[0.] Define: $f_0:=f \in \HDM'$.
\item[1.] For $n\geq 1$ do:
\begin{itemize}
\item[1.1] Find $r_n\psp{i} \in \HDiMi$, $i \in \ii{N}$, such that
\begin{equation}\label{greedy2}
(r_n\psp{1}, \dotsc, r_n\psp{N}) \in \argmin_{(s\psp{1}, \dotsc, s\psp{N})\in \bigtimes_{i=1}^N\HDiMi}
J_{f_{n-1}} \!\left( \tp_{i=1}^N r\psp{i}\! \right).
\end{equation}
\item[1.2] Minimize $J_f$ on the span of $\seq{\tp_{i=1}^N r\psp{i}_k}{k \in \ii{n} }$; i.e., find $\alpha\psp{n} \in \Real^n$ such that
\begin{equation}\label{galerkin}
\alpha\psp{n} = \argmin_{\beta\in\Real^n} \, J_f \! \left( \sum_{k=1}^n \beta_k \tp_{i=1}^N r_k\psp{i} \!\right).
\end{equation}
\item[1.3] Define:
$f_n:= f - \mathcal{A} \left(\sum_{k=1}^n\alpha\psp{n}_k \tp_{i=1}^N r_k\psp{i}\right) \in \HDM'$.
\item[1.4] If $\norm{f_n}_{\HDM'} \geq \mathtt{TOL}$, then proceed to
iteration $n+1$; else, stop.
\end{itemize}
\end{itemize}
\end{alg}

For future reference, we define $\psi_n \in \HDM$ as the unique solution of the
problem
\[ a(\psi_n, \varphi) = f_n(\varphi)\qquad \forall\, \varphi \in \HDM. \]
Clearly, for all $n$ up to the (existing or not) termination of the
corresponding algorithm,
\begin{equation}\label{errors}
\psi_n = \left\{\begin{array}{ll}
\psi_{n-1} - \tp_{i=1}^N r\psp{i}_n & \text{for the Pure Greedy Algorithm},\\
\psi - \sum_{k=1}^n \alpha\psp{n}_k \tp_{i=1}^N r\psp{i}_k & \text{for the Orthogonal Greedy Algorithm},
\end{array} \right.
\end{equation}
where $\psi = \psi_0$ is the unique solution of \eqref{FP-elliptic-3a}. Proving
the convergence of the algorithms amounts to showing that the sequences
$\seq{\psi_n}{n\geq 0}$ defined by \eqref{errors} converge to $0$ in $\HDM$.

\subsection{Correctness of the algorithms} The proof of the correctness of
\cref{PGA} (respectively \cref{OGA}) amounts to showing that, given
$f_{n-1} \in \HDM'$ (respectively $(f_{n-1},\alpha\psp{n-1}) \in \HDM' \times
\Real^{n-1}$), the loop \texttt{1} returns a well-defined member of $\HDM'$
(resp.\ $\HDM' \times \Real^n$).

We start by observing that, thanks to the first part of \cref{lem:TP}, the set
of $N$-way tensor products of ensembles of functions $\HDiMi$, $i \in \ii{N}$,
is a subset of $\HDM$, thereby rendering the minimization \emph{problems}
\eqref{greedy1} and \eqref{greedy2} sound. However, the existence of
\emph{solutions} $(r_n\psp{1}, \dotsc, r_n\psp{N})$ to these problems is quite
another matter: it will be proved using \cref{lem:belowZero} and
\cref{thm:existence} below.

\begin{lemma}\label[lemma]{lem:belowZero}
Suppose that $f \in \HDM' \setminus \{0\}$ and consider the functional $J_{f}$,
as in \eqref{FP-elliptic-3a}. Then, there exists $(r\psp{1}, \dotsc, r\psp{N})$
in $\bigtimes_{i=1}^N \HDiMi$ such that
\[ J_{f}\left(\tp_{i=1}^N r\psp{i}\right) < 0.\]
\end{lemma}

\begin{proof} Consider any functional $f \in \HDM' \setminus \{ 0 \}$ and
assume that the thesis is false; i.e., $ J_{f}\left( \tp_{i=1}^N
r\psp{i}\right) \geq 0$ for all ensembles $(r\psp{1}, \dotsc, r\psp{N}) \in
\bigtimes_{i=1}^N \HDiMi$; then,
\[ \frac{1}{2}\, a\left( \tp_{i=1}^N r\psp{i} , \tp_{i=1}^N r\psp{i}\right) \geq f\left(\tp_{i=1}^N r\psp{i} \right) \qquad \forall\ (r\psp{1},\dotsc,r\psp{N}) \in \bigtimes_{i=1}^N \HDiMi. \]
Given a particular ensemble $(r\psp{1},\dotsc,r\psp{N}) \in \bigtimes_{i=1}^N
\HDiMi$, we can replace $r\psp{1}$ with $\varepsilon r\psp{1}$ and, by virtue
of the bilinearity of $a$ and the linearity of $f$, we obtain
\begin{equation}\label{TP-ineq}
\frac{1}{2}\,\varepsilon^2 a\left( \tp_{i=1}^N r\psp{i} , \tp_{i=1}^N r\psp{i}\right) \geq \varepsilon\, f\left(\tp_{i=1}^N r\psp{i} \right).
\end{equation}
By combining the inequalities resulting from dividing both sides of
\eqref{TP-ineq} by positive $\varepsilon$ and taking the one-sided limit
$\varepsilon \to 0_+$ and from dividing \eqref{TP-ineq} by a negative
$\varepsilon$ and taking the one-sided limit $\varepsilon \to 0_-$ we get that
\[ f\left(\tp_{i=1}^N r\psp{i} \right) = 0. \]
As this is valid for any ensemble $(r\psp{1},\dotsc,r\psp{N}) \in
\bigtimes_{i=1}^N \HDiMi$, \cref{lem:inclusions} implies that it is valid,
in particular, for any ensemble $(r\psp{1},\dotsc,r\psp{N}) \in
\bigtimes_{i=1}^N \CIC(D_i)$, whence \cref{lem:distrTP} implies that $f =
0$. As this contradicts the hypotheses of the lemma, its thesis holds.
\end{proof}

We are now in a position to prove the existence of solutions to problems
\eqref{greedy1} and \eqref{greedy2}.

\begin{theorem}\label{thm:existence}
Given $f_{n-1} \in \HDM'$, each of the problems \eqref{greedy1} and
\eqref{greedy2} has a solution.
\end{theorem}
\begin{proof} Since problems \eqref{greedy1} and \eqref{greedy2} are completely
analogous, it suffices to consider one of them---say, \eqref{greedy1}. Then, as
$(0,\dotsc,0)$ is a solution of \eqref{greedy1} and \eqref{greedy2} when
$f_{n-1} = 0$, we assume from now on that $f_{n-1} \neq 0$.

By \eqref{identity} and the coerciveness of $a$, $J_{f_{n-1}}(\varphi)
\geq - \frac{1}{2} a(\psi, \psi)$ for all $\varphi \in \HDM$, where $\psi$ is
the unique solution of \eqref{FP-elliptic-4} in $\HDM$ when $f = f_{n-1}$. As,
by \cref{lem:TP}, the $N$-way tensor product of functions in $\HDiMi$, $i
\in \ii{N}$, is a subset of $\HDM$, $J_{f_{n-1}}$ is bounded from below over
that manifold. That is,
\begin{equation}\label{inf}
\mathfrak{m} := \inf_{(s\psp{1},\dotsc,s\psp{N}) \in \bigtimes_{i=1}^N \HDiMi} J_{f_{n-1}}\left(\tp_{i=1}^N s\psp{i}\right) > -\infty.
\end{equation}
It follows from \cref{lem:belowZero} that $\mathfrak{m}<0$. Our aim is to
show that the infimum $\mathfrak{m}$ is attained at an element of the form
$\tp_{i=1}^N r\psp{i}$ with $(r\psp{1}, \dotsc, r\psp{N}) \in \bigtimes_{i=1}^N
(\HDiMi \setminus \{ 0 \})$.

From \eqref{inf}, there exists a sequence $\seq{\tp_{i\in\ii{N}} r_k\psp{i}}{k
\geq 1}$ of $N$-way tensor products of functions in $\HDiMi$, $i\in \ii{N}$,
such that
\[ \lim_{k \rightarrow \infty} J_{f_{n-1}}\left(\tp_{i=1}^N r_k\psp{i}\right) = \mathfrak{m}.\]
On noting that, from the definition of $a$ in \eqref{definition-of-a}, for all $\varphi \in \HDM$,
\begin{equation*}
\begin{split}
J_{f_{n-1}}(\varphi) & = \frac{1}{2}\,a(\varphi - \psi, \varphi - \psi) - \frac{1}{2}\,a(\psi, \psi)
\geq \frac{1}{4}\, a(\varphi, \varphi) - a(\psi, \psi)\\
& \geq \frac{1}{4} \min\left(\frac{\lambda_{\mathrm{min}}}{4\wi}, c\right)\, \norm{\varphi}_{\HDM}^2 - a(\psi, \psi),
\end{split}
\end{equation*}
it follows, by setting $\varphi = \tp_{i\in\ii{N}}
r_k\psp{i}$, that the sequence $\seq{\tp_{i\in\ii{N}} r_k\psp{i}}{k \geq 1}$ is
bounded in $\HDM$; in other words, there exists $C>0$ such that (cf.\
\eqref{factor2}):
\begin{equation}\label{factor3}
\norm{\tp_{i=1}^N r_k\psp{i}}_{\HDM}^2 = \prod_{i=1}^N \norm{r_k\psp{i}}_{\LTOOMi}^2
+ \sum_{j=1}^N \Bigg(\prod_{\allibutj}^N \norm{r_k\psp{i}}_{\LTOOMi}^2 \Bigg) \norm{\grad(r_k\psp{j} / M_j)}_{[\LL^2_{M_j}(D_j)]^d}^2 \leq C
\end{equation}
for all $k \geq 1$. Since the value of $\tp_{i\in\ii{N}} r_k\psp{i}$ is
unaltered by multiplying the first $N-1$ factors by positive constants
$c_{1,k}, \dotsc, c_{N-1,k}$, respectively, and dividing the final factor by
the product $c_{1,k}\dotsm c_{N-1,k}$, we can assume without loss of generality
that
\begin{equation}\label{factor3a}
\norm[n]{r_k\psp{i}}_{\LTOOMi}^2 = 1, \qquad i \in \ii{N-1}.
\end{equation}
Thus, it follows from \eqref{factor3} that
\begin{equation}\label{factor4}
\norm{r_k\psp{N}}_{\LL^2_{1/M_{\!N}}(D_N)}^2 + \norm{r_k\psp{N}}_{\LL^2_{1/M_{\!N}}(D_N)}^2 \sum_{j=1}^{N-1} \norm{\grad(r_k\psp{j}/M_j)}_{[\LL^2_{M_j}(D_j)]^d}^2 + \norm{\grad(r_k\psp{N}/M_N)}_{[\LL^2_{M_{\!N}}(D_N)]^d}^2
\leq C.
\end{equation}

Since the sequence $\seq{\tp_{i\in\ii{N}} r_k\psp{i}}{k \geq 1}$ is bounded in
$\HDM$, and $\HDM$ is a Hilbert space, and therefore reflexive, the sequence
has a weakly convergent subsequence in $\HDM$, denoted by
$\seq{\tp_{i\in\ii{N}} r_{\phi(k)}\psp{i}}{k \geq 1}$; we denote its weak limit by
$r \in \HDM$. Since $J_{f_{n-1}}$ is convex on $\HDM$ and continuous (and
thereby also semicontinuous) in the strong topology of $\HDM$, it is weakly
lower-semicontinuous on $\HDM$. Hence
\[ J_{f_{n-1}}(r) \leq \liminf_{k \rightarrow \infty} J_{f_{n-1}}\left(\tp_{i=1}^N r_{\phi(k)}\psp{i}\right) = \lim_{k \rightarrow \infty} J_{f_{n-1}}\left(\tp_{i=1}^N r_k\psp{i}\right) = \mathfrak{m} < 0.\]
Thus we deduce that $r \neq 0$ (as $r = 0$ would imply that $J_{f_{n-1}}(r) =
0$); hence, $r \in \HDM \setminus \{ 0 \}$.

According to \eqref{factor3a} and \eqref{factor4} each subsequence
$\seq{r_{\phi(k)}\psp{i}}{k \geq 1}$, is bounded in the respective space $\LTOOMi$, for $i
\in \ii{N}$. Then, $\seq{r_{\phi(k)}\psp{i}}{k \geq 1}$ has a weakly convergent
subsequence in $\LTOOMi$, say $\seq{r_{\phi'(k)}\psp{i}}{k \geq 1}$, for $i \in \ii{N}$; let
us denote by $r\psp{i} \in \LTOOMi$ the corresponding weak limits. As by \cref{lem:inclusions}, $\CIC(D_i) \subset \HDiMi \subset \LTOOMi$, for all $\varphi \in
\CIC(D_i)$ the mapping $\xi \in \LTOOMi \mapsto \langle \varphi, \xi\rangle_{\LTOOMi}$ defines a bounded linear functional on $\LTOOMi$. Thus, $\seq{ r_{\phi'(k)}\psp{i}/M_i }{k \geq 1}$ converges to $r\psp{i}/M_i$
in $\mathcal{D}'(D_i)$ for $i \in \ii{N}$. Hence, by
\cref{lem:distrTP},
\begin{equation}\label{SE1}
\lim_{k\to\infty} \tp_{i=1}^N \frac{r_{\phi'(k)}\psp{i}}{M_i} = \tp_{i=1}^N \frac{r\psp{i}}{M_i} = \frac{\tp_{i=1}^N r\psp{i}}{\maxw} \quad \text{in } \mathcal{D}'(\D).
\end{equation}
Similarly, the inclusion $\CIC(\D) \subset \HDM$ (cf. \cref{lem:inclusions}) and the fact that, for all $\varphi \in \CIC(\D)$, the mapping $\xi \in \HDM \mapsto \langle \varphi, \xi \rangle_{\LTOOM}$ defines a bounded linear functional on $\HDM$ imply
\begin{equation}\label{SE3}
\lim_{k\to\infty} \tp_{i=1}^N \frac{ r\psp{i}_{\phi'(k)} }{M_i} = \lim_{k \to \infty} \frac{\tp_{i=1}^N r\psp{i}_{\phi'(k)}}{\maxw} = \frac{r}{\maxw} \quad \text{in } \mathcal{D}'(\D)
\end{equation}
on account of $r$ being the $\HDM$-weak limit of the sequence $\seq{ \tp_{i\in\ii{N}} r\psp{i}_{\phi(k)} }{ k \geq 1}$. As $\mathcal{D}'(\D)$ is a Hausdorff topological space, the limits in \eqref{SE1} and \eqref{SE3} have to coincide. That is,
\[ \maxw^{-1} r = \maxw^{-1} \tp_{i=1}^N r\psp{i} \qquad \text{in } \mathcal{D}'(\D) .\]
Hence, $r = \tp_{i=1}^N r\psp{i}$ almost everywhere. As $r \in \HDM
\setminus \{ 0 \}$ and has a tensor-product structure, the second part of
\cref{lem:TP} implies that $r\psp{i} \in \HDiMi \setminus \{ 0 \}$ for $i
\in \ii{N}$.
Now,
\[ J_{f_{n-1}}\left( \tp_{i=1}^N r\psp{i} \right) = J_{f_{n-1}}(r) \leq \mathfrak{m}.\]
Recalling the definition of $\mathfrak{m}$ from \eqref{inf}, we have thus shown
that the infimum in \eqref{inf} is attained at $\tp_{i=1}^N r\psp{i}$. Thus,
$(r\psp{1}, \dotsc, r\psp{N}) \in \bigtimes_{i=1}^N (\HDiMi \setminus \{0\})$
is a solution to problem \eqref{greedy1}.

\end{proof}

Having proved that the minimization problems \eqref{greedy1} of \cref{PGA}
and \eqref{greedy2} of \cref{OGA} have solutions, establishing the
correctness of what is left of the algorithms is straightforward. The Galerkin
problem \texttt{1.2} of \cref{OGA} is well-defined and has a unique solution
for each $n\geq 1$, because it is equivalent to the minimization of a coercive quadratic
form over a finite-dimensional linear space. Then, at last, the definition of
the $n$-th residual in step \texttt{1.2} of \cref{PGA} and in step
\texttt{1.3} of \cref{OGA} are correct on noting that $\mathcal{A}$ maps
$\HDM$ into $\HDM'$.

\medskip

In the next section we establish the convergence of the two algorithms.

\section{Convergence of the Algorithms}\label{sec:convergence}

\subsection{Euler--Lagrange equations}

\begin{lemma}\label[lemma]{lem:EL} Local minimizers $(r_n\psp{1}, \dotsc, r_n\psp{N})$
of the minimization problems \eqref{greedy1} or \eqref{greedy2} satisfy the
following Euler--Lagrange equation system: For all $(s\psp{1}, \dotsc,
s\psp{N}) \in \bigtimes_{i\in\ii{N}} \HDiMi$,
\begin{equation}\label{euler-lagrange}
a \Biggl( \tp_{i=1}^N r_n\psp{i}, \sum_{j=1}^N \tp_{\allibutj}^N r_n\psp{i} \otimes_j s\psp{j} \Biggr) = f_{n-1} \Biggl( \sum_{j=1}^N\tp_{\allibutj}^N r_n\psp{i} \otimes_j s\psp{j} \Biggr).
\end{equation}
From this, it follows that, for the Pure Greedy Algorithm (\cref{PGA}):
\begin{equation}\label{orthogonality}
a\left(\psi_n,\tp_{i=1}^N r_n\psp{i} \right) = 0.
\end{equation}
\end{lemma}

\begin{proof}
Let $(r_n\psp{1}, \dotsc, r_n\psp{N})$ be a solution to the minimization
problem \eqref{greedy1} or \eqref{greedy2}. Then, given any ensemble $(s\psp{1}, \dotsc,
s\psp{N})$, \eqref{euler-lagrange} is but a way of writing that the derivative
of
\[ J_{f_{n-1}}\left(\tp_{i=1}^N \left(r_n\psp{i} + \varepsilon s\psp{i}\right) \right) \]
with respect to $\varepsilon$ is zero when evaluated at $\varepsilon = 0$.
As, by hypothesis, $(r_n\psp{1}, \dotsc, r_n\psp{N})$ is a local minimizer of $J_{f_{n-1}}$ and $\varepsilon
\mapsto \mathfrak{J}_n(\varepsilon):=J_{f_{n-1}}\left(\tp_{i=1}^N \left(r_n\psp{i} + \varepsilon
s\psp{i}\right) \right)$ is regular enough, the fact that $\mathfrak{J}'_n(0)=0$ implies that
\eqref{euler-lagrange} holds.

Setting $(s\psp{1}, \dotsc, s\psp{N}) = (r_n\psp{1}, \dotsc, r_n\psp{N})$ in
\eqref{euler-lagrange}, we obtain from the definition of the $\psi_n$ that
\begin{equation}\label{orthogonality1}
a\left( \tp_{i=1}^N r_n\psp{i}, \tp_{i=1}^N r_n\psp{i} \right) = a\left( \psi_{n-1}, \tp_{i=1}^N r_n\psp{i} \right).
\end{equation}
Combining this with \eqref{errors} we obtain \eqref{orthogonality}.
\end{proof}

\begin{samepage}
\begin{remark}\noindent
\begin{enumerate}
\item The above lemma only states that local minima of the minimization problem
\eqref{greedy1} and \eqref{greedy2} satisfy the Euler--Lagrange equation
\eqref{euler-lagrange}. The converse may be false, of course: although the functional
that is minimized is quadratic, the set over which it is minimized is nonlinear, so
there is no reason why a stationary point should be a local minimum.
\item In what follows we make liberal use of the norm $\norm{\tcdot}_a :=
a(\tcdot,\tcdot)^{1/2}$ on $\HDM$, which, thanks to its
equivalence with $\norm{\tcdot}_{\HDM}$, makes no difference when making
topological statements (such as convergence).
\end{enumerate}
\end{remark}
\end{samepage}

\begin{lemma}\label[lemma]{lem:mostEnergy} Let $(r_n\psp{1}, \dotsc, r_n\psp{N})$ be a
global minimizer for the minimization problem \eqref{greedy1} of the
\cref{PGA}. Then,
\begin{equation}\label{most-energy}
\norm{\tp_{i=1}^N r_n\psp{i}}_a
= \frac{ a\left( \psi_{n-1}, \tp_{i=1}^N r_n\psp{i} \right) }{ \norm{ \tp_{i=1}^N r_n\psp{i} }_a}
= \sup_{s \in \tp_{i\in\ii{N}}\HDiMi \setminus \{0\}} \frac{a\left( \psi_{n-1}, s \right)}{\norm{s}_a}.
\end{equation}
\end{lemma}

\begin{proof}
The first equality in \eqref{most-energy} comes directly from
\eqref{orthogonality1}. Now, analogously to \eqref{identity}, $J_{f_{n-1}}$ can
be written as
\[ J_{f_{n-1}}(\varphi) = \frac{1}{2}a(\varphi - \psi_{n-1}, \varphi - \psi_{n-1}) - \frac{1}{2}a(\psi_{n-1}, \psi_{n-1}) \qquad \forall\,\varphi \in \HDM. \]
Combining this representation of $J_{f_{n-1}}$ with the fact that $r_n :=
\tp_{i\in\ii{N}} r_n\psp{i}$ minimizes $J_{f_{n-1}}$ among the members of
$\tp_{i\in\ii{N}} \HDiMi$ and the first equality of \eqref{most-energy},
according to which $a\left( \psi_{n-1}, r_n\right) = \norm{r_n}_a^2$, we have,
for all $s \in \tp_{i\in\ii{N}} \HDiMi \setminus \{ 0 \}$, that
\begin{equation*}
\norm{\psi_{n-1} - \frac{a\left( \psi_{n-1}, r_n \right)}{\norm{r_n}_a^2} r_n }_a^2
= \norm{\psi_{n-1} - r_n}_a^2
\leq \norm{\psi_{n-1} - \frac{a(\psi_{n-1},s)}{\norm{s}_a^2} s}_a^2.
\end{equation*}
Therefore,
\[ \frac{a(\psi_{n-1},r_n)^2}{a(r_n,r_n)} \geq \frac{a(\psi_{n-1},s)^2}{a(s,s)}. \]
Taking the supremum over $s \in \tp_{i\in\ii{N}}\HDiMi \setminus \{0\}$ and noting that $r_n$ is an admissible $s$ we
get the second equality in \eqref{most-energy}.
\end{proof}

\subsection{Convergence}
\begin{theorem}\label{thm:PGA-converges}
The Pure Greedy Algorithm (\cref{PGA}) converges to the solution $\psi$ to
\eqref{FP-elliptic-4}.
\end{theorem}

\begin{proof}
Let $\seq{ (r_n\psp{1}, \dotsc, r_n\psp{N}) }{n \geq 1}$ be a sequence in
$\bigtimes_{i=1}^N \HDiMi$ returned by the Pure Greedy Algorithm and let us adopt the shorthand notation $r_n := \tp_{i\in\ii{N}} r_n\psp{i}$. Then, from
\eqref{errors} and \eqref{orthogonality} in \cref{lem:EL} we obtain
\[ \norm{\psi_{n-1}}_a^2 = \norm{\psi_n + r_n}_a^2 = \norm{\psi_n}_a^2 + \norm{r_n}_a^2 . \]
Hence the sequence $\seq{\norm{\psi_n}_a}{n \geq 0}$ is nonnegative and
monotonic nonincreasing, and therefore converges in $\Real$; by summing the
above expression over $n$ we then deduce that
\begin{equation}\label{finiteEnergy}
\sum_{n=1}^\infty a( r_n, r_n ) < \infty.
\end{equation}
Let us define the function $\phi \colon \natural \rightarrow \natural$
recursively by $\phi(1) := 1$ and
\[ \phi(k) := \argmin_{n > \phi(k-1)} \left\{ \norm{r_n}_a \leq \norm{ r_{\phi(k-1)} }_a \right\}, \quad k \geq 2. \]
From \eqref{finiteEnergy} the function $\phi$ is well-defined and strictly
monotonic increasing. Hence, it is suitable for defining subsequences. As each
$(r_n\psp{1},\dotsc,r_n\psp{N})$ is a global solution to the problem
\eqref{greedy1} with the instance $f_{n-1}$, via \eqref{errors} and
\cref{lem:mostEnergy} we have, for $n \geq m \geq 1$,
\begin{equation*}
\begin{split}
\norm{\psi_{\phi(n)-1} - \psi_{\phi(m)-1}}_a^2
& = \norm{ \psi_{\phi(n)-1} }_a^2 + \norm{ \psi_{\phi(m)-1} }_a^2 - 2 a \left(\psi_{\phi(n)-1}, \psi_{\phi(n)-1} + \sum_{k=\phi(m)}^{\phi(n)-1} r_{k} \right)\\
& = \norm{\psi_{\phi(m)-1}}_a^2 - \norm{\psi_{\phi(n)-1}}_a^2 - 2\sum_{k=\phi(m)}^{\phi(n)-1} a\left( \psi_{\phi(n)-1}, r_k \right)\\
& \leq \norm{\psi_{\phi(m)-1}}_a^2 - \norm{\psi_{\phi(n)-1}}_a^2 + 2\sum_{k=\phi(m)}^{\phi(n)-1} \norm{r_k}_a \norm{r_{\phi(n)}}_a\\
& \leq \norm{\psi_{\phi(m)-1}}_a^2 - \norm{\psi_{\phi(n)-1}}_a^2 + 2\sum_{k=\phi(m)}^{\phi(n)-1} \norm{r_k}_a^2.
\end{split}
\end{equation*}
From the convergence of $\seq{\norm{\psi_{\phi(n)-1}}_a}{n \geq 1}$ in $\Real$
and \eqref{finiteEnergy}, we deduce that the sequence
$\seq{\psi_{\phi(n)-1}}{n \geq 1}$ is a Cauchy sequence in $\HDM$ and thus
converges to some $\psi_\infty \in \HDM$. Another consequence of the global
optimality of each $(r_n\psp{1}, \dotsc, r_n\psp{N})$ is: For all $(s\psp{1},
\dotsc, s\psp{N}) \in \bigtimes_{i\in\ii{N}} \HDiMi$ and $n \geq 1$,
\begin{equation*}
\begin{split}
\frac{1}{2}a\left( \tp_{i=1}^N s\psp{i}, \tp_{i=1}^N s\psp{i} \right) - a\left(\psi_{\phi(n)-1}, \tp_{i=1}^N s\psp{i} \right) & \geq J_{f_{\phi(n)-1}}\left( r_{\phi(n)} \right)\\
& = \frac{1}{2}a\left(r_{\phi(n)}, r_{\phi(n)}\right) - f_{\phi(n)-1}\left( r_{\phi(n)} \right)\\
& = -\frac{1}{2} a\left(r_{\phi(n)}, r_{\phi(n)}\right).
\end{split}
\end{equation*}
Taking the limit as $n$ tends to infinity at both ends, and noting that by
\eqref{finiteEnergy} the right-hand side of the last inequality converges to
$0$, we obtain
\[ \frac{1}{2} a \left( \tp_{i=1}^N s\psp{i}, \tp_{i=1}^N s\psp{i} \right) - a\left( \psi_{\infty}, \tp_{i=1}^N s\psp{i} \right) \geq 0. \]
Thus, \cref{lem:belowZero} implies that $\psi_{\infty} = 0$. Hence the
sequence $\seq{\norm{\psi_{\phi(n)-1}}}{n \geq 1}$ converges to zero as $n
\rightarrow \infty$. As the sequence $\seq{\norm{\psi_n}_a}{n \geq 0}$ is
monotonic nonincreasing and $\seq{\phi(n)-1}{n \geq 1}$ is a monotonic
increasing infinite sequence in $\natural$, if follows that the full sequence
$\seq{\norm{\psi_n}}{n \geq 1}$ converges to the common limit in $\Real$: $0 =
\norm{\psi_{\infty}}_a$, giving  $\lim_{n\to\infty} \psi_{n} = 0$
in $\HDM$. 
\end{proof}

The following corollary is a direct consequence of \cref{thm:PGA-converges}
and will prove useful later on.

\begin{corollary}\label[corollary]{cor:tensorDensity}\noindent
Let $F_i$ be a dense subset of $\HDiMi$ for $i \in \ii{N}$. Then, the span of
$\tp_{i\in\ii{N}} F_i$ is dense in $\HDM$.
\end{corollary}
\begin{proof}

Let $\tau \in \HDM$. Applying \cref{thm:PGA-converges} to the case in which
the right-hand side functional $f \in \HDM'$ of problem \eqref{FP-elliptic-4}
is $\varphi \mapsto a(\tau,\varphi)$ (i.e., the $\HDM$ \emph{approximation}
problem) it follows that $\tau$ can be approximated arbitrarily closely by
finite sums of the form $\sum_{m\in\ii{M}} \tp_{i\in\ii{N}} r_m\psp{i}$, where $M \in \natural$ and
$r_m\psp{i} \in \HDiMi$ for $m \in \ii{M}$ and $i \in \ii{N}$. Thus, if we can
show that $\tp_{i\in\ii{N}} F_i$ is dense in the manifold
$\tp_{i\in\ii{N}}\HDiMi$, our desired result will stand.

Let, then, $r\psp{i} \in \HDiMi$, for $i \in \ii{N}$. From the density of $F_i$
in $\HDiMi$ for each $i \in \ii{N}$, there exists a sequence
$\seq{r_n\psp{i}}{n \geq 1}$ in $F_i$, which converges to $r\psp{i}$ in
$\HDiMi$. Now,
\[ \delta_n := \tp_{i=1}^N r\psp{i} - \tp_{i=1}^N r_n\psp{i} = \sum_{k=1}^N \tp_{i=1}^N t_{n,k}\psp{i}, \qquad\text{where}\qquad t_{n,k}\psp{i} := \left\{ \begin{array}{ll}
r_n\psp{i} & \text{if}\ i > k,\\
r\psp{i} - r_n\psp{i} & \text{if}\ i = k,\\
r\psp{i} & \text{if}\ i < k.
\end{array}\right.
\]
Then, (cf.\ \eqref{factor2}),
\begin{equation*}
\norm{ \delta_n }_{\HDM}^2
\leq \sum_{k=1}^N \left[ \prod_{i=1}^N \norm{t_{n,k}\psp{i}}_{\LTOOMi}^2
+ \sum_{j=1}^N \prod_{\allibutj}^N \norm{ t_{n,k}\psp{i} }_{\LTOOMi}^2 \norm{ \grad\negthickspace\Bigg(\frac{t_{n,k}\sspsp{j}}{M_j}\Bigg)}_{[\LL^2_{M_j}(D_j)]^d}^2 \right].
\end{equation*}

As each product term on the right-hand side above consists of $N-1$ bounded
factors and one vanishing factor as $n\to\infty$, the full expression tends to
zero as $n$ tends to infinity and, therefore, so does the left-hand side.
The desired result follows.
\end{proof}

\begin{remark}\label[remark]{rem:densities} Suppose that, for each $i \in \ii{N}$,
\begin{equation}\label{conditionPartialDensity}
\CIC(D_i) \text{ is dense in } \HDiMi.
\end{equation}
Then, as $\mathrm{span}\left( \tp_{i=1}^N \CIC(D_i) \right) \subset \CIC(\D) \subset \HDM$,
we have, thanks to \cref{cor:tensorDensity}, that
\begin{equation}\label{fullDensity}
\CIC(\D) \text{ is dense in } \HDM.
\end{equation}
Springs obeying the FENE model \eqref{FENE-model} comply with
\eqref{conditionPartialDensity} under the condition $b_i \geq 2$ as is proved
in Remark 3.7 of \cite{Masmoudi}. Springs obeying the CPAIL model
\eqref{CPAIL-model}, in turn, comply with \eqref{conditionPartialDensity} as it
is shown in \cref{lem:CPAIL-results} in \cref{sec:CPAIL-results}, under
the condition $b_i \geq 3$. So, in these two cases, \eqref{fullDensity} holds.

Interesting as \eqref{fullDensity} is, we make no use of it in this work and
that is why we shall not adopt \eqref{conditionPartialDensity} as a hypothesis
on a par with hypotheses \ref{hyp:potential} and \ref{hyp:partialCompEmb} above
or hypotheses \ref{hyp:asymptotics}, \ref{hyp:potential2} and
\ref{hyp:power-like} below. However, we do use \eqref{conditionPartialDensity}
as an ingredient of the proof of the compliance of FENE and CPAIL
spring potentials with \cref{hyp:asymptotics} of \cref{sec:characterization} (cf.\
\cref{cor:easyAsymptotics} in \cref{sec:ev-asymptotics}).
\end{remark}

\begin{theorem}\label{thm:OGA-converges}
The Orthogonal Greedy Algorithm (\cref{OGA}) converges to the solution
$\psi$ to problem \eqref{FP-elliptic-4}.
\end{theorem}
\begin{proof}

We first note that thanks to \eqref{errors}, the optimality of $\alpha\psp{n}$ in
\eqref{galerkin} and the optimality of $(r_n\psp{1}, \dotsc, r_n\psp{N})$ in
\eqref{greedy2} (via \cref{lem:EL}),
\begin{equation*}
\norm{\psi_n}_a^2 = \norm{\psi - \sum_{k=1}^n \alpha\psp{n}_k \tp_{i=1}^N r_k\psp{i}}_a^2 \leq \norm{ \psi_{n-1} - \tp_{i=1}^N r_n\psp{i} }_a^2 = \norm{\psi_{n-1}}_a^2 - \norm{\tp_{i=1}^N r_n\psp{i}}_a^2.
\end{equation*}
Thus, just like in the proof of \cref{thm:PGA-converges}, we have that the
real sequence $\seq{\norm{\psi_n}_a}{n \geq 0}$ is decreasing and thus convergent
and that $\sum_{n \geq 1} a\left(\tp_{i\in\ii{N}} r_n\psp{i}, \tp_{i\in\ii{N}}
r_n\psp{i}\right) < \infty$. As $\seq{\psi_n}{n \geq 0}$ is a bounded sequence
in the Hilbert space $\HDM$, a weakly convergent subsequence
$\!\seq{\psi_{\phi(n)}}{\!n \geq 1}$ can be extracted and we denote the weak limit by
$\psi_\infty$. From the optimality of $(r_{\phi(n) + 1}\psp{1}, \dotsc
r_{\phi(n)+1}\psp{N})$ with respect to problem \eqref{greedy2} we have by
\cref{lem:EL} that, for all $(s\psp{1}, \dotsc, s\psp{N}) \in
\!\bigtimes_{i\in\ii{N}}\HDiMi$,
\[ \frac{1}{2} a\left( \tp_{i=1}^N s\psp{i}, \tp_{i=1}^N s\psp{i} \right) - a\left(\psi_{\phi(n)}, \tp_{i=1}^N s\psp{i} \right) \geq -\frac{1}{2} a\left( \tp_{i=1}^N r_{\phi(n)+1}\psp{i}, \tp_{i=1}^N r_{\phi(n)+1}\psp{i} \right). \]
Taking the limit $n \to \infty$ at both sides yields
\[ \frac{1}{2} a\left( \tp_{i=1}^N s\psp{i}, \tp_{i=1}^N s\psp{i} \right) - a\left(\psi_\infty, \tp_{i=1}^N s\psp{i} \right) \geq 0, \]
whence, via \cref{lem:belowZero}, $\psi_\infty = 0$.
By
Galerkin orthogonality for \eqref{galerkin}, $a(\psi -
\psi_{\phi(n)},\psi_{\phi(n)}) = 0$. That is, $\norm{\psi_{\phi(n)}}_a^2 = a(\psi,
\psi_{\phi(n)})$. Hence, $\lim_{n\to\infty} \norm{\psi_{\phi(n)}}_a^2 = \lim_{n\to\infty} a(\psi,
\psi_{\phi(n)}) = a(\psi,\psi_\infty)= 0$. As the full sequence of norms $\seq{\norm{\psi_n}_a}{n\geq
0}$ is monotonic decreasing, the full sequence $\seq{\psi_n}{n \geq 0}$
converges strongly to $0$ in $\HDM$.
\end{proof}

\subsection{Rate of convergence} The theory of nonlinear approximation provides
us with some estimates on the rate of convergence of \cref{PGA} and
\cref{OGA}. Following \cite{DvT} we introduce the space
\begin{equation}\label{definition-of-A1}
\mathcal{A}_1 := \bigcup_{M > 0} \overline{ \mathcal{A}_1^o(M) },
\end{equation}
where
\begin{multline}\label{A1oM}
\mathcal{A}_1^o(M) := \bigg\{ \varphi \in \HDM \colon \varphi = \sum_{k \in \Lambda} c_k w_k,\ w_k \in \tp_{i=1}^N \HDiMi,\ \norm{w_k}_{a} = 1,\\
\abs{\Lambda} < \infty \quad \text{and} \quad \sum_{k \in \Lambda} \abs{c_k} \leq M \bigg\},
\end{multline}
together with the norm
\begin{equation}\label{A1-norm}
\norm{\varphi}_{\mathcal{A}_1} := \inf \left\{ M > 0 \colon \varphi \in \overline{ \mathcal{A}_1^o(M) } \right\}.
\end{equation}
The importance of this space becomes apparent in the light of the following two
theorems.

\begin{theorem}[Theorem 3.6 of \cite{DvT}]\label{thm:PGA-rate}
If the solution $\psi$ of \eqref{FP-elliptic-4} is a member of $\mathcal{A}_1$, then
the $n$-th error $\psi_n$ of the Pure Greedy Algorithm (\cref{PGA})
satisfies
\[ \norm{\psi_n}_a \leq \norm{\psi}_{\mathcal{A}_1} n^{-1/6}. \]
\end{theorem}

\begin{theorem}[Theorem 3.7 of \cite{DvT}]\label{thm:OGA-rate}
If the solution $\psi$ of \eqref{FP-elliptic-4} is a member of $\mathcal{A}_1$, then
the $n$-th error $\psi_n$ of the Orthogonal Greedy Algorithm (\cref{OGA})
satisfies
\[ \norm{\psi_n}_a \leq \norm{\psi}_{\mathcal{A}_1} n^{-1/2}. \]
\end{theorem}

\begin{remark}\noindent
\begin{enumerate}
\item
Pure Greedy Algorithm-based approximations such as \cref{PGA} have been
proved to obey the slightly improved rate (see \cite[Remark 2.3.11]{Temlyakov}
and references therein)
\[ \norm{\psi_n}_a \leq 4 \norm{\psi}_{\mathcal{A}_1} n^{-11/62}. \]
\item In \cite[Theorem 4.1]{CEL:2011} it is shown that the convergence of the
Orthogonal Greedy Algorithm is exponentially fast if the factor spaces
and the full ansatz space (in our setting the $\HDiMi$ and $\HDM$,
respectively) are finite-dimensional.
\end{enumerate}
\end{remark}

We note that $\mathcal{A}_1$ will remain the same space if in its
definition---in \eqref{A1oM}, in particular---we replace the energy norm
$\norm{\tcdot}_a$ with the standard norm of $\HDM$, as these two norms are
equivalent. Then, $\varphi \in \HDM$ will be a member of $\mathcal{A}_1$ if,
and only if, there exists an $M^*>0$ such that, for all $\varepsilon > 0$,
there is a $\chi_\varepsilon \in \HDM$ that satisfies
\begin{gather*}
\norm{\varphi - \chi_\varepsilon}_{\HDM} \leq \varepsilon,
\quad \chi_\varepsilon = \sum_{k \in \Lambda\psp{\varepsilon}} c\psp{\varepsilon}_k w\psp{\varepsilon}_k,
\quad \abs[n]{\Lambda\psp{\varepsilon}} < \infty,
\quad \sum_{k \in \Lambda\psp{\varepsilon}} \abs[n]{c\psp{\varepsilon}_k} \leq M^*;
\end{gather*}
and, for $k \in \Lambda\psp{\varepsilon}$, $\norm[n]{w\psp{\varepsilon}_k}_{\HDM} = 1$ and $w\psp{\varepsilon}_k \in \tp_{i=1}^N \HDiMi$.

By virtue of the isometric isomorphism described in \eqref{iso-iso}, the above
relations imply that
\begin{gather*}
\norm{\maxw^{-1}\varphi - \maxw^{-1}\chi_\varepsilon}_{\HOM} \leq \varepsilon,
\quad \maxw^{-1}\chi_\varepsilon = \sum_{k \in \Lambda\psp{\varepsilon}} c\psp{\varepsilon}_k \maxw^{-1} w\psp{\varepsilon}_k,
\end{gather*}
and, for $k \in \Lambda\psp{\varepsilon}$,
$\norm[n]{\maxw^{-1} w\psp{\varepsilon}_k}_{\HOM} = 1$
and
$\maxw^{-1} w\psp{\varepsilon}_k \in \tp_{i=1}^N \HOMi$,
the last relation being a consequence of the tensor-product structure of the
Maxwellian $\maxw$. Thus we have shown that $\maxw^{-1} \varphi \in \HOM$ can
be approximated to within any positive tolerance $\varepsilon$ in the norm of
$\HOM$ by finite linear combinations of normalized members of $\tp_{i\in\ii{N}}
\HOMi$ with the coefficients of the linear combinations having their absolute
sum bounded by $M^*$. In other words, the membership of $\varphi \in
\mathcal{A}_1$ implies the membership of $\maxw^{-1} \varphi$ in the
$\HOM$-based analogue of $\mathcal{A}_1$, namely,
\begin{equation}\label{definition-of-B1}
\mathcal{B}_1 := \bigcup_{M > 0} \overline{ \mathcal{B}_1^o(M) },
\end{equation}
where
\begin{multline}\label{B1oM}
\mathcal{B}_1^o(M) := \Bigg\{ \varphi \in \HOM \colon \varphi = \sum_{k \in \Lambda} c_k w_k,\ w_k \in \tp_{i=1}^N \HOMi,\ \norm{w_k}_{\HOM} = 1,\\
\abs{\Lambda} < \infty \quad \text{and} \quad \sum_{k \in \Lambda} \abs{c_k} \leq M \Bigg\},
\end{multline}
and
\begin{equation}\label{B1-norm}
\norm{\varphi}_{\mathcal{B}_1} := \inf \left\{ M > 0 \colon \varphi \in \overline{ \mathcal{B}_1^o(M) } \right\}.
\end{equation}
In a completely analogous way, the membership of $\maxw^{-1}\varphi$ in
$\mathcal{B}_1$ implies the membership of $\varphi$ in $\mathcal{A}_1$. We then have
the relations
\begin{equation}\label{A1-B1-isometric}
\mathcal{A}_1 = \maxw\, \mathcal{B}_1,
\qquad \norm{\tcdot}_{\mathcal{A}_1} = \norm{ \maxw^{-1} \tcdot }_{\mathcal{B}_1},
\end{equation}
where the last equality follows from the fact that the coefficients of the
approximations to $\varphi$ are the same as the coefficients of the corresponding
approximations to $\maxw^{-1}\varphi$.

\medskip

As the definition of $\mathcal{A}_1$ given in \eqref{definition-of-A1} is fairly abstract, it
is of interest to have conditions in terms of regularity that guarantee
membership in $\mathcal{A}_1$ analogous to the conditions provided in
\cite[Remark 4]{LLM} for the separated representation strategy applied to the
Laplacian defined on a tensor product of one-dimensional domains. This is the
theme of the next section. Because of \eqref{A1-B1-isometric}, we
can pose the problem in terms of membership in the $\HOM$-based $\mathcal{B}_1$
instead with no loss of generality and substantial gain in succinctness;
thus we shall henceforth phrase our results in terms of $\mathcal{B}_1$ rather than $\mathcal{A}_1$.

\section{Characterization of a subspace of rapidly converging
solutions}\label{sec:characterization}

\subsection{Eigenvalues}
We need the following two abstract lemmas, which state standard results
(essentially, the Hilbert--Schmidt theorem and some of its corollaries). As we
could not find these results in the literature in the precise form stated here,
we provide brief proofs of them.

\begin{lemma}\label[lemma]{lem:abstractEV} Let $H$ and $V$ be separable
infinite-dimensional Hilbert spaces, with $V \compEmb H$ and $\overline{V} = H$
in the norm of $H$. Let $a\colon V \times V \to \Real$ be a nonzero, symmetric,
bounded and elliptic bilinear form. Then, there exist sequences of real numbers
$\seq{\lambda_n}{n \in \mathbb{N}}$ and unit $H$-norm members of $V$
$\seq{e_n}{n \in \mathbb{N}}$, which solve the following problem: \emph{Find $\lambda \in \Real$ and
$e \in H \setminus \{ 0 \}$ such that}
\begin{equation}\label{variational-ev}
a(e,v) = \lambda \langle e, v\rangle_H \quad \forall\,v \in V.
\end{equation}
The $\lambda_n$, which can be assumed to be in increasing order with respect to $n$,
are positive, bounded from below away from $0$, and
$\lim_{n\to\infty}\lambda_n = \infty$.

Additionally, the $e_n$ form an $H$-orthonormal system whose $H$-closed span is
$H$ and the rescaling $e_n/\sqrt{\lambda_n}$ gives rise to an $a$-orthonormal
system whose $a$-closed span is $V$.
\end{lemma}
\begin{proof} This proof is an adaptation of the proof of Theorem~IX.31 in
\cite{Brezis}. The Lax--Milgram lemma implies the existence of an operator
$\tilde T\colon H \rightarrow V$ where, given $h \in H$, $\tilde T(h)$ is defined
as the unique solution in $V$ to the variational problem
\begin{equation}\label{definition-of-T-tilde}
a(\tilde T(h),v) = \langle h, v\rangle_H \quad \forall\, v \in V.
\end{equation}
It also follows, via the elliptic stability estimate of the Lax--Milgram lemma
and the continuity of the embedding $V \hookrightarrow H$, that $\tilde T$ is
bounded. Let $i\colon V \rightarrow H$ denote the embedding operator that maps $V$
into $H$, i.e., $v \in V \mapsto i(v) = v \in H$. Then, $T := i \circ \tilde T$
is a bounded operator defined on $H$ with values in $H$; as $i\colon V \rightarrow H$
is a compact linear operator, it follows that $T \colon H \rightarrow H$ is a compact
linear operator. Further, for all $(h,h') \in H \times H$,
\begin{multline*}
\langle T(h),h' \rangle_H = \langle \tilde T(h), h' \rangle_H = \langle h', \tilde T(h) \rangle_H = a(\tilde T(h'), \tilde T(h))\\
= a(\tilde T(h), \tilde T(h')) = \langle h, \tilde T(h') \rangle_H = \langle h, T(h') \rangle_H,
\end{multline*}
whence $T$ is self-adjoint. Thus, thanks to Theorem~VI.11 in
\cite{Brezis}, there exists an $H$-orthonormal system $\seq{ e_n
}{n \geq 1}$ of eigenvectors of $T$ such that
\begin{equation}\label{Fourier-Parseval}
h = \sum_{n=1}^\infty \langle h, e_n \rangle_H e_n \quad\text{and}\quad
\norm{h}_H^2 = \sum_{n=1}^\infty \langle h, e_n \rangle_H^2\qquad \forall\,h \in H.
\end{equation}

As, for all $h \in H$, $\langle T(h),h\rangle_H = a(\tilde T(h), \tilde T(h))$
and $a$ is $V$-elliptic, all the eigenvalues of $T$ are nonnegative. Also, as
$T$ is bounded, the set of its eigenvalues is also bounded. Now, by Theorem~VI.8 in
\cite{Brezis}, the set of nonzero eigenvalues of $T$ is either empty, or finite, or
countable with $0$ as its only accumulation point. However, on account of
\eqref{Fourier-Parseval}, the latter alternative is then the one that holds.

If $0$ were an eigenvalue of $T$, there would exist $e \in H \setminus \{ 0 \}$
such that $T(e) = 0$; i.e., $e \in \mathrm{Ker}(T)$. However, from
\eqref{definition-of-T-tilde} we then have that $e \in V^{\perp_H}$. As $H = \overline{V} \oplus
V^{\perp_H}$ in the norm of $H$ and $V$ is dense in $H$, $V^{\perp_H} = \{ 0
\}$, which contradicts $e \neq 0$. Therefore, $0$ is not an eigenvalue of $T$.

From the above, we can take the eigenvectors $e_n$ of \eqref{Fourier-Parseval}
as associated to positive eigenvalues $\mu_n$ bounded from above, arranged in
decreasing order ($\mu_{n+1} \leq \mu_n$ for $n \geq 1$) with $\lim_{n \to
\infty} \mu_n = 0$. A consequence of the absence of $0$ from the spectrum of
$T$ is that all the eigenvectors of $T$ have to be members of the smaller space
$V$.

Assuming that $\mu \neq 0$ and $e \in V \setminus \{ 0 \}$, $T(e) = \mu e$ if,
and only if, $a(e,w) = \mu^{-1} \langle e, w\rangle_H$ for all $w\in V$. Then,
all the eigenvalues of the eigenvalue problem \eqref{variational-ev} are
reciprocals of eigenvalues of $T$ with the possible exception of $0$. However,
from the $V$-ellipticity of $a$, $0$ cannot be an eigenvalue of the problem
\eqref{variational-ev}. On defining $\lambda_n := \mu_n^{-1}$ and setting the
$e_n$ to be the same as in \eqref{Fourier-Parseval} we obtain the desired
existence and distribution statements about of the eigenvalues of
\eqref{variational-ev}.

We observe from $a(e_n,e_m) = \lambda_n \langle e_n, e_m \rangle_H$, $n \geq
1$, that $\seq{ e_n/\sqrt{\lambda_n} }{n \geq 1}$ is an $a$-orthonormal system
in $V$. Let us denote the $a$-closure of its span by $\hat{V}$. Then, $v \in
\hat{V}^{\perp_a}$ if, and only if, $a(v,e_n) = 0$ for all $n \geq 1$. As each
$e_n$ is an eigenfunction of the problem \eqref{variational-ev} associated to a
nonzero eigenvalue, it follows from \eqref{Fourier-Parseval} that $v = 0$ and
therefore $\hat{V}^{\perp_a} = \{ 0 \}$. Thus, $V = \hat{V} \oplus
\hat{V}^{\perp_a} = \hat{V}$. This, together with \eqref{Fourier-Parseval}
itself, completes the proof.
\end{proof}

\begin{lemma}\label[lemma]{lem:coefDecay}
Let the spaces $H$, $V$ and the bilinear form $a$ be as in the statement of
\cref{lem:abstractEV} and let $(\lambda_n,e_n) \in \Real_{>0} \times V$ be
the eigenpairs of \eqref{variational-ev} obtained there. Then,
\begin{equation}\label{Fourier-Parseval-H}
h = \sum_{n=1}^\infty \langle h, e_n \rangle_H e_n \quad\text{and}\quad
\norm{h}_H^2 = \sum_{n=1}^\infty \langle h, e_n \rangle_H^2\qquad \forall\,h \in H,
\end{equation}
and
\begin{equation}\label{Fourier-Parseval-V}
v = \sum_{n=1}^\infty a\left(v, \frac{e_n}{\sqrt{\lambda_n}}\right) \frac{e_n}{\sqrt{\lambda_n}} \quad\text{and}\quad
\norm{v}_a^2 = \sum_{n=1}^\infty a\left( v, \frac{e_n}{\sqrt{\lambda_n}} \right)^2 \qquad \forall\,v \in V.
\end{equation}
Further,
\begin{equation}\label{fastDecay}
h \in H \quad\text{and}\quad \sum_{n=1}^\infty \lambda_n \langle h, e_n\rangle_H^2 < \infty \iff h \in V.
\end{equation}
\end{lemma}
\begin{proof}
The expression \eqref{Fourier-Parseval-H} is just a restatement of
\eqref{Fourier-Parseval} in the proof of \cref{lem:abstractEV}, but unlike
there, here we emphasize that the $e_n$ belong to $V$. The expression
\eqref{Fourier-Parseval-V} comes from the identity between $V$ and the
$a$-closed span of the $a$-orthonormal set $\seq{ e_n/\sqrt{\lambda_n} }{n \geq
1}$---part of the statement of \cref{lem:abstractEV}---via, for example,
Theorem VI.9 of \cite{Brezis}.

As $(\lambda_n,e_n)$ is an eigenpair of \eqref{variational-ev}, $a(v ,e_n /
\sqrt{\lambda_n}) = \sqrt{\lambda} \langle v, e_n \rangle_H$ for all $v \in V$;
this and the second expression of \eqref{Fourier-Parseval-V} give the
right-to-left implication in \eqref{fastDecay}. Let us now consider an $h \in
H$ that satisfies the left-hand side of \eqref{fastDecay}. As the $e_n$ are
members of $V$, the partial sums
\[ h_k := \sum_{n=1}^k \langle h, e_n\rangle_H e_n \]
also belong to $V$. The $a$-orthonormality of the $e_n/\sqrt{\lambda_n}$ leads
to the equality, for $1 \leq k < l$,
\[ \norm{ h_l - h_k}_a^2 = \sum_{n = k+1}^l \lambda_n \langle h, e_n \rangle_H^2. \]
As the real series $\sum_{n = 1}^\infty \lambda_n \langle h, e_n \rangle_H^2$
is assumed to converge, the above expression tends to $0$ as $k$ and $l$ tend
to $\infty$. Hence, $\seq{ h_k }{k \geq 1}$ is a Cauchy sequence in $V$ and
thus converges to some $\hat h \in V$. As $V$ is continuously embedded in $H$
(part of being compactly embedded), the limit $\hat h$ has to be the same limit
the $h_k$ have in $H$. That is, $h = \hat h \in V$. This completes the proof of
\eqref{fastDecay}.
\end{proof}

The hypotheses of \cref{lem:abstractEV} and \cref{lem:coefDecay} are satisfied by the eigenvalue
problems
\begin{equation}\label{partial-ev}
\langle e\psp{i}, \varphi\rangle_{\HOMi} = \lambda\psp{i} \langle e\psp{i}, \varphi \rangle_{\LTMi} \qquad \forall\,\varphi \in \HOMi,
\end{equation}
(for $i \in \ii{N}$ here and in what follows), and
\begin{equation}\label{full-ev}
\langle e, \varphi\rangle_{\HOM} = \lambda \langle e, \varphi \rangle_{\LTM} \qquad \forall\,\varphi \in \HOM,
\end{equation}
whence their solutions do have the distribution, orthogonality and spanning
properties stated in that lemma (the hypothesis $\overline{V} = H$, which
is not discussed elsewhere, follows from the density of infinitely
differentiable and compactly supported functions in any weighted $\LL^2$
space). In particular, they have sequences of solutions (eigenpairs)
$\seq{(\lambda\psp{i}_n, e\psp{i}_n)}{n \in \natural}$ and
$\seq{(\lambda_n,e_n)}{n \in \natural}$, respectively, with
\begin{equation}\label{HOMi-from-ev}
\varphi \in \LTMi \quad \text{and} \quad \sum_{n=1}^\infty \lambda\psp{i}_n \langle \varphi, e\psp{i}_n \rangle_{\LTMi}^2 < \infty \iff \varphi \in \HOMi,
\end{equation}
and
\begin{equation}\label{HOM-from-ev}
\varphi \in \LTM \quad \text{and} \quad \sum_{n=1}^\infty \lambda_n \langle \varphi, e_n \rangle_{\LTM}^2 < \infty \iff \varphi \in \HOM.
\end{equation}
Next, we exploit the special tensor-product structure of the full Maxwellian
$\maxw$ to characterize the eigenpairs of its associated eigenvalue problem
\eqref{full-ev} in terms of the eigenpairs of the eigenvalue problem
\eqref{partial-ev} associated to the partial Maxwellians $M_i$.

\begin{lemma}\label[lemma]{lem:tensor-ev}
The net $\seq{(\lambda_{\vect{n}}, e_{\vect{n}})}{\vect{n} = (n_1, \dotsc, n_N)
\in \natural^N}$ is a full system of solutions of the eigenvalue problem
\eqref{full-ev}, where
\begin{equation}\label{tensor-ev}
\lambda_{\vect{n}} := 1 + \sumi (\lambda\psp{i}_{n_i} - 1)
\qquad\text{and}\qquad e_{\vect{n}} := \tp_{i=1}^N e\psp{i}_{n_i}.
\end{equation}
\end{lemma}
\begin{proof}
Given $\tau = \tp_{i\in\ii{N}} \tau\psp{i} \in \tp_{i\in\ii{N}} \CIC(D_i)$, we have that
\begin{equation*}
\begin{split}
&\left\langle e_{\vect{n}}, \tau \right\rangle_{\HDM}
 = \left\langle e_{\vect{n}}, \tau \right\rangle_{\LTM}
+ \sum_{j=1}^N \left\langle \grad e\psp{j}_{n_j}, \grad \tau\psp{j} \right\rangle_{[\LTMj]^d} \prod_{\allibutj}^N \left\langle e\psp{i}_{n_i}, \tau\psp{i} \right\rangle_{\LTMi}\\
&\qquad = \left\langle e_{\vect{n}}, \tau \right\rangle_{\LTM}
+ \sum_{j=1}^N (\lambda\psp{j}_{n_j} - 1) \left\langle e\psp{j}_{n_j}, \tau\psp{j} \right\rangle_{\LTMj} \prod_{\allibutj}^N \left\langle e\psp{i}_{n_i}, \tau\psp{i} \right\rangle_{\LTMi}
= \lambda_{\vect{n}} \left\langle e_{\vect{n}}, \tau \right\rangle_{\LTM}.
\end{split}
\end{equation*}
Since the span of $\tp_{i=1}^N \HDiMi$ is dense in $\HDM$ (as is readily seen from
\cref{cor:tensorDensity} and \eqref{iso-iso}), the equality of the first and
the last expression in the chain of equalities above is valid for all $\tau \in
\HDM$. Hence, $(\lambda_{\vect{n}}, e_{\vect{n}})$ is an eigenpair of
\eqref{full-ev}. Further, we deduce from the chain of equalities above that $e_{\vect{n}}$ is
orthogonal to $e_{\vect{m}}$ in both $\LTM$ and $\HOM$ if $\vect{n} \neq
\vect{m}$.

From \eqref{Fourier-Parseval-V} in \cref{lem:coefDecay}, for $i \in \ii{N}$,
$\overline{\mathrm{span}\seq{ e\psp{i}_n }{n \geq 1}} = \HOMi$.
Hence, by \eqref{cor:tensorDensity} and \eqref{iso-iso}, 
\[ \overline{ \tp_{i=1}^N \mathrm{span}\seq{ e\psp{i}_n }{n \geq 1} }
\subset \overline{ \mathrm{span}\seq{ e_{\vect{n}} }{ \vect{n} \in \natural^N } } = \HOM. \]
Thus, $\seq{ e_{\vect{n}} }{ \vect{n} \in \natural^N }$ forms an orthogonal
system that spans $\HOM$. Therefore, by Theorem~VI.9 of
\cite{Brezis}, all the eigenpairs of the (full) Maxwellian eigenvalue problem
\eqref{full-ev} have the form $(\lambda_{\vect{n}}, e_{\vect{n}})$ as given in
\eqref{tensor-ev} (modulo linear combinations of eigenfunctions belonging to
the same eigenspace).
\end{proof}

It follows from \cref{lem:tensor-ev} that the eigenvalues and eigenfunctions
of \eqref{full-ev} are more naturally indexed by $\natural^N$ than by
$\natural$; in what follows, we shall refrain from indexing \emph{contra
natura}.

\subsection{Characterization via summability of Fourier coefficients} As by
\cref{lem:tensor-ev} the sequence $\seq{(\lambda_{\vect{n}},
e_{\vect{n}})}{\vect{n}\in\natural^N}$ is a full system of eigenpairs of
\eqref{full-ev}, \eqref{Fourier-Parseval-V} in \cref{lem:coefDecay} ensures
that, for all $\tau \in \HOM$,
\begin{equation*}
\tau
= \sum_{\vect{n} \in \natural^N} \left\langle \tau, \frac{e_{\vect{n}}}{\sqrt{\lambda_{\vect{n}}}} \right\rangle_{\HOM} \frac{e_{\vect{n}}}{\sqrt{\lambda_{\vect{n}}}}
= \sum_{\vect{n} \in \natural^N} \sqrt{\lambda_{\vect{n}}} \left\langle \tau, e_{\vect{n}} \right\rangle_{\LTM} \frac{e_{\vect{n}}}{\sqrt{\lambda_{\vect{n}}}} \qquad \text{in } \HOM.
\end{equation*}
Hence, given the tensor-product structure of the $e_{\vect{n}}$ and the unit
$\HOM$-norm of the $e_{\vect{n}}/\sqrt{\lambda_{\vect{n}}}$, we can guarantee
that $\tau \in \mathcal{B}_1$ (cf.\ \eqref{definition-of-B1}) if
\[ \sum_{\vect{n} \in \natural^N}\sqrt{\lambda_{\vect{n}}} \abs[n]{\langle\tau,e_{\vect{n}}\rangle_{\LTM}} < \infty. \]
In turn, this holds if
\begin{equation}\label{sufficient-for-B1}
A := \sum_{\vect{n} \in \natural^N} \frac{\lambda_{\vect{n}}}{\sigma_{\vect{n}}} < \infty \qquad \text{and}\qquad
B := \sum_{\vect{n} \in \natural^N} \sigma_{\vect{n}} \langle\tau, e_{\vect{n}} \rangle_{\LTM}^2 < \infty,
\end{equation}
where $\seq{\sigma_{\vect{n}}}{\vect{n}\in\natural^N}$ is a sequence of positive
real numbers that are to be chosen below. We note that the requirement of $B$
being finite can be seen---for $\sigma_{\vect{n}} = \lambda_{\vect{n}}$, for
example, this is certainly the case, as follows from \eqref{HOM-from-ev}---as a
regularity requirement on $\tau$. Thus, there is a trade-off in
\eqref{sufficient-for-B1} between the requirement that the $\sigma_{\vect{n}}$
grow fast enough to ensure the finiteness of $A$ and the desirability of the
$\sigma_{\vect{n}}$ growing slow enough to avoid demanding more regularity than
necessary of the functions $\tau$ for which $B$ is finite.

As a first step in formalizing the above we consider, given a net $\Sigma =
\seq{\sigma_{\vect{n}}}{\vect{n} \in \natural^N}$ with entries in
$\Real_{>0}$, the space of all those $\LTM$ functions for which the term $B$, as defined in
\eqref{sufficient-for-B1}, is finite:
\begin{subequations}\label{defHSigmaM}
\begin{equation}\label{HSigmaM}
\HH^\Sigma_\maxw(\D) := \left\{ \varphi \in \LTM \colon \sum_{\vect{n} \in \natural^N} \sigma_{\vect{n}}
\left\langle \varphi, e_{\vect{n}} \right\rangle_{\LTM}^2 < \infty \right\}.
\end{equation}
We equip $\HH^\Sigma_\maxw(\D)$ with the norm
\begin{equation}\label{nHSigmaM}
\norm{\varphi}_{\HH^\Sigma_\maxw(\D)} := \left (\sum_{\vect{n} \in \natural^N} \sigma_{\vect{n}} \left\langle \varphi, e_{\vect{n}} \right\rangle_{\LTM}^2 \right)^{1/2}.
\end{equation}
\end{subequations}
It is readily seen that, if there exists a $\sigma > 0$ with
$\sigma_{\vect{n}} \geq \sigma$ for all $\vect{n} \in \natural^N$,
then $\HH^\Sigma_\maxw(\D)$ is a separable Hilbert space that is continuously
embedded in $\LTM$. Further, if there exists a $\sigma' > 0$ such that
$\sigma_{\vect{n}} \geq \sigma' \lambda_{\vect{n}}$ for all $\vect{n} \in
\natural^N$, then $\HH^\Sigma_\maxw(\D)$ is continuously embedded in $\HOM$ and,
thanks to \cref{lem:compEmb}, it is compactly embedded in $\LTM$.

At this stage we could just choose $\Sigma$ to be, e.g., $\sigma_{\vect{n}} =
\lambda_{\vect{n}} \norm{\vect{n}}_2^\alpha$ for some $\alpha > N$ and an
application of a multiple series version of the integral test for convergence
(see, for example, \cite[Proposition 7.57]{GhorpadeLimaye}) would render the
sum $A$ in \eqref{sufficient-for-B1} finite. However, the resulting
space $\HH^\Sigma_\maxw(\D)$ would then still have quite an abstract description.
What we therefore wish to do instead is to choose each $\sigma_{\vect{n}}$ as a suitable polynomial
function of the $(\lambda\psp{1},\dotsc,\lambda\psp{N})$. Then, under certain
reasonable conditions, which we will make explicit below, we shall
be able to characterize the resulting space in terms of regularity properties.
One of these conditions has to do with the fact that we can only know that $A$
of \eqref{sufficient-for-B1} is finite, with $\sigma_{\vect{n}}$ as a certain
polynomial function of the $\lambda_{\vect{n}}$, if we have some information
about the asymptotic behavior of the $\lambda_{\vect{n}}$. Consequently, we adopt the
following hypothesis.

\begin{hypothesis}\label{hyp:asymptotics}
For each $i \in \ii{N}$ there exist positive real numbers $c\psp{i}_1$ and $c\psp{i}_2$ and $n\psp{i}
\in \natural$ such that (with $d$ signifying the common dimension of the single-spring
configuration domains $D_i$)
\[ c\psp{i}_1 n^{2/d} \leq \lambda\psp{i}_n \leq c\psp{i}_2 n^{2/d}\qquad \forall n \geq n\psp{i},
\]
where $\lambda\psp{i}_n$ is the $n$-th member of the (ordered, with
repetitions according to multiplicity) sequence of eigenvalues of
\eqref{partial-ev}.
\end{hypothesis}

\begin{remark}\label[remark]{rem:asymptotics}
\cref{hyp:asymptotics} basically consists of assuming that
the eigenvalues of the problem \eqref{partial-ev} behave like the eigenvalues
of a regular elliptic operator, such as the Poisson operator. If the partial
Maxwellian $M_i$ comes from either the FENE model \eqref{FENE-model} or the
CPAIL model \eqref{CPAIL-model} \cref{hyp:asymptotics} holds; see
\cref{cor:easyAsymptotics} in \cref{sec:ev-asymptotics} for a proof (see
also \cref{rem:RussianLiterature}).
\end{remark}

\begin{theorem}\label{thm:HMix-in-B1}
Let $\Tau\psp{m} = \seq{ \tau\psp{m}_{\vect{n}} }{ \vect{n} \in \natural^N}$
be defined by
\begin{equation}\label{HMixIndices}
\tau\psp{m}_{\vect{n}} := \prod_{i=1}^N \left( \lambda\psp{i}_{n_i} \right)^m \qquad \forall\,\vect{n} \in \natural^N.
\end{equation}
Then, $\HH^{\Tau\psp{m}}_\maxw(\D) \subset \mathcal{B}_1$ if $m > \frac{d}{2} + 1$.
\end{theorem}
\begin{proof}
According to the previous discussion, the stated inclusion will hold once we have
shown that the infinite sum over $\vect{n} \in \natural^N$ of $\lambda_{\vect{n}} /
\tau\psp{m}_{\vect{n}}$ converges; i.e., that $A$ in
\eqref{sufficient-for-B1} is finite. To prove this, we start by noting that,
modulo a decrease of $c\psp{i}_1$ and an increase of $c\psp{i}_2$, we can take
$n\psp{i}=1$ in \cref{hyp:asymptotics} as a consequence of all the
$\lambda\psp{i}_n$ being positive; we do so from now on. This, together
with \eqref{tensor-ev} and \cref{hyp:asymptotics}, yields that
\begin{equation}\label{chain}
\frac{\lambda_{\vect{n}}}{\tau\psp{m}_{\vect{n}}}
\leq \frac{ \sum_{i=1}^N \lambda\psp{i}_{n_i} }{\prod_{i=1}^N \left( \lambda\psp{i}_{n_i} \right)^m}
\leq C \frac{ \sum_{i=1}^N n_i^{2/d} }{\prod_{i=1}^N \left( n_i^{2/d} \right)^m}
\end{equation}
for all $\vect{n} \in \natural^N$ and some $C > 0$ that depends on the
$c\psp{i}_1$, the $c\psp{i}_2$, $N$ and $d$ only. Clearly, it will be enough to
show that the right-most expression in \eqref{chain} results in a convergent series. Now,
\begin{equation*}
\sum_{\vect{n}} \frac{ \sum_{i=1}^N n_i^{2/d} }{\prod_{i=1}^N \left( n_i^{2/d} \right)^m}
= \sum_{\vect{n}} \sum_{i=1}^N \frac{n_i^{2/d-2 m/d}}{\prod_{\substack{j=1\\j\neq i}}^N n_j^{2 m/d}}
= \sum_{i=1}^N \left(\prod_{\substack{j=1\\j\neq i}}^N \sum_{n_j=1}^\infty n_j^{-2m/d}\right) \left(\sum_{n_i=1}^\infty n_i^{2/d(1-m)}\right)
\end{equation*}
where the constraint on $m$ ensures that all the resulting one-dimensional sums are finite.

\end{proof}

For later reference we introduce another family of weights that also produces
subspaces of $\mathcal{B}_1$.

\begin{theorem}\label{thm:HUnif-in-B1}
Let $\Upsilon\psp{m} = \seq{ \upsilon\psp{m}_{\vect{n}} }{ \vect{n} \in \natural^N}$ be defined by
\begin{equation}\label{HUnifIndices}
\upsilon\psp{m}_{\vect{n}} := \left(\sum_{i=1}^N \lambda\psp{i}_{n_i} \right)^m \qquad \forall\,\vect{n}\in \natural^N.
\end{equation}
Then, $\HH^{\Upsilon\psp{m}}_\maxw(\D) \subset \mathcal{B}_1$ if $m > 1 + {\textstyle\frac{1}{2}}N d$.
\end{theorem}
\begin{proof}
Using \cref{hyp:asymptotics} and the already mentioned multiple series
version of the integral test for convergence it can be shown that the result
hinges on the finiteness of the integral $\int_{[1,\infty)^N} \left(
\sum_{i=1}^N x_i^{2/d} \right)^{1-m} \dd x$. Thanks to the equivalence
of the $2/d$-quasinorm to the $2$-norm in Euclidean space and since $[1,\infty)^N \subset \{ x \in \Real_{\geq 0}^N \colon \norm{x}_2 \geq 1
\}$, the finiteness of that integral is, in turn, implied by the finiteness of
the integral
\begin{equation*}
\int_{ \{ x \in \Real_{\geq 0}^N \colon \norm{x}_2 \geq 1 \} } \norm{x}_2^{\frac{2}{d}(1-m)} \dd x
= C_N \int_1^\infty r^{\frac{2}{d}(1-m)+N-1} \dd r,
\end{equation*}
where $C_N$ is the $(N-1)$-dimensional volume of the surface $\{ x \in
\Real_{\geq 0}^N \colon \norm{x}_2 = 1 \}$. As it is assumed that $m > 1 +
\frac{1}{2} N d$, the last of the above integrals is finite and the proof
is completed.

\end{proof}

The definition of $\HH^{\Tau\psp{m}}_{\maxw}(\D)$ given by
\eqref{HMixIndices} is less abstract than the definition of
$\mathcal{B}_1$ (given in \eqref{definition-of-B1}). However, we can describe
subspaces of the former space in even less abstract terms by showing that
certain regularity conditions translate into summability conditions expressed
in terms of Fourier coefficients, such as those that define
$\HH^{\Tau\psp{m}}_{\maxw}(\D)$ (cf.\ \eqref{HSigmaM}). In order to
understand the appropriate regularity requirements for this purpose, we need to
study the regularity properties of certain degenerate elliptic operators in
Maxwellian-weighted Sobolev spaces.

\subsection{Characterization via membership in mixed-order weighted Sobolev
spaces} We start by adopting two further hypotheses.

\begin{hypothesis}\label{hyp:potential2}
For $i \in \ii{N}$ the spring potential $U_i$ is
monotonic increasing and convex.
\end{hypothesis}

\begin{hypothesis}\label{hyp:power-like} For $i \in \ii{N}$ there exists a
distance $\gamma_i \in (0,\sqrt{b_i})$, an exponent $\alpha_i > 1$
and a function $h_i \in \CC^3\left([0,\gamma_i]\right)$ that is positive on
$[0,\gamma_i]$, such that
\[ M_i(\ponf) = h_i(\mathfrak{d}_i(\ponf))\, \mathfrak{d}_i(\ponf)^{\alpha_i} \]
for all $\ponf \in D_i$ such that $\mathfrak{d}_i(\ponf) \in (0,\gamma_i)$,
where $\mathfrak{d}_i$ is the distance-to-the-boundary function in $D_i$.
\end{hypothesis}

\begin{remark}\label[remark]{rem:hypotheses2}
\cref{hyp:potential2} can be regarded as a strengthening of
\cref{hyp:potential}. It is easy to check that springs obeying the FENE
model \eqref{FENE-model} or the CPAIL model \eqref{CPAIL-model} comply with it.

With \cref{hyp:power-like} we are restricting ourselves, essentially, to
power weights. The compliance of the FENE and the CPAIL models with it is also
easy to check if their parameter $b_i$ is greater than $2$ in the FENE case and
greater than $3$ in the CPAIL case.
\end{remark}

\begin{lemma}\label[lemma]{lem:partialDensity2}
For $i \in \ii{N}$,
\begin{enumerate}
\item[(a)] the space $\CIC(D_i)$ is dense in $\HOMi$;
\item[(b)] the space $\CC^\infty(\overline{D_i})$ is dense in $\HH^m_{M_i}(D_i)$, for $m \in \natural$.
\end{enumerate}
\end{lemma}
\begin{proof}
In Proposition 9.10 (resp.\ Theorem 7.2) of \cite{Kufner:1985} the result (a)
(resp.\ (b)) is stated for weights that are powers greater than $1$ (resp.\
greater or equal than $0$) of the distance-to-the-boundary function; the
bilateral boundedness of the function $h_i$ by positive constants, implied by
\cref{hyp:power-like}, extends the statement to our case.
\end{proof}

The additional requirements on the potentials $U_i$, $i \in \ii{N}$, and the preceding lemma
allow us to prove a first elliptic regularity result.

\begin{lemma}\label[lemma]{lem:ellipReg} Let $i \in \ii{N}$ and suppose that $g \in \LTMi$;
 then, there exists a constant $C_i > 0$, independent of $g$, such that the solution $z \in \HOMi$ of
\begin{equation}\label{partial-elliptic-problem}
\langle z, \varphi \rangle_{\HOMi} = \langle g, \varphi \rangle_{\LTMi}
\qquad \forall\,\varphi \in \HOMi
\end{equation}
obeys the regularity estimate
\begin{equation*}
\norm{z}_{\HH^2_{M_i}(D_i)} + \norm{\frac{1}{M_i} \div(M_i\grad z)}_{\LTMi}
\leq C_i \norm{g}_{\LTMi}.
\end{equation*}
\end{lemma}
\begin{proof}
By \cref{hyp:potential} and \cref{hyp:potential2} the
function $V_i \colon D_i \rightarrow \Real$ defined by $V_i :=
\frac{1}{2} U_i(\frac{1}{2}\abs{\tcdot}^2)$ is convex and tends to
$+\infty$ as its argument approaches the boundary of $D_i$ from within. Then, it
follows from Theorem~3.4 of \cite{dPL1} and the density of $\CIC(D_i)$ in
$\HH^1_{M_i}(D_i)$ given in part (a) of \cref{lem:partialDensity2} that
there exists a unique solution $\tilde z$ in $\{ u \in \HH^2_{M_i}(D_i) \colon
\grad V_i \cdot \grad u \in \LTMi \}$ to
\begin{equation}\label{partial-elliptic-problem-2}
\onehalf \tilde z - \onehalf \lapl \tilde z + \grad V_i \cdot \grad \tilde z = \onehalf g,
\end{equation}
considered as an equation in $\LTMi$, and it obeys the estimates
\begin{subequations}\label{estimate}
\begin{align}
\norm{\tilde z}_{\LTMi} \leq 2 \norm{\onehalf g}_{\LTMi} & = \norm{g}_{\LTMi},\\
\norm{\grad \tilde z}_{[\LTMi]^d} \leq 2\sqrt{2} \norm{\onehalf g}_{\LTMi} & = \sqrt{2}\norm{g}_{\LTMi},\\
\norm{\grad\grad \tilde z}_{[\LTMi]^{d \times d}} \leq 4 \norm{\onehalf g}_{\LTMi} & = 2 \norm{g}_{\LTMi}.
\end{align}
\end{subequations}
The regularity of $M_i$ and $\tilde z$ admits the use of the
Leibniz formula for the product of a regular distribution and a continuously
differentiable function provided in \cref{lem:LeibnizFormula} in \cref{sec:distrResults}.
We can then write $M_i \lapl \tilde z - 2 M_i \grad V_i \cdot \grad
\tilde z = \div(M_i\grad \tilde z)$ (for this we have used that $M_i$ is
proportional to $\exp(-2 V_i)$ (cf.\ \eqref{partial-maxw}). Plugging this into
\eqref{partial-elliptic-problem-2} gives
\begin{equation}\label{uglyEstimate}
\norm{\frac{1}{M_i}\div(M_i\grad\tilde z)}_{\LTMi} \leq \norm{g}_{\LTMi} + \norm{\tilde z}_{\LTMi} \leq 2 \norm{g}_{\LTMi}.
\end{equation}

Multiplying \eqref{partial-elliptic-problem-2} by $2 M_i$ and using the Leibniz
formula for the product of a regular distribution and a continuously
differentiable function again, we find that
\[ \int_{D_i} \tilde z \varphi M_i + \int_{D_i} \grad \tilde z \cdot \grad \varphi M_i = \int_{D_i} g \varphi \]
for all $\varphi \in \CIC(D_i)$. It follows from the density of $\CIC(D_i)$ in
$\HH^1_{M_i}(D_i)$ and the uniqueness of the solution $z$ of
\eqref{partial-elliptic-problem} that $z = \tilde z$ and hence \eqref{estimate}
and \eqref{uglyEstimate} give the desired result.
\end{proof}

In order to obtain an iterated elliptic regularity result, we need the
technical lemma that follows.

\begin{lemma}[Hardy inequalities]\label{lem:Hardy} Let $H > 0$. Then, there
exists $C_H > 0$ such that
\begin{equation}\label{standardHardy}
\int_0^H \frac{1}{y^2} \left( \int_0^y f(s) \dd s \right)^2 \dd y \leq C_H \int_0^H f(s)^2 \dd s\qquad\forall\,f\in \LL^1((0,H)).
\end{equation}
If $\alpha > 1$, then there exists $C_{H,\alpha}$ such that
\begin{equation}\label{anotherHardy}
\int_0^H y^{\alpha-2} f(y)^2 \dd y \leq C_{H,\alpha} \int_0^H y^\alpha \left[ f(y)^2 + f'(y)^2 \right] \dd y \qquad\forall\,f\in\HH^1_{(\tcdot)^\alpha}((0,H)).
\end{equation}
\end{lemma}
\begin{proof}
The inequality \eqref{standardHardy} follows from the standard Hardy inequality (the $H = \infty$ case);
see, for example, \cite[Proposition VIII.18.1]{DiBenedetto}. Alternatively, see \cite[Theorem 1.14]{OK} for a
very general form, which encompasses \eqref{standardHardy}.

To prove \eqref{anotherHardy} we will use a procedure inspired by the proof of
Theorem 8.2 of \cite{Kufner:1985}. The first ingredient is the inequality
\[ \int_0^H y^{\alpha-2} f(y)^2 \dd y \leq \const \int_0^H y^\alpha f'(y)^2 \dd y \]
valid for all $f$ in $\CC^1([0,H])$ such that $f(H) = 0$ (see, e.g.,
\cite[Example 6.8.ii]{OK}). Let now $\varphi_0$ and $\varphi_1$ form a smooth
partition of unity subordinate to the covering $H = (0,2H/3) \cup (H/3, H)$.
Then, given any $f \in \CC^1([0,H])$, let $f_0 := \varphi_0 f$ and $f_1 :=
\varphi_1 f$. Using the above inequality, the validity of \eqref{anotherHardy} for
$\CC^1([0,H])$ functions follows from
\begin{equation*}
\begin{split}
\norm{ f }_{\LL^2_{(\tcdot)^{\alpha-2}}((0,H))}
& \leq \norm{ f_0 }_{\LL^2_{(\tcdot)^{\alpha-2}}((0,2H/3))} + \norm{ f_1 }_{\LL^2_{(\tcdot)^{\alpha-2}}((H/3,H))}\\
& \leq \const*^{1/2} \norm{ f_0' }_{\LL^2_{(\tcdot)^\alpha}((0,2H/3))} + \norm{ (\tcdot)^{-1} f_1 }_{\LL^2_{(\tcdot)^\alpha}((H/3,H))}\\
& \leq \const*^{1/2} \norm{ \varphi_0 f' + \varphi_0' f}_{\LL^2_{(\tcdot)^\alpha}((0,2H/3))} + 3/H \norm{ f_1 }_{\LL^2_{(\tcdot)^\alpha}((H/3,H))}\\
& \leq \const \norm{ f' }_{\LL^2_{(\tcdot)^\alpha}((0,2H/3))} + \const \norm{ f }_{\LL^2_{(\tcdot)^\alpha}((0,2H/3))} + \const \norm{ f }_{\LL^2_{(\tcdot)^\alpha}((H/3,H))}\\
& \leq \const \left( \norm{f}_{\LL^2_{(\tcdot)^\alpha}((0,H))}^2 + \norm{f'}_{\LL^2_{(\tcdot)^\alpha}((0,H))}^2 \right)^{1/2}.
\end{split}
\end{equation*}
The validity of the inequality for all $f \in \HH^1_{(\tcdot)^\alpha}((0,H))$ is
then a consequence of the density of $\CC^1([0,H])$ functions in
$\HH^1_{(\tcdot)^\alpha}((0,H))$, the completeness of
$\LL^2_{(\tcdot)^{\alpha-2}}((0,H))$ and the continuity of the injection of that
latter space into $\LL^2_{(\tcdot)^\alpha}((0,H))$.
\end{proof}

We shall now iterate \cref{lem:ellipReg}: extra regularity for $g$ implies extra
regularity for $z$.

\begin{lemma}\label[lemma]{lem:H2-to-H4}
Let $i \in \ii{N}$ and $g \in \HH^2_{M_i}(D_i)$. Then, there exists a constant $C_i>0$,
independent of $g$, such that the solution $z \in \HOMi$ of
\begin{equation}\label{partial-elliptic-problem-3}
\langle z, \varphi \rangle_{\HOMi} = \langle g, \varphi \rangle_{\LTMi} \qquad \forall\,\varphi\in\HOMi
\end{equation}
obeys the regularity estimate
\[ \norm{z}_{\HH^4_{M_i}(D_i)} \leq C_i \norm{g}_{\HH^2_{M_i}(D_i)}. \]
\end{lemma}
\begin{proof}\resetconst
The core of this proof is based on Lemmas 3.1 and 3.3 of \cite{French:1987}.
As their adaptation to our geometry is nontrivial, we give a
detailed argument. Note that in this proof we shall omit the spring index $i$ in
order to avoid cluttering the notation.

\underline{Part 1}: We start by describing a change of coordinates and how
\eqref{partial-elliptic-problem} transforms under it.

Given $\ponf \in \Real^d$, let $\ponf'$ denote
$(p_1,\dotsc,p_{d-1}) \in \Real^{d-1}$. Let $\zeta$ be some constant in
$(0,1)$ and let us define, for $\varepsilon \in (0,\zeta]$, the sets
$\tilde U_\varepsilon := P'_\varepsilon \times \left( 0, \varepsilon \gamma \right)$
and $U_\varepsilon := S(\tilde U_\varepsilon)$,
where $\gamma$ is the distance (with its spring index omitted) mentioned in
\cref{hyp:power-like} and
\begin{equation*}
P'_\varepsilon := \begin{cases}
\varepsilon (-\nicefrac{\pi}{2},\nicefrac{\pi}{2}) & \text{if } d = 2,\\
\varepsilon (-1,1) \times \varepsilon (-\nicefrac{\pi}{2},\nicefrac{\pi}{2}) & \text{if } d = 3
\end{cases}
\end{equation*}
and $S \colon \tilde U_\zeta \rightarrow U_\zeta$ is defined by the formula
\begin{equation*}
S(\ponf) = \begin{cases}
(\sqrt{b}-p_2) \left( \cos(p_1),\ \sin(p_1) \right) & \text{if } d = 2,\\
(\sqrt{b}-p_3) \left( \sqrt{1-p_1^2} \cos(p_2),\ \sqrt{1-p_1^2} \sin(p_2),\ p_1 \right) & \text{if } d = 3,
\end{cases}
\qquad \forall\,\ponf \in \tilde U_\zeta.
\end{equation*}
Note that if $0 < \varepsilon_1 < \varepsilon_2 \leq \zeta$ then $U_{\varepsilon_1}
\subset U_{\varepsilon_2} \subset D$ and $\tilde U_{\varepsilon_1} \subset
\tilde U_{\varepsilon_2}$. Having its domain and defining formula carefully
crafted for the purpose, the transformation $S$ turns out to be invertible,
orientation-preserving and $\CC^\infty(\overline{\scriptstyle \tilde
U_\zeta})$-regular. All of this is easy to see if one takes into account that
$S$ is a variant of the polar (resp.\ spherical) to Cartesian coordinate
transformation if $d = 2$ (resp.\ $d = 3$) with the radial variable being
measured from the boundary of $D$ and increasing towards its center. We denote the
inverse of $S$ by $T$; it, too, has uniformly bounded derivatives of all orders.

If $f$ is a function with domain $U_\zeta$ we will write $\tilde f := f \circ
S$. Then, $\varphi \in \CC^\infty(\overline{U_\zeta}) \iff \tilde\varphi \in
\CC^\infty(\overline{\scriptstyle \tilde U_\zeta})$. If $m$ is a positive
integer, part (b) of \cref{lem:partialDensity2} states that $\CC^\infty(\overline{D})$
is dense in $\HH^m_M(D)$; as, for any $\varepsilon \in (0,\zeta]$,
$U_\varepsilon$ is regular enough (a Lipschitz domain),
$\CC^\infty(\overline{U_\varepsilon})$ is exactly the set of restrictions to
$U_\varepsilon$ of $\CC^\infty(\overline{D})$ functions, whence
$\CC^\infty(\overline{U_\varepsilon})$ is dense in $\HH^m_M(U_\varepsilon)$ as
well. We also have from \cref{lem:weightedSpacesMap} in \cref{sec:distrResults},
that $f \in \HH^m_M(U_\varepsilon) \iff \tilde f
\in \HH^m_{\tilde M}(\tilde U_\varepsilon)$ and that
\begin{equation}\label{tildeBounds}
c_1(m) \norm[n]{\tilde f}_{\HH^m_{\tilde M}(\tilde U_\varepsilon)} \leq
\norm{f}_{\HH^m_M(U_\varepsilon)}
\leq c_2(m) \norm[n]{\tilde f}_{\HH^m_{\tilde M}(\tilde U_\varepsilon)}
\qquad \forall\,f\in\HH^m_M(U_\varepsilon),
\end{equation}
where the positive constants $c_1$ and $c_2$ depend on $m$ but can be chosen to
be independent of $\varepsilon$.

\begin{figure}
\centering\includegraphics[scale=1.0225]{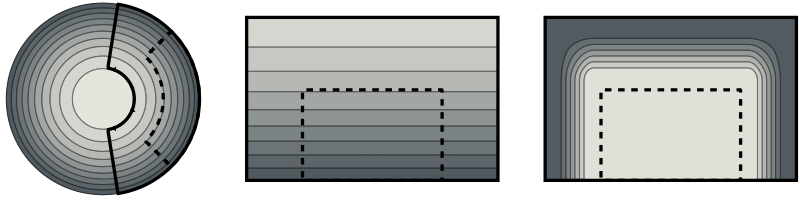} 
\caption{Illustration of the construction used in the proof of
\cref{lem:H2-to-H4}. From left to right: Contour plot of $M$ on $D$ with
$U_\zeta$ and $U_\varepsilon$ enclosed by the thick continuous and dashed
lines, respectively; contour plot of $\tilde M$ on $\tilde U_\zeta$ with
$\tilde U_\varepsilon$ enclosed by the thick dashed line; contour plot of an
admissible $\tilde \omega$ on $\tilde U_\zeta$ with $\tilde U_\varepsilon$
enclosed by the thick dashed line}
\label{rectangulize}
\end{figure}

As $\CC^\infty({\overline U_\zeta})$ is mapped by composition with $S$ to
$\CC^\infty(\overline{\scriptstyle \tilde U_\zeta})$ bijectively, it follows
that $\CC^\infty(\overline{\scriptstyle \tilde U_\zeta})$ is dense in
$\HH^1_{\tilde M}(\tilde U_\zeta)$. The rules of calculus and the density of
$\CC^\infty(\overline{U_\zeta})$ and $\CC^\infty(\overline{\scriptstyle \tilde
U_\zeta})$ functions in $\HH^1_M(U_\zeta)$ and $\HH^1_{\tilde M}(\tilde
U_\zeta)$, respectively, give the equalities
\begin{subequations}\label{mappedBF}
\begin{equation}\label{mappedBF-1}
\int_{U_\zeta} u\, v M = \int_{\tilde U_\zeta} \tilde u\, \tilde v \tilde M a
\qquad\text{and}\qquad
\int_{U_\zeta} \grad u \cdot \grad v M = \int_{\tilde U_\zeta} \grad \tilde u A \grad \tilde v^{\trans} \tilde M,
\end{equation}
where we have used the shorthand notations
\begin{equation}
a = \det(\grad S) \qquad\text{and}\qquad A = (\grad S)^{-1} (\grad S)^{-\trans} \det(\grad S).
\end{equation}
\end{subequations}
The first equality in \eqref{mappedBF-1} is valid for $u$ and $v$ in
$\LL^2_{M}(U_\zeta)$ and the second for $u$ and $v$ in $\HH^1_M(U_\zeta)$.
Direct calculations give
\begin{equation}\label{partsOfA}
A = A^{\trans}, \quad A_{k,d} = A_{d,k} = 0 \quad\forall\,k\in\ii{d-1}, \quad A_{d,d} = a, \quad\text{and}\quad \partial_k a = 0 \quad \forall\,k\in\ii{d-1},
\end{equation}
which we will exploit later. The need for the last equality in \eqref{partsOfA}
is the rationale behind taking of the sine of the polar angle instead of
the polar angle itself as the first argument of the transformation $S$ in the
case of $d = 3$. Additionally, by construction, $\tilde M$ is a function of the
radial variable $p_d$ only---namely, for all $\ponf \in \tilde U_\zeta$,
$\tilde M(\ponf) = h(p_d) p_d^\alpha$, where $h$ and $\alpha$, with the spring
index omitted, are those of \cref{hyp:power-like}.

Let us fix $\varepsilon \in (0,\zeta)$. For localization purposes we pick a
$\CC^\infty(U_\zeta)$ function $\omega$ with range $[0,1]$, identically $1$ in
$U_\varepsilon$, with support bounded away from $\partial U_\zeta \setminus
\partial D$ and such that $\partial_d \tilde\omega(\ponf) = 0$ if $\ponf$ is
within a finite distance $\gamma' > 0$ of $\partial \tilde U_\zeta \cap
T(\partial D)$ (such a function is readily constructed as $\omega = \tilde
\omega \circ T$ where $\tilde\omega (\ponf)= s(\ponf') t(p_d)$ and $s$ and $t$
are suitable mollified step functions). See \cref{rectangulize} for a
depiction of the construction so far.

Now, as every member of $\CIC(\tilde U_\zeta)$ can be put in the form $\varphi
\circ S$, where $\varphi \in \CIC(U_\zeta) \subset\HH^1_M(U_\zeta)$,
\eqref{partial-elliptic-problem-3} and the equalities in \eqref{mappedBF} imply
that $\tilde z$ obeys the distributional equation
\begin{equation}\label{mappedPDE}
-\div\left(\grad \tilde z \lsp A \tilde M \right) + \tilde z a \tilde M = \tilde g \tilde M a
\quad \text{in}\ \tilde U_\zeta.
\end{equation}

\underline{Part 2}: In this part we show that the relevant derivatives of
$\tilde z$ in directions tangential to the radial (i.e., the $d$-th) coordinate
possess additional regularity. The argument is a nontrivial
adaptation of Lemma 3.1 of \cite{French:1987}.

Let $k \in \ii{d-1}$. Then, using the Leibniz formula (cf.\ \cref{lem:LeibnizFormula})
and simple consequences of
\eqref{partsOfA} and the fact that $\tilde M(\ponf)$ depends on $p_d$ only, the distributional equation
\eqref{mappedPDE} conduces to
\begin{multline}\label{sevenTerms}
-\div\left(\grad(\tilde\omega \partial_k \tilde z) \lsp A \tilde M\right) + \tilde \omega\partial_k\tilde z \lsp a \tilde M\\
= \tilde\omega\partial_k\tilde g \lsp a \tilde M + \tilde \omega \grad\grad \tilde z : \partial_k A \, \tilde M + \tilde \omega\grad\tilde z \cdot \div(\partial_k A)\tilde M + \tilde\omega \grad\tilde M \cdot (\grad\tilde z \lsp \partial_k A)\\ - 2 \grad\tilde\omega \cdot (\grad(\partial_k \tilde z) \lsp A \tilde M) - \partial_k \tilde z \lsp \div(\grad \tilde \omega \lsp A) \tilde M - \partial_k \tilde z \lsp \grad \tilde M \cdot (\grad \tilde \omega \lsp A).
\end{multline}
We want to show that all the terms resulting above are (the linear combination of)
members of the space $a \tilde M \LL^2_{\tilde M}(\tilde U_\zeta) = \tilde M
\LL^2_{\tilde M}(\tilde U_\zeta)$. Of the resulting seven terms above, the
first three, the fifth and the sixth pose no problem, thanks to the regularity
of $g$ and \cref{lem:ellipReg}. The fourth vanishes after making full use of
the equalities in \eqref{partsOfA} and the sole dependence of $\tilde M$ on the
radial variable---this is what the fourth equation in \eqref{partsOfA} is truly
for. The membership of the seventh term in $\tilde M \LL^2_{\tilde M}(\tilde
U_\zeta)$ stems from observing that
\begin{multline*}
\norm{\partial_k \tilde z \lsp \grad \tilde M \cdot (\grad\tilde\omega \lsp A ) / \tilde M}_{\LL^2_{\tilde M}(\tilde U_\zeta)}
= \norm{\partial_k \tilde z \lsp \partial_d \tilde M \lsp \partial_d \tilde\omega \lsp a / \tilde M}_{\LL^2_{\tilde M}(P'_{\zeta} \times (\gamma',\zeta\gamma))}\\
\leq \sup_{P'_{\zeta}\times(\gamma', \zeta\gamma)} \left( \partial_d \tilde M \lsp \partial_d\tilde\omega \lsp a / \tilde M \right) \norm{\partial_k \tilde z}_{\LL^2_{\tilde M}(\tilde U_\zeta)} < \infty.
\end{multline*}

Let $\hat f$, given some function $f$ defined on $\tilde U_\zeta$, denote $f
\circ T = f \circ S^{-1}$. Also, let $f_{(k)}$ denote the ratio of the
right-hand side of \eqref{sevenTerms} and $\tilde M a$. Then, \eqref{sevenTerms}
and the identities in \eqref{mappedBF} give
\begin{equation}\label{backTransform}
-\div(M \grad \widehat{\tilde\omega \partial_k\tilde z}) + \widehat{\tilde\omega \partial_k\tilde z} M = \widehat{f_{(k)}}
\end{equation}
in $U_\zeta$, with $\widehat{f_{(k)}} \in \LL^2_M(U_\zeta)$. As the support of
$\tilde \omega$ is bounded away from $\partial\tilde U_\zeta \setminus
T(\partial D)$, we can extend $\widehat{\tilde\omega \partial_k\tilde z}$ and
$\widehat{f_{(k)}}$ to the whole of $D$ by zero while still satisfying
\eqref{backTransform}. Then, \cref{lem:ellipReg} ensures that the extension
of $\widehat{\tilde\omega \partial_k\tilde z}$ belongs to $\HH^2_M(D)$. It
follows that $\tilde\omega \partial_k\tilde z$ belongs to $\HH^2_{\tilde
M}(U_\zeta)$ and, consequently, $\partial_k\tilde z \in \HH^2_{\tilde M}(\tilde
U_\varepsilon)$.

This procedure can be iterated. Within $U_\varepsilon$ the identity \eqref{sevenTerms}
particularizes to
\begin{equation*}
-\div\left(\grad(\partial_k \tilde z) \lsp A \tilde M \right) + \partial_k \tilde z \lsp a \tilde M
= \partial_k g \lsp a \tilde M + \grad\grad \tilde z : \partial_k A \,\tilde M + \grad\tilde z \cdot \div(\partial_k A) \tilde M.
\end{equation*}
Let $g_{(k)} := \partial_k \tilde g + \grad\grad \tilde z : \partial_k A \lsp /
a + \grad \tilde z \cdot \div(\partial_k A)/a$ and let us redefine $\tilde
\omega$ so that the role of $\tilde U_\zeta$ is now taken up by $\tilde
U_\varepsilon$ and the role of the latter is taken up by $\tilde U_\delta$,
where $\delta$ is some fixed number in $(0,\varepsilon)$. Thus, we can obtain
an analogue of \eqref{sevenTerms} for $\tilde\omega\partial_{l,k}\tilde z$,
where $l,k  \in \ii{d-1}$:
\begin{multline*}
-\div(\grad(\tilde\omega \partial_{l,k}\tilde z) \lsp A \tilde M) + \tilde\omega \partial_{l,k}\tilde z \lsp a \tilde M\\
= \tilde\omega \partial_l g_{(k)} \lsp a \tilde M + \tilde\omega \grad\grad\partial_k \tilde z : \partial_l A \,\tilde M + \tilde\omega \grad\partial_k\tilde z \cdot \div(\partial_l A)\tilde M + \tilde\omega \grad \tilde M \cdot (\grad(\partial_k \tilde z) \lsp \partial_l A)\\ - 2\grad\tilde\omega \cdot (\grad(\partial_{l,k}\tilde z) \lsp A \tilde M) - \partial_{l,k}\tilde z \lsp \div(\grad\tilde\omega \lsp A)\tilde M - \partial_{l,k}\tilde z \lsp \grad M \cdot (\grad\tilde\omega \lsp A).
\end{multline*}
Analogously to the study of the first-order tangential derivatives we need all
seven terms on the right-hand side of the above equation to belong to
$\tilde M \LL^2_{\tilde M}(\tilde
U_\varepsilon)$ now. As, at this stage, we know that $\partial_k \tilde z \in
\HH^2_{\tilde M}(\tilde U_\varepsilon)$, the second, the third, the fifth and
the sixth term above pose no difficulties. The fourth term and the seventh term
can be dealt with just as their counterparts in \eqref{sevenTerms}. When
it comes to the first term, it is enough to show that $\partial_l g_{(k)} \in
\LL^2_{\tilde M}(\tilde U_\varepsilon)$. Now,
\begin{equation*}
\partial_l g_{(k)} = \partial_{l,k} \tilde g + \grad\grad \tilde z : \partial_l (\partial_k A \lsp / a) + \grad\tilde z \cdot \partial_l(\div(\partial_k A) / a) + \partial_l \grad\grad \tilde z : (\partial_k A \lsp / a) + \partial_l \grad \tilde z \cdot \div(\partial_k A) / a.
\end{equation*}
The first three terms above are clearly in $\LL^2_{\tilde M}(\tilde
U_\varepsilon)$---the first because of our hypotheses on $g$; so is the fifth,
for the second derivatives of $\tilde z$ have the desired integrability. Finally,
$\partial_l \grad\grad \tilde z \in \LL^2_{\tilde M}(\tilde
U_\varepsilon)$ because $\partial_l \tilde z \in \HH^2_{\tilde M}(\tilde
U_\varepsilon)$, as shown above. Proceeding with the argument one finds, after
localization, that $\partial_{k,l} \tilde z \in \HH^2_{\tilde M}(\tilde
U_\delta)$. We mention in passing that by closely following the arguments above
the linear operators $g \in \HH^2_M(D) \mapsto \partial_l \tilde z \in
\HH^2_{\tilde M}(\tilde U_\varepsilon)$ and $g \in \HH^2_M(D) \mapsto
\partial_{k,l}\tilde z \in \HH^2_{\tilde M}(\tilde U_\delta)$ can be seen to be continuous;
i.e., bounded.

\underline{Part 3}: In this part we show the additional regularity of some
derivatives of $\tilde z$ that involve the radial direction.

Expanding and rearranging the distributional equation \eqref{mappedPDE}, noting
the sole dependence of $\tilde M$ on the last component of its
argument and the properties of $A$ given by \eqref{partsOfA} we get
\begin{equation}\label{definition-of-f}
\begin{split}
- \frac{1}{\tilde M a} \partial_d (\partial_d \tilde z \lsp a \tilde M)
& = \tilde g - \tilde z + \frac{1}{a} \sum_{k=1}^{d-1}\sum_{j=1}^{d-1} (\partial_{j,k} \tilde z \lsp A_{j,k} + \partial_j \tilde z \lsp \partial_k A_{j,k}) =: f
\end{split}
\end{equation}
in $\tilde U_\delta$. From the previous part of the proof and our assumptions
on $g$ we have that $f \in \HH^2_{\tilde M}(\tilde U_\delta)$. Multiplying
 \eqref{definition-of-f} by $\tilde M a$ and integrating with respect to the $d$-th
variable we obtain
\begin{equation}\label{with-limit}
(\partial_d \tilde z \lsp a \tilde M)[\ponf',p_d] - \lim_{s \to 0_+} (\partial_d \tilde z \lsp a \tilde M)[\ponf',s]
= \int_0^{p_d} (f a \tilde M) [\ponf',s] \dd s
\end{equation}
for almost every $\ponf'$ in $P'_{\delta}$. We note in passing that in
this part of the proof we reserve square brackets for arguments of functions.
Our first task is to show that the limit on the left-hand side of \eqref{with-limit} vanishes.
To this end, we first observe that, for $p_d, s \in (0,\delta\gamma)$,
\[ p_d^{\alpha/2} \partial_d \tilde z[\ponf',p_d] = s^{\alpha/2}\partial_d\tilde z[\ponf',s] + \int_s^{p_d} \diff{}{\sigma}\left( \sigma^{\alpha/2} \partial_d\tilde z[\ponf',\sigma] \right) \dd\sigma, \]
whence
\begin{equation*}
p_d^{\alpha/2} \abs{\partial_d \tilde z[\ponf',p_d]}
\leq s^{\alpha/2}\abs{\partial_d\tilde z[\ponf',s]}
+ \abs{\int_s^{p_d} \frac{\alpha}{2}\sigma^{\alpha/2-1} \partial_d \tilde z[\ponf',\sigma] \dd\sigma}
+ \abs{\int_s^{p_d} \sigma^{\alpha/2} \partial_{d,d}\tilde z[\ponf',\sigma] \dd\sigma}.
\end{equation*}
Furthermore,
\begin{setlength}{\multlinegap}{0pt}
\begin{multline*}
p_d^\alpha \abs{\partial_d \tilde z[\ponf',p_d]}^2\\
\begin{aligned}
& \leq 3 s^\alpha \abs{\partial_d\tilde z[\ponf',s]}^2
+ \frac{3\alpha^2}{4} \abs{\int_s^{p_d} \sigma^{\alpha/2-1} \partial_d \tilde z[\ponf',\sigma] \dd\sigma}^2
+ 3 \abs{\int_s^{p_d} \sigma^{\alpha/2} \partial_{d,d}\tilde z[\ponf',\sigma] \dd\sigma}^2\\
& \leq 3 s^\alpha \abs{\partial_d\tilde z[\ponf',s]}^2
+ \frac{3\alpha^2}{4} \abs{p_d-s} \int_s^{p_d}\sigma^{\alpha-2}\abs{\partial_d\tilde z[\ponf',\sigma]}^2\dd\sigma
+ 3 \abs{p_d-s} \int_s^{p_d} \sigma^\alpha \abs{\partial_{d,d}\tilde z[\ponf',\sigma]}^2\dd\sigma\\
& \leq 3 s^\alpha \abs{\partial_d\tilde z[\ponf',s]}^2
+ \frac{3\alpha^2}{4} \delta\gamma \int_0^{\delta\gamma} \sigma^{\alpha-2}\abs{\partial_d\tilde z[\ponf',\sigma]}^2\dd\sigma
+ 3 \delta\gamma \int_0^{\delta\gamma} \sigma^\alpha \abs{\partial_{d,d}\tilde z[\ponf',\sigma]}^2\dd\sigma.
\end{aligned}
\end{multline*}
\end{setlength}
Integrating this chain of inequalities with respect to $s$ from $0$ to $\delta\gamma$
and applying the Hardy inequality \eqref{anotherHardy} stated in \cref{lem:Hardy} we
obtain
\begin{equation*}
\begin{split}
& \delta\gamma\, p_d^\alpha \abs{\partial_d \tilde z[\ponf',p_d]}^2\\
& \leq 3 \int_0^{\delta\gamma}\! s^\alpha \abs{\partial_d \tilde z[\ponf',s]}^2 \dd s
+ \frac{3\alpha^2}{4} (\delta\gamma)^2 \int_0^{\delta\gamma} \sigma^{\alpha-2} \abs{\partial_d \tilde z[\ponf',\sigma]}^2 \dd \sigma
+ 3 (\delta\gamma)^2\int_0^{\delta\gamma}\! \sigma^\alpha \abs{\partial_{d,d} \tilde z[\ponf',\sigma]}^2 \dd\sigma\\
& \leq 3 \int_0^{\delta\gamma} s^\alpha \abs{\partial_d \tilde z[\ponf',s]}^2 \dd s
+ \frac{3\alpha^2 C_{\delta\gamma,\alpha}}{4} (\delta\gamma)^2 \int_0^{\delta\gamma} \sigma^\alpha \abs{\partial_d \tilde z[\ponf',\sigma]}^2 \dd \sigma\\
& \qquad \qquad + 3 (\delta\gamma)^2 \left(\frac{\alpha^2 C_{\delta\gamma,\alpha}}{4} + 1\right) \int_0^{\delta\gamma} \sigma^\alpha \abs{\partial_{d,d} \tilde z[\ponf',\sigma]}^2 \dd\sigma.
\end{split}
\end{equation*}
Dividing by $\delta\gamma$, integrating with respect to $\ponf'$ in
$P'_{\delta}$, using the bilateral boundedness of $h$ and $a$ by positive
constants, and consolidating the constants, we get the trace-inequality-like bound
\begin{equation}\label{trace-like}
\int_{P'_{\delta}} p_d^\alpha \abs{\partial_d \tilde z[\ponf',p_d]}^2 \dd\ponf'
\leq \const \norm{ \partial_d \tilde z }_{\HH^1_{\tilde M}(\tilde U_\delta)}^2.
\end{equation}
Thus,
\begin{equation*}
\int_{P'_{\delta}} \abs{ (\partial_d\tilde z \lsp a \tilde M)[\ponf',p_d] } \dd \ponf'
\leq \const\, h[p_d]^{1/2} p_d^{\alpha/2}\left(\int_{P'_{\delta}} (\abs{\partial_d \tilde z}^2 \tilde M a)[\ponf] \dd \ponf' \right)^{1/2} \to 0 \quad \text{as } p_d \to 0_+,
\end{equation*}
which implies the vanishing of the limit in \eqref{with-limit}.

Let us define $w\colon \tilde U_\delta \rightarrow \Real$ by
\begin{equation}\label{definition-of-w}
w[\ponf] := \frac{\partial_d (\tilde M a)[\ponf]}{(\tilde M a)[\ponf]} \partial_d \tilde z [\ponf]
= \frac{ \left( (h a)[p_d] p_d^\alpha \right)' }{ (h a)[p_d]^2 p_d^{2\alpha} } \int_0^{p_d} (f a \tilde M)[\ponf',s] \dd s,
\end{equation}
where we have taken the liberty of treating $a$ as an univariate function,
which it is in the algebraic sense. The equality is valid for almost every
$\ponf \in \tilde U_\delta$. Note that $w$ is a member of $\LL^2_{\tilde
M}(\tilde U_\delta)$ because $\grad M \cdot \grad z / M \in \LL^2_M(U_\delta)$;
this, in turn, is a consequence of \cref{lem:ellipReg}. We intend to show
that $w \in \HH^2_{\tilde M}(U_\delta)$. Let $\seq{f_n}{n \in \natural}$ be a
sequence of $\CC^\infty(\overline{\scriptstyle \tilde U_\delta})$ functions
converging to $f$ in $\HH^2_{\tilde M}(\tilde U_\delta)$ (its existence
having been discussed in Part 1) and let
\begin{equation}\label{definition-of-wn}
\begin{split}
&w_n[\ponf]  := \frac{ \left( (h a)[p_d] p_d^\alpha \right)' }{ (h a)[p_d]^2 p_d^{2\alpha} } \int_0^{p_d} (f_n a \tilde M)[\ponf',s] \dd s\\
& = \int_0^{p_d} \frac{ (h a)'[p_d] p_d + \alpha (h a)[p_d]}{ (h a)[p_d]^2 } \smartxi^\alpha \frac{(h a f_n)[\ponf',s]}{p_d} \dd s = \mathrm{aux}[p_d] \int_0^1 \xi^\alpha (h a f_n)[\ponf',p_d \xi] \dd \xi
\end{split}
\end{equation}
Here we have written $\mathrm{aux}[p_d]$ in place of the first fraction in the
second integral and denoted the function $\ponf \in \tilde U_\zeta \mapsto
h(p_d) \in \Real$ by $h$ as well. The second equality comes via the change of
variable $\xi = s/p_d$. As the function $h$ and the determinant $a$ have uniform
$\CC^3$ and $\CC^\infty$ regularity in $U_\delta$, the function $\mathrm{aux}
\in \CC^2(\overline{\scriptstyle \tilde U_\delta})$ and $w_n$ is twice
continuously differentiable in the $d$-th direction.

Differentiating the last integral representation of $w_n$ with respect to its
$d$-th variable twice and then reversing the change of variable we obtain
\begin{equation*}
\begin{split}
\partial_{d,d} w_n[\ponf]
& = \sum_{k=0}^{2} \binom{2}{k} \diff{^{2-k} \mathrm{aux}}{p_d^{2-k}}[p_d] \int_0^1 \partial_d^k(h a f_n)[\ponf', p_d \xi] \lsp \xi^{\alpha+k} \dd \xi\\
& = \sum_{k=0}^2 \binom{2}{k} \diff{^{2-k} \mathrm{aux}}{p_d^{2-k}}[p_d]\int_0^{p_d} \partial_d^k(h a f_n)[\ponf',s] \lsp \smartxi^{\alpha+k} \frac{1}{p_d} \dd s,
\end{split}
\end{equation*}
whence, as $s/p_d \in (0,1)$ if $s \in (0,p_d)$,
\begin{equation*}
\begin{split}
p_d^{\alpha/2} \abs{\partial_{d,d} w_n[\ponf]}
& \leq \frac{1}{p_d} \int_0^{p_d} \left( \sum_{k=0}^2 \binom{2}{k}\abs{\diff{^{2-k} \mathrm{aux}}{p_d^{2-k}}[p_d] \lsp \partial_d^k (h a f_n)[\ponf',s] } \smartxi^{\alpha/2 + k} \right) s^{\alpha/2} \dd s\\
& \leq \frac{1}{p_d} \int_0^{p_d} \left( \sum_{k=0}^2 \binom{2}{k}\abs{\diff{^{2-k} \mathrm{aux}}{p_d^{2-k}}[p_d] \lsp \partial_d^k (h a f_n)[\ponf',s] } \right) s^{\alpha/2} \dd s\\
& \leq \frac{\const}{p_d} \int_0^{p_d} \left( \abs{(h a f_n)}^2 + \abs{\partial_d(h a f_n)}^2 + \abs{\partial_{d,d}(h a f_n)}^2 \right)^{1/2}[\ponf',s] s^{\alpha/2} \dd s
\end{split}
\end{equation*}
for some $\const* > 0$ independent of $\ponf = (\ponf',p_d)$. We square the
resulting inequality, integrate it with respect to $p_d$ from $0$ to
$\delta\gamma$, use the Hardy inequality \eqref{standardHardy} in
\cref{lem:Hardy} and note yet again the bilateral boundedness of
$h$ and $a$ by positive constants to obtain
%
\begin{multline*}
\int_0^{\delta\gamma} p_d^\alpha (h a)[p_d] \abs{\partial_{d,d} w_n[\ponf]}^2 \dd p_d\\
\leq \const \int_0^{\delta\gamma} \left( \abs{(h a f_n)}^2 + \abs{\partial_d(h a f_n)}^2 + \abs{\partial_{d,d}(h a f_n)}^2 \right) \negmedspace [\ponf]\, s^\alpha (h a)[p_d] \dd p_d,
\end{multline*}
where $\const*$ is still independent of $\ponf' \in P'_{\delta}$. Integrating
this with respect to $\ponf' \in P'_{\delta}$, using the regularity of $h$ and
$a$ and taking into account that $(\tilde M a)[\ponf] = (h a)[p_d] p_d^\alpha$
for all $\ponf \in \tilde U_\delta$ one gets
\[ \norm{ \partial_{d,d} w_n }_{\LL^2_{\tilde M}(\tilde U_\delta)} \leq \const \norm{ f_n }_{\HH^2_{\tilde M}(\tilde U_\delta)}. \]
This argument can be carried over to all derivatives of order less than or equal
to two of $w_n$ (including zeroth order derivatives of $w_n$, meaning $w_n$ itself). The result is
\[ \norm{ w_n }_{\HH^2_{\tilde M}(\tilde U_\delta)} \leq \const \norm{f_n}_{\HH^2_{\tilde M}(\tilde U_\delta)}. \]

As $\HH^2_{\tilde M}(\tilde U_\delta)$ is a Hilbert space, there exists a
subsequence $\seq{ w_{\phi(n)} }{n \geq 1}$ with a weak limit $w^* \in
\HH^2_{\tilde M}(\tilde U_\delta)$. By the continuity of the injection of
$\HH^2_{\tilde M}(\tilde U_\delta)$ into $\LL^2_{\tilde M}(\tilde U_\delta)$,
$w^*$ is also the weak limit of the $w_{\phi(n)}$ in $\LL^2_{\tilde M}(\tilde
U_\delta)$.

Now, given any $\chi \in \LL^2_{\tilde M}(\tilde U_\delta)$,
\begin{setlength}{\multlinegap}{0pt}
\begin{multline*}
\int_{\tilde U_\delta} \left( \frac{\mathrm{aux}[p_d]}{p_d} \int_0^{p_d} \smartxi^\alpha (h a \chi)[\ponf',s] \dd s \right)^2 \tilde M[\ponf] \dd \ponf\\
\begin{aligned}
& = \int_{\tilde U_\delta} \frac{\mathrm{aux}[p_d]^2}{p_d^2} \left( \int_0^{p_d} \smartxi^{\alpha/2} s^{\alpha/2} (h a \chi)[\ponf',s]\dd s \right)^2 h[p_d] \dd \ponf\\
& \leq \const \int_{P'_\delta} \int_0^{\delta\gamma} \frac{1}{p_d^2} \left( \int_0^{p_d} s^{\alpha/2} (h a \chi)[\ponf',s]\dd s \right)^2 \dd p_d \dd \ponf'\\
& \leq \const \int_{P'_\delta} \int_0^{\delta\gamma} s^\alpha \abs{(h a \chi)[\ponf',s]}^2 \dd s \dd \ponf'\\
& \leq \const \norm{\chi}_{\LL^2_{\tilde M}(\tilde U_\delta)}^2.
\end{aligned}
\end{multline*}
\end{setlength}%
Hence, the operation that defines $w$ (resp.\ $w_n$) in terms of $f$ (resp.\
$f_n$) in \eqref{definition-of-w} (resp.\ \eqref{definition-of-wn}) is a
bounded map from $\LL^2_{\tilde M}(\tilde U_\delta)$ to itself. Therefore,
$\lim_{n\to\infty} f_n = f$ in $\LL^2_{\tilde M}(\tilde U_\delta)$ implies
$\lim_{n\to\infty} w_n = w$ in the same space. Thus, $w$ and the weak limit
$w^*$ have to be the same measurable function and so $w \in \HH^2_{\tilde
M}(\tilde U_\delta)$. We get the bound $\norm{w}_{\HH^2_{\tilde M}(\tilde
U_\delta)} \leq \liminf_{n \to \infty} \norm[n]{w_{\phi(n)}}_{\HH^2_{\tilde
M}(\tilde U_\delta)} \leq \shiftedconst{-3}
\norm[n]{f_{\phi(n)}}_{\HH^2_{\tilde M}(\tilde U_\delta)}$. As (with no loss of
generality) we can assume that the $f_n$ are scaled so that their
$\HH^2_{\tilde M}(\tilde U_\delta)$ norm is identically equal to the same norm
of $f$, it follows that
\begin{equation*}
\norm{w}_{\HH^2_{\tilde M}(\tilde U_\delta)} \leq \shiftedconst{-3} \norm{f}_{\HH^2_{\tilde M}(\tilde U_\delta)}.
\end{equation*}

From \eqref{definition-of-f} and \eqref{definition-of-w},
\[ -\partial_{d,d} \tilde z = f + \frac{\partial_d(\tilde M a)}{\tilde M a} \partial_d \tilde z = f + w, \]
whence $\norm{\partial_{d,d} \tilde z}_{\HH^2_{\tilde M}(\tilde U_\delta)} \leq
(1+\shiftedconst{-3}) \norm{f}_{\HH^2_{\tilde M}(\tilde U_\delta)} \leq \const \norm{\tilde
g}_{\HH^2_{\tilde M}(\tilde U_\delta)}$. We know from the previous part that
all second derivatives of $\tilde z$ that do not involve the $d$-th
direction have $\HH^2_{\tilde M}(\tilde U_\delta)$ norms bounded by the
$\HH^2_{\tilde M}(\tilde U_\delta)$ norm of $\tilde g$. This and the
corresponding result for $\partial_{d,d}\tilde z$ is enough to be able to bound
all derivatives of $\tilde z$ of order less than or equal to four, and thus
deduce that
\[ \norm{\tilde z}_{\HH^4_{\tilde M}(\tilde U_\delta)} \leq \const \norm{g}_{\HH^2_{M}(D)} \]
or, equivalently in the light of \eqref{tildeBounds}, that
\begin{equation}\label{localizedH4}
\norm{z}_{\HH^4_M(U_\delta)} \leq \const \norm{g}_{\HH^2_M(U_\delta)}.
\end{equation}

\underline{Part 4}: By modifying the transformation $S$ one can get a
localized bound of the form \eqref{localizedH4} for any origin-centered
rotation of $U_\delta$. It follows that \eqref{localizedH4} remains valid
(with some other constant $\const*$) if we replace $U_\delta$ by the annulus/spherical
shell $\{ \ponf \in D \colon \abs{\ponf} > \sqrt{b}-\delta\gamma \}$.

Let $D_0$ be the ball $B(0,\sqrt{b}-\delta\gamma/2) \compEmb D$. As
$\CIC(D_0) \subset \CIC(D)$, we have that
\[ \langle z, \varphi \rangle_{\HH^1_{M}(D_0)} = \langle g, \varphi \rangle_{\LL^2_{M}(D_0)} \qquad \forall\, \varphi \in \CIC(D_0). \]
The existence of a positive infimum of $M$ in $D_0$ implies that $z$ is the
weak solution to a regular (i.e., uniformly) elliptic problem in $D_0$ with
$\HH^2(D_0)$ right-hand side. It follows, via the $\CC^{2,1}(D_0)$ regularity
of $M$ (see, e.g., \cite[Theorem 8.10]{GT}), that for some $\const > 0$,
\[ \norm{z}_{\HH^{4}_{M}(D_0')} \leq \const* \norm{g}_{\HH^2_{M}(D_0)} \leq \const* \norm{g}_{\HH^2_{M}(D)}, \]
with $D_0' := B(0,\sqrt{n}-3\delta\gamma/4) \compEmb D_0$.

Combining this last estimate with the result in the annulus/spherical shell mentioned above (which in union with $D_0'$ covers $D$), we obtain
the desired global bound.
\end{proof}

The next lemma is an almost trivial corollary of
\cref{lem:H2-to-H4}, yet it is a true iterate of \cref{lem:ellipReg} in
the sense that the hypothesis on the right-hand side function is the thesis on
the solution in \cref{lem:ellipReg}. This makes it suitable for the
arguments that will be used in the proof of \cref{lem:partialEquiv}.

\begin{lemma}\label[lemma]{lem:iteratedEllipReg}
Let $i \in \ii{N}$ and suppose that $g \in \HH^2_{M_i}(D_i)$ and that $M_i^{-1}\div(M_i\grad g)
\in \LTMi$. Then, there exists a constant $C_i > 0$, independent of $g_i$, such that the solution $z \in \HOMi$ of
\[ \langle z, \varphi \rangle_{\HOMi} = \langle g, \varphi\rangle_{\LTMi} \qquad\forall\,\varphi\in\HOMi \]
obeys the regularity estimate
\begin{multline*}
\norm{ z }_{\HH^4_{M_i}(D_i)}
+ \norm{ \frac{1}{M_i}\div(M_i \grad z) }_{\HH^2_{M_i}(D_i)}
+ \norm{ \frac{1}{M_i}\div\left(M_i \grad\left[\frac{1}{M_i}\div(M_i\grad z)\right]\right) }_{\LTMi}\\
\leq C_i \left( \norm{ g }_{\HH^2_{M_i}(D_i)} + \norm{\frac{1}{M_i} \div(M_i\grad g) }_{\LTMi} \right).
\end{multline*}
\end{lemma}
\begin{proof}
This follows directly from \cref{lem:H2-to-H4} on noting that
$M_i^{-1}\div(M_i\grad z) = g - z$ in the distributional sense first, and then in the
sense of measurable functions.
\end{proof}

\begin{lemma}\label[lemma]{lem:partialEquiv}
Let $i \in \ii{N}$. The following statements of equivalence hold:
\begin{multline}\label{partialFastDecay2}
\tau \in \HH^2_{M_i}(D_i) \quad \text{and} \quad \frac{1}{M_i} \div(M_i \grad \tau) \in \LTMi\\
\iff \tau \in \LTMi \quad \text{and} \quad \sum_{n=1}^\infty (\lambda_n\psp{i})^2 \langle \tau, e\psp{i}_n \rangle_{\LTMi}^2 < \infty;
\end{multline}
and
\begin{multline}\label{partialFastDecay4}
\tau \in \HH^4_{M_i}(D_i),\ \frac{1}{M_i} \div(M_i \grad \tau) \in \HH^2_{M_i}(D_i) \quad \text{and} \quad \frac{1}{M_i}\div\left(M_i\grad \left[\frac{1}{M_i}\div(M_i\grad \tau)\right]\right) \in \LTMi\\
\iff \tau \in \LTMi \quad \text{and} \quad \sum_{n=1}^\infty (\lambda_n\psp{i})^4 \langle \tau, e\psp{i}_n \rangle_{\LTMi}^2 < \infty.
\end{multline}
\end{lemma}
\begin{proof}
We will omit the spring index when proving \eqref{partialFastDecay2} and
\eqref{partialFastDecay4}. We start by denoting by $L$ the operator that
associates each $\varphi \in \mathrm{W}^{2,1}_{\loc}(D)$ with the distribution
$M^{-1}\div(M\grad\varphi)$ (this is a well-defined distribution because of the
regularity of $\varphi$ and $M$; cf.\ \cref{lem:LeibnizFormula} in \cref{sec:distrResults}).
We also write $\hat L := -L + I$ where $I$ is the
operator that associates each distribution with itself.
Let us define the Hilbert spaces (and associated norms)
\begin{subequations}\label{partial-Ht2}
\begin{gather}
\tilde\HH^2_M(D) := \left\{ \varphi \in \HH^2_M(D) \colon L(\varphi) \in \LL^2_M(D) \right\},\\
\norm{\varphi}_{\tilde\HH^2_M(D)}^2 := \norm{\varphi}_{\HH^2_M(D)}^2 + \norm{L(\varphi)}_{\LL^2_M(D)}^2
\end{gather}
\end{subequations}
and
\begin{subequations}\label{partial-Ht4}
\begin{gather}\tilde\HH^4_M(D) := \left\{ \varphi \in \HH^4_M(D)\colon L(\varphi) \in \tilde\HH^2_M(D) \right\},\\
\norm{\varphi}_{\tilde\HH^4_M(D)}^2 := \norm{\varphi}_{\HH^4_M(D)}^2 + \norm{L(\varphi)}_{\tilde\HH^2_M(D)}^2.
\end{gather}
\end{subequations}

Because of the definition of $\tilde\HH^2_M(D)$, $\hat L \colon \HH^2_M(D) \rightarrow \LL^2_M(D)$ is a bounded
linear operator. As for every $\varphi \in \LL^2_M(D)$
the solution $z \in \HH^1_M(D)$ to
\[ \langle z, \psi \rangle_{\HH^1_M(D)} = \langle f, \psi\rangle_{\LL^2_M(D)} \qquad\forall\,\psi \in \HH^1_M(D) \]
exists and, thanks to \cref{lem:ellipReg}, is bounded in $\tilde\HH^2_M(D)$,
$\hat L^{-1}\colon\LL^2_M(D) \rightarrow \tilde\HH^2_M(D)$ is well-defined and
bounded. Similarly, by the definition of $\tilde\HH^4_M(D)$ and
\cref{lem:iteratedEllipReg}, $\hat L^2 \colon \tilde\HH^4_M(D)
\rightarrow \LL^2_M(D)$ is a bounded linear operator with a bounded inverse.

Let $\tau \in \tilde\HH^2_M(D)$; i.e., $\tau$ complies with the left-hand side of
\eqref{partialFastDecay2}. It then follows that $f_\tau := -L(\tau) + \tau \in
\LL^2_M(D)$ and Parseval's identity thus yields
\begin{equation*}
\infty > \norm{f_\tau}_{\LL^2_M(D)}^2
= \sum_{n \geq 1} \langle f_\tau, e_n\rangle_{\LL^2_M(D)}^2
= \sum_{n \geq 1} \langle \tau, e_n\rangle_{\HH^1_M(D)}^2
= \sum_{n \geq 1} \lambda_n^2 \langle \tau, e_n\rangle_{\LL^2_M(D)}^2,
\end{equation*}
where $\langle f_\tau, e_n \rangle_{\LL^2_M(D)} = \langle \tau,
e_n\rangle_{\HH^1_M(D)}$ follows by the density of $\CIC(D)$ in
$\HH^1_M(D)$.

To prove the converse, note that the eigenfunctions $e_n$ of \eqref{partial-ev}
are solutions of $e_n = \hat L^{-1}(\lambda_n e_n)$, whence
$\norm{e_n}_{\tilde\HH^2_M(D)} \leq C \norm{\lambda_n e_n}_{\LTM} = C
\lambda_n$. Consequently, the partial sums
\[ \tau_k := \sum_{n=1}^k \langle \tau, e_n\rangle_{\LL^2_M(D)} e_n \]
are members of $\tilde\HH^2_M(D)$. Then, if $k\leq l$, the
$\LL^2_M(D)$-orthonormality of the $e_n$ yields that
\begin{equation*}
\norm{\hat L(\tau_l) - \hat L(\tau_k)}_{\LL^2_M(D)}^2
= \norm{\sum_{n=k+1}^l \langle \tau, e_n\rangle_{\LL^2_M(D)} \hat L(e_n)}_{\LL^2_M(D)}^2
= \sum_{n=k+1}^l \lambda_n^2 \langle \tau, e_n\rangle_{\LL^2_M(D)}^2.
\end{equation*}
As the sum $\sum_{n\geq 1}\lambda_n^2 \langle \tau, e_n\rangle_{\LL^2_M(D)}^2$
is assumed to converge, the sequence $\seq{ \hat L(\tau_k) }{k \geq 1}$ is a
Cauchy sequence in $\LL^2_M(D)$ and hence it converges to some $f^* \in
\LL^2_M(D)$. The continuity of $\hat L^{-1}$ implies that the sequence $\seq{
\tau_k }{k \geq 1}$ converges in $\tilde\HH^2_M(D)$ to $\hat L^{-1}(f^*)$. The
same sequence converges in $\LL^2_M(D)$ to $\tau$. The continuity of the
injection of $\HH^2_M(D)$ into $\LL^2_M(D)$ then implies that $\tau = \hat
L^{-1}(f^*) \in \tilde\HH^2_M(D)$. This completes the proof of
\eqref{partialFastDecay2}.

Let us suppose now that $\tau$ in $\tilde\HH^4_M(D)$; i.e., $\tau$ complies with the
left-hand side of \eqref{partialFastDecay4}. It follows that $f_\tau :=
-L(\tau) + \tau \in \tilde\HH^2_M(D)$ and $g_\tau: = -L(f_\tau) + f_\tau \in
\LL^2_M(D)$. Parseval's identity gives
%
\begin{multline*}
\infty > \norm{g_\tau}_{\LL^2_M(D)}^2
= \sum_{n \geq 1} \langle g_\tau, e_n \rangle_{\LL^2_M(D)}^2
= \sum_{n \geq 1} \langle f_\tau, e_n \rangle_{\HH^1_M(D)}^2\\
= \sum_{n \geq 1} \lambda_n^2 \langle f_\tau, e_n \rangle_{\LL^2_M(D)}^2
= \sum_{n \geq 1} \lambda_n^4 \langle \tau, e_n \rangle_{\LL^2_M(D)}^2,
\end{multline*}
%
where the second equality follows, similarly as above, by the density of
$\CIC(D)$ in $\HH^1_M(D)$ thanks to the boosted regularity of $f_\tau$. The latter
also allows the use of \eqref{partial-ev} to obtain the third equality.

To prove the converse we first note that each $e_n$ is a solution of $e_n =
\hat L^{-2}(\lambda_n^2 e_n)$, whence $\norm{e_n}_{\tilde\HH^4_M(D)} \leq C
\norm{\lambda_n e_n}_{\LL^2_M(D)} = C \lambda_n$. Thus, the partial sums
$\tau_k$ are members of $\tilde\HH^4_M(D)$; hence, if $k \leq l$,
\begin{equation*}
\norm{ \hat L^2(\tau_l) - \hat L^2(\tau_k) }_{\LL^2_M(D)}^2
 = \norm{ \sum_{n = k+1}^l \langle \tau, e_n \rangle_{\LL^2_M(D)} \hat L^2(e_n) }_{\LL^2_M(D)}^2
 = \sum_{n = k+1}^l \lambda_n^4 \langle \tau, e_n \rangle_{\LL^2_M(D)}^2.
\end{equation*}
The finiteness of the sum $\sum_{n \geq 1}\lambda_n^4\langle
\tau,e_n\rangle_{\LL^2_M(D)}$ thus implies that $\seq{ \hat L^2(\tau_k) }{k \geq
1}$ is a Cauchy sequence, which by virtue of the completeness of $\LL^2_M(D)$
converges to some $g^* \in \LL^2_M(D)$. The continuity of $\hat L^{-2}$
implies that the $\tau_k$ converge to $\hat L^{-2} g^*$ in $\tilde\HH^4_M(D)$.
As the partial sums converge in $\LL^2_M(D)$ to $\tau$, $\tau = \hat L^{-2} g^*
\in \tilde\HH^4_M(D)$. We have thus proved \eqref{partialFastDecay4}.
\end{proof}

We intend to exploit the previous single-domain results in order to say
something about the multi-domain case. To this end, we define, for $i \in \ii{N}$,
the distributional operators $L_i\colon \{ \varphi \in \LL^1_{\loc}(\D) \colon \grad[\conf_i]\varphi \in [\mathrm{W}^{1,1}_{\loc}(\D)]^d \} \rightarrow
\mathcal{D}'(\D)$ by
\begin{equation*}
L_i(\varphi) := M_i^{-1} \div[\conf_i](M_i \grad[\conf_i]\varphi).
\end{equation*}
We also define $\hat L_i := -L_i + I$, where $I$ is the identity operator
mapping $\mathcal{D}'(\D)$ onto itself. An easily verifiable and important property of these operators is
that, as long as their composition is well-defined, they commute with respect to their spring index. Hence, we can naturally
use multi-indices in $\natural_0^N$ to denote the repeated application of
distinct $L_i$ or $\hat L_i$:
\begin{equation*}
L_\beta := L_1^{\beta_1} \circ \dotsm \circ L_N^{\beta_N}, \qquad
\hat L_\beta := \hat L_1^{\beta_1} \circ \dotsm \circ \hat L_N^{\beta_N},
\end{equation*}
where any zero among the $\beta_i$ is assumed to give rise to the identity
operator. For these multi-indices we define the function $\abs{\beta}_\infty :=
\max_{i \in \ii{N}} \beta_i$. Now, for derivatives in
$\mathcal{D}'(\D)$, the multi-indices belong to $\natural_0^{N d}$ and come
naturally grouped in $N$ length-$d$ sub-multi-indices (one for each factor of
the Cartesian product $\D = D_1 \times \dotsm \times D_N$). Thus we define
the function $\abs{\tcdot}_{\infty,1}\colon \natural_0^{N d}
\rightarrow \natural_0$ by
\begin{equation*}
\abs{\vect{\alpha}}_{\infty,1} = \abs{(\alpha_1,\dotsc,\alpha_N)}_{\infty,1} := \max_{i \in \ii{N}} \abs{\alpha_i}_1 = \max_{i \in \ii{N}} \abs{\alpha_i};
\end{equation*}
that is, the maximal order among the component single-domain
multi-indices.

With this compact notation, we now define the Hilbert spaces (with corresponding norms)
\begin{subequations}\label{H2mix}
\begin{gather}
\label{H2mixSpace}
\tilde\HH^{2,\mathrm{mix}}_\maxw(\D)
:= \left\{ \varphi \in \LTM \colon \partial_{\vect{\alpha}} \in \LTM,\ \abs{\vect{\alpha}}_{\infty,1} \leq 2;\ L_\beta(\varphi) \in \LTM,\ \abs{\beta}_\infty = 1 \right\},\\
\label{H2mixNorm}
\norm{\varphi}_{\tilde\HH^{2,\mathrm{mix}}_\maxw(\D)}^2
:= \sum_{\substack{\vect{\alpha} \in \natural_0^{N d}\\\abs{\vect{\alpha}}_{\infty,1}\leq 2}} \norm{\partial_{\vect{\alpha}} \varphi}_{\LTM}^2
+ \sum_{\substack{\beta \in \natural_0^d\\\abs{\beta}_\infty = 1}} \norm{L_\beta(\varphi)}_{\LTM}^2
\end{gather}
\end{subequations}
and
\begin{subequations}\label{H4mix}
\begin{multline}\label{H4mixSpace}
\tilde\HH^{4,\mathrm{mix}}_\maxw(\D)
:= \Big\{ \varphi \in \LTM \colon \partial_{\vect{\alpha}} \in \LTM,\ \abs{\vect{\alpha}}_{\infty,1} \leq 4;\\
L_\beta(\varphi) \in \HH^2_\maxw(\D),\ \abs{\beta}_\infty = 1;\ L_\beta(\varphi) \in \LTM,\ \abs{\beta}_\infty = 2 \Big\},
\end{multline}
\begin{equation}\label{H4mixNorm}
\norm{\varphi}_{\tilde\HH^{4,\mathrm{mix}}_\maxw(\D)}^2
:= \sum_{\substack{\vect{\alpha} \in \natural_0^{N d}\\\abs{\vect{\alpha}}_{\infty,1}\leq 4}} \norm{\partial_{\vect{\alpha}} \varphi}_{\LTM}^2
+ \sum_{\substack{\beta \in \natural_0^d\\ \abs{\beta}_\infty = 1}} \norm{L_\beta(\varphi)}_{\HH^2_\maxw(\D)}^2
+ \sum_{\substack{\beta \in \natural_0^d\\ \abs{\beta}_\infty = 2}} \norm{L_\beta(\varphi)}_{\LTM}^2.
\end{equation}
\end{subequations}
The following lemma holds.
\begin{lemma}\label[lemma]{lem:fullEmbed}
For $m \in \{2, 4\}$, $\tilde\HH^{m,\mathrm{mix}}_{\maxw}(\D) \subset
\HH^{\Tau\psp{m}}_\maxw(\D)$.
\end{lemma}
\begin{proof}
We recall that, by \cref{lem:tensor-ev}, $\seq{ (\lambda_{\vect{n}},
e_{\vect{n}}) }{ \vect{n} \in \natural^N }$ as defined in \eqref{tensor-ev}
is a complete set of solutions of the $\maxw$-weighted eigenvalue problem
\eqref{full-ev} and that the latter have tensor-product structure. Also, by
the definitions in \eqref{defHSigmaM} and \eqref{HMixIndices},
$\HH^{\Tau\psp{m}}_\maxw(\D)$ is the space of $\LTM$
functions whose squared Fourier coefficients, weighted with the coefficients
defined by
\[ \tau\psp{m}_{\vect{n}} = \prod_{i=1}^N \left( \lambda\psp{i}_{n_i} \right)^m \qquad \forall\,\vect{n} \in \natural^N, \]
have finite sum.

If $\varphi \in \tilde\HH^{m,\mathrm{mix}}_{\maxw}(\D)$, one can apply to it
each operator $\hat L_i$ a total of $m/2$ times \emph{cumulatively} and the result will lie in $\LTM$;
i.e., $\hat L_{(m/2,\dotsc,m/2)}(\varphi) \in \LTM$. By Parseval's identity,
%
\begin{equation}\label{ParsevalMix}
\begin{split}
\infty > \norm{ \hat L_{(m/2,\dotsc,m/2)} \varphi }_{\LTM}^2
& = \sum_{\vect{n} \in \natural^N} \left\langle \hat L_{(m/2,\dotsc,m/2)}(\varphi), e_{\vect{n}} \right\rangle_{\LTM}^2
= \sum_{\vect{n} \in \natural^N} \prod_{i=1}^N \left(\lambda\psp{i}_{n_i}\right)^m \langle \varphi, e_{\vect{n}} \rangle_{\LTM}^2,
\end{split}
\end{equation}
where the second equality is justified by the density of $\CIC(\D)$ functions
in $\HOM$, the regularity of $\varphi$ and the Cartesian product form of the
domain $\D$ and the tensor-product form of the Maxwellian weight function $\maxw$. As the finiteness of the last expression in \eqref{ParsevalMix} is exactly the condition for membership in $\HH^{\Tau\psp{m}}_\maxw(\D)$, we obtain the desired result.

\end{proof}

We recall that \cref{thm:HMix-in-B1} gives a condition on the parameter of
the space $\HH^{\Tau\psp{m}}_\maxw(\D)$ under which it
becomes a subspace of the abstract space $\mathcal{B}_1$, which in turn is
connected by \eqref{A1-B1-isometric} to the space $\mathcal{A}_1$ of fast
convergence of the greedy algorithms (cf.\ \cref{thm:PGA-rate} and
\cref{thm:OGA-rate}). Then, from \cref{lem:fullEmbed} it is apparent
that the arguably less abstract
space $\tilde\HH^{m,\mathrm{mix}}_\maxw(\D)$ will be a
subspace of $\mathcal{B}_1$ for a suitable choice of the parameter.
We shall now make this statement more precise.

\begin{theorem}\label{thm:HtMix-in-B1} Let $\HH^{\Tau\psp{4}}_\maxw(\D)$ be
defined according to \eqref{HMixIndices}, where $d \in \{2, 3\}$, as elsewhere,
is the common dimensionality of the Cartesian factors that make up $\D$; then,
\[ \tilde\HH^{4,\mathrm{mix}}_\maxw(\D) \subset \HH^{\Tau\psp{4}}_\maxw(\D) \subset \mathcal{B}_1. \]
\end{theorem}
\begin{proof}
\Cref{lem:fullEmbed} gives that $\tilde\HH^{4,\mathrm{mix}}_\maxw(\D)
\subset \HH^{\Tau\psp{4}}_\maxw(\D)$. As $4$ is greater than $1+\frac{d}{2}$
for both $d = 2$ and $d = 3$, \cref{thm:HMix-in-B1} gives that
$\HH^{\Tau\psp{4}}_\maxw(\D) \subset \mathcal{B}_1$.

\end{proof}

\begin{remark}\label[remark]{rem:On-HMix}
If the hypotheses we have been making throughout this work (i.e., Hypotheses
\ref{hyp:potential}, \ref{hyp:partialCompEmb}, \ref{hyp:asymptotics},
\ref{hyp:potential2} and \ref{hyp:power-like}) are met, nothing in our
arguments essentially restricts the results to the physically relevant cases $d
= 2$ and $d = 3$. In particular, in the case of $d = 1$, the combination of
\cref{thm:HMix-in-B1} with \cref{lem:fullEmbed} and the fact that $2 > 1+\frac{1}{2}$ yields that
\[ \tilde\HH^{2,\mathrm{mix}}_\maxw(\D) \subset \HH^{\Tau\psp{2}}_\maxw(\D) \subset \mathcal{B}_1 .\]

Sobolev spaces of dominating mixed smoothness akin to $\tilde\HH^{2,\mathrm{mix}}_\maxw(\D)$ can also be shown to be subspaces of the regularity class $\mathcal{B}_1$ in the case of the classical Poisson problem studied in \cite{LLM}: Find $\psi \in \HH^1_0(\D)$ (with the standard meaning of
the Sobolev space $\HH^1_0(\D)$; i.e., the set of all elements of $\HH^1(\D)$ that have
zero trace on $\partial\D$---not a zero-weighted Sobolev space!) such that
\[ \langle \psi, \varphi \rangle_{\HH^1_0(\D)} = \langle f, \varphi \rangle_{\LL^2(\D)} \quad \forall\,\varphi\in \HH^1_0(\D),\]
where $\D = D_1 \otimes \dotsm \otimes D_N$ and each $D_i$, $i \in \ii{N}$, is an open
interval. The corresponding greedy algorithms seek approximations that are
linear combinations of $\tp_{i\in\ii{N}} \HH^1_0(D_i)$ functions. The argument
of \cref{thm:HMix-in-B1} above holds in this case without any change, and so, given
that the $n$-th eigenvalue of the corresponding analogue to the partial-domain
eigenvalue problem \eqref{partial-ev} is proportional to $n^2$, we have that
\begin{equation*}
\left\{
\varphi \in \LL^2(\D)\colon \sum_{\vect{n} \in \natural^N} \left[ \left( \sum_{i=1}^N n_i^2 \right)^2 + \prod_{i=1}^N \left( n_i^2 \right)^2 \right] \langle \varphi, e_{\vect{n}} \rangle_{\LL^2(\D)}^2 < \infty
\right\} \subset \mathcal{B}_1.
\end{equation*}
In this non-degenerate setting it is possible to identify the space on the
left-hand side of the above expression with
\begin{equation*}
\HH^{2,\mathrm{mix}}(\D) \cap \HH^1_0(\D) := \left\{ \varphi \in \HH_0^1 \colon \partial_\alpha \varphi \in \LL^2(\D),\ \abs{\alpha}_\infty = \max_{1\leq i \leq N} \alpha_i \leq 2 \right\}.
\end{equation*}
This characterization should be contrasted with the condition for membership in $\mathcal{A}_1$ (which is identical to $\mathcal{B}_1$ in this unweighted setting)
derived in \cite[Remark 4]{LLM}, which demands, instead, that the true solution belongs to
$\HH^m(\D) \cap \HH^1_0(\D)$, with $m > 1 + N/2$. In fact the characterization given in \cite[Remark 4]{LLM}
can be generalized to the requirement that the exact solution belongs to $\HH^m(\D) \cap \HH^1_0(\D)$, with
$m > 1 + N d / 2$, when the factor domains are no longer one-dimensional but $d$-dimensional; and, thanks
to \cref{thm:HUnif-in-B1}, such a characterization in terms of standard Sobolev spaces (rather than
spaces of dominating mixed smoothness) also has a counterpart in our degenerate setting.

An attractive feature of spaces of dominating mixed
smoothness is that their regularity index is independent of $N$ and such spaces are
therefore more naturally suited to (high-dimensional) tensor-product settings such as ours.
We note in this respect that we conjecture
that the reverse of the inclusion stated in \cref{lem:fullEmbed} also holds, implying
equality of the two spaces there---just as in \cref{lem:partialEquiv} for the single-domain
spaces. We suspect that the proof of this would involve tensorizing the statements of
\cref{lem:ellipReg} and \cref{lem:iteratedEllipReg}.
However, even if \cref{lem:fullEmbed} held with an equality of spaces, there would
still be some slack between the discussed mixed smoothness levels and the lower bound of the admissible parameter $m$ such that $\HH^{\Tau\psp{m}}_\maxw(\D) \subset \mathcal{B}_1$. The reason is that we have gone about obtaining elliptic
regularity results by two integer degrees of differentiation at a time. Consequently,
we have not defined anything we could label $\tilde\HH^m_{M_i}(D_i)$ or
$\tilde\HH^{m,\mathrm{mix}}_\maxw(\D)$ with $m \notin \{2,4\}$ while being consistent with the
definitions we have given for the single-spring spaces $\tilde\HH^2_{M_i}(D_i)$ in
\eqref{partial-Ht2} and $\tilde\HH^4_{M_i}(D_i)$ in \eqref{partial-Ht4}, and with
the multi-spring spaces $\tilde\HH^{2,\mathrm{mix}}_\maxw(\D)$ in \eqref{H2mix} and
$\tilde\HH^{4,\mathrm{mix}}_\maxw(\D)$ in \eqref{H4mix}. Given the presence of
the second-order operators of the form $M_i^{-1}\div(M_i\grad \tcdot)$ and
$\maxw^{-1}\div(\maxw \grad \tcdot)$ in the definition of these even-indexed
spaces, it is not immediately clear what a suitable explicit definition of the analogous
odd-indexed---let alone non-integer-indexed---spaces might be.
We shall address this question by using function space interpolation.

We start with the fact that, given two nets of positive weights
$\Sigma\psp{i} = \seq{ \sigma\psp{i}_{\vect{n}} }{ \vect{n} \in \natural^N
}$, $i \in \{1, 2\}$, one can show that for $\theta \in (0,1)$ the (real) $(\theta,2)$-interpolation space
between them obeys
\[ \left[\HH^{\Sigma\psp{1}}_\maxw(\D), \HH^{\Sigma\psp{2}}_\maxw(\D)\right]_{\theta,2} = \HH^{\tilde\Sigma}_\maxw(\D), \]
where
$\tilde\Sigma = \seq{ \tilde\sigma_{\vect{n}} }{ \vect{n} \in \natural^N }$ and $\tilde\sigma_{\vect{n}} := \left(\sigma\psp{1}_{\vect{n}}\right)^{1-\theta} \left(\sigma\psp{2}_{\vect{n}}\right)^{\theta}$ for all $\vect{n} \in \natural^N$,
with equivalence of norms (the proof is a simple modification of the argument
given in \cite[Chapter 23]{Tartar}). As, according to the definition in \eqref{HMixIndices},
\begin{equation*}
\tau\psp{2\theta}_{\vect{n}} = \left( \tau\psp{0}_{\vect{n}} \right)^{1-\theta} \left( \tau\psp{2}_{\vect{n}} \right)^\theta
\quad\text{and}\quad \tau\psp{2+2\theta}_{\vect{n}} = \left( \tau\psp{2}_{\vect{n}} \right)^{1-\theta} \left( \tau\psp{4}_{\vect{n}} \right)^\theta
\end{equation*}
for all $\theta \in (0,1)$ and $\vect{n}$ in $\natural^N$, it follows that
\begin{equation*}
\HH^{\Tau\psp{2\theta}}_\maxw(\D) = \left[ \HH^{\Tau\psp{0}}_\maxw(\D), \HH^{\Tau\psp{2}}_\maxw(\D) \right]_{\theta,2}
\quad\text{and}\quad \HH^{\Tau\psp{2+2\theta}}_\maxw(\D) = \left[ \HH^{\Tau\psp{2}}_\maxw(\D), \HH^{\Tau\psp{4}}_\maxw(\D) \right]_{\theta,2},
\end{equation*}
with equivalence of norms. Given that the inclusions in \cref{lem:fullEmbed}
are actually continuous embeddings and $\HH^{\Tau\psp{0}}_\maxw(\D) = \LTM$, it follows that
\begin{equation*}
\left[\LTM,\tilde\HH^{2,\mathrm{mix}}_\maxw(\D)\right]_{\theta,2} \subset \HH^{\Tau\psp{2\theta}}
\quad\text{and}\quad
\left[\tilde\HH^{2,\mathrm{mix}}_\maxw(\D), \tilde\HH^{4,\mathrm{mix}}_\maxw(\D)\right]_{\theta,2} \subset \HH^{\Tau\psp{2+2\theta}}_\maxw(\D),
\end{equation*}
with continuous embedding. Since whenever $m > 1 + \frac{d}{2}$ we have that $\HH^{\Tau\psp{m}}_\maxw(\D)
\subset \mathcal{B}_1$, defining $\tilde\HH^{m,\mathrm{mix}}_\maxw(\D)$ as the
interpolation space appearing on the left-hand side of the corresponding inclusion above if $m \in (0,2)$ or $m \in (2,4)$ and as $\LTM$ if $m = 0$ is an appealing idea, for then we can simply state that
\[ m > 1+\frac{2}{2} \implies \tilde\HH^{m,\mathrm{mix}}_\maxw(\D) \subset \mathcal{B}_1. \]
\end{remark}

\section{Conclusions and directions for future work}
\label{sec:conclusions}

We proved the well-posedness (\cref{thm:existence}) and convergence
(\cref{thm:PGA-converges} and \cref{thm:OGA-converges}) of two greedy
algorithms, which seek approximations to solutions of high-dimensional
and degenerate Fokker--Planck equations using a separated representation
procedure. We then gave sufficient conditions on the true solution of
the equation for the fast-convergence of the approximations given by those
algorithms; first, in terms of summability of Fourier coefficients
(\cref{thm:HMix-in-B1}), and then, in terms of regularity
(\cref{thm:HtMix-in-B1}). In the process of proving these main results, a
number of auxiliary results
were proved, some of which are of interest in their own right; e.g., function
spaces with tensor-product weights inherit compact embedding
(\cref{lem:compEmb}) and density (\cref{cor:tensorDensity}) properties
from the spaces corresponding to the weights that appear as factors of the
tensor product weight; the existence of elliptic regularity results for
single-spring degenerate elliptic problems (particularly
\cref{lem:iteratedEllipReg}); and eigenvalue asymptotics in the same
degenerate setting (\cref{lem:conditionsForWeyl} in
\cref{sec:ev-asymptotics}).


The greedy algorithms described in \cref{sec:SR} are abstract. They entail
obtaining the true minima of functionals in nonlinear manifolds embedded in
infinite-dimensional function spaces (cf.\ \eqref{greedy1} and
\eqref{greedy2}). Any practical implementation of the separated representation
strategy must then introduce a discretization  (e.g., by a finite element method or
a spectral method) and a procedure for the approximation of those minima in the
resulting discretized manifolds (e.g., an alternating direction scheme, Newton
iteration, etc.). The mathematical analysis of the effects of the discretization and
the use of approximate minimization algorithms on the convergence of the greedy
algorithms is the subject of ongoing research.
On a related note, we are also interested in the implementation of the
combination of the separated representation strategy and the alternating
direction scheme described in \eqref{FP-weak-SD-x-2} and \eqref{FP-weak-SD-q-2}
in order to approximate the full Fokker--Planck equation \eqref{FP-full}.
Further up in model complexity is the coupling between the full Fokker--Planck
equation and the Navier--Stokes equations for the velocity and pressure of an
incompressible solvent, which is also of interest to us. The Navier--Stokes--Fokker--Planck
system is a fully coupled macro-micro system, since
the configuration probability density function given by
the Fokker--Planck equation feeds into the Navier--Stokes equations a
contribution to the extra-stress tensor while the Navier--Stokes velocity field
enters in the Fokker--Planck equation (cf.\ \cite[\S 15.2]{Bird2}). An
important property of the full Fokker--Planck equation is that its solution
is almost everywhere nonnegative and has unit integral over
the configuration space $\D$ (i.e., it is a probability density function) at
almost every point in time $t$ and space $\vect{x}$ if the initial condition
has those properties. It is of interest to learn whether the separated
representation strategy can be adapted to give approximations that also
preserve the property of being a probability density function.


The generalization of our results to other tensor-product-based
high-dimensional PDEs is also of interest. In particular, the adaptation of the separated
representation strategy to the Fokker--Planck equations for the configuration
of \emph{bead-rod} polymer chains (see, e.g., \cite[\S 11.3]{Bird2}) is of
relevance; these models are not covered by our arguments, at least not in their
present form.

\appendix

\section{Basic results for Sobolev spaces weighted with CPAIL Maxwellians}
\label{sec:CPAIL-results}

We shall derive some key properties of the function spaces associated with the
CPAIL force model \eqref{CPAIL-model}, with parameter $b \geq 3$, using the
corresponding properties of the function spaces associated with the FENE force
model \eqref{FENE-model}, with parameter $2b/3$.

Let $b \geq 3$. It follows from \eqref{FENE-model}, \eqref{CPAIL-model} and
\eqref{partial-maxw} that the Maxwellian $M_\mathrm{C}$ associated to a spring
obeying the CPAIL model with parameter $b$ and the Maxwellian $M_\mathrm{F}$
associated to a spring obeying the FENE model with parameter $2b/3$ are,
respectively,
\[ M_\mathrm{C}(\ponf) = Z_\mathrm{C} \exp(-\abs{\ponf}^2/6) \left( 1 - \frac{\abs{\ponf}^2}{b} \right)^{b/3}, \qquad \ponf \in D_\mathrm{C} = B\left(0,\sqrt{b}\right) \subset \Real^d \]
and
\[ M_\mathrm{F}(\ponf) = Z_\mathrm{F} \left( 1 - \frac{\abs{\ponf}^2}{2 b/3} \right)^{b/3}, \qquad \ponf \in D_\mathrm{F} = B\left(0,\sqrt{2b/3}\right) \subset \Real^d, \]
where $Z_\mathrm{C}$ and $Z_\mathrm{F}$ are positive constants whose specific
values are of no particular relevance below. Let us denote by $T$ the invertible map
$\ponf \in D_\mathrm{C} \mapsto \sqrt{2/3}\,\ponf \in D_\mathrm{F}$. On defining
$\tilde M\colon D_\mathrm{C} \rightarrow \Real$ via $\tilde M := M_\mathrm{F} \circ
T$ we find that there exist positive constants $c_1$ and $c_2$ such that $c_1
\tilde M \leq M_{\mathrm{C}} \leq c_2 \tilde M$. This implies that
$\HH^1_{M_\mathrm{C}}(D_\mathrm{C})$ and $\HH^1_{\tilde M}(D_\mathrm{C})$ (the
latter is well-defined since $\tilde M^{-1}$ inherits from $M_\mathrm{C}^{-1}$
its $\LL^1_{\loc}(D_\mathrm{C})$ regularity---thereby falling under the
hypotheses of \cite[Theorem 1.11]{KO}) are algebraically and topologically the
same space. The same is true of the pairs of spaces given by
$\LL^2_{M_\mathrm{C}}(D_\mathrm{C})$ and $\LL^2_{\tilde M}(D_\mathrm{C})$ and
$\HH(M_\mathrm{C};D_\mathrm{C})$ and $\HH(\tilde M;D_\mathrm{C})$.

Now, $T$ and $T^{-1}$ are $[\CC^\infty(\overline{D_\mathrm{C}})]^d$ and
$[\CC^\infty(\overline{D_\mathrm{F}})]^d$ functions, respectively. Then, an
argument analogous to \cref{lem:weightedSpacesMap} leads to the fact that
the composition with $T^{-1}$ is a well-defined, invertible, linear and bounded
operator between $\HH^1_{\tilde M}(D_\mathrm{C})$ and
$\HH^1_{M_\mathrm{F}}(D_\mathrm{F})$ and also between $\LL^2_{\tilde
M}(D_\mathrm{C})$ and $\LL^2_{M_\mathrm{F}}(D_\mathrm{F})$, and its inverse is
the composition with $T$. By \eqref{iso-iso}, composition with $T^{-1}$ is
also such an operator between $\HH(D_\mathrm{C};\tilde M)$ and
$\HH(D_\mathrm{F};M_\mathrm{F})$ having as its inverse the composition with
$T$.

We can thus use the connection between the $M_\mathrm{F}$-weighted spaces and
the $\tilde M$-weighted spaces and the connection between the latter and the
$M_\mathrm{C}$-weighted spaces to state that
\begin{equation*}
\HH^1_{M_\mathrm{F}}(D_\mathrm{F}) \compEmb \LL^2_{M_\mathrm{F}}(D_\mathrm{F}) \implies
\HH^1_{M_\mathrm{C}}(D_\mathrm{C}) \compEmb \LL^2_{M_\mathrm{C}}(D_\mathrm{C})
\end{equation*}
and
\begin{equation*}
\overline{\CIC(D_\mathrm{F})}^{\HH(D_\mathrm{F};M_\mathrm{F})} = \HH(D_\mathrm{F};M_\mathrm{F}) \implies
\overline{\{ f \circ T \colon f \in \CIC(D_\mathrm{F})\}}^{\HH(D_\mathrm{C};M_\mathrm{C})} = \HH(D_\mathrm{C};M_\mathrm{C}).
\end{equation*}
As $2b/3 \geq 2$, the statements on the left-hand side of the above
implications hold (as noted in \cref{rem:hypotheses} and
\cref{rem:densities}); consequently, so do the statements on each right-hand
side. By noting that, on account of its infinite differentiability, the
composition with $T$ maps $\CIC(D_\mathrm{F})$ into $\CIC(D_\mathrm{C})$ and
that $M_\mathrm{C}$ itself is a $\CC^\infty(D_\mathrm{C})$ function, we have
proved the following lemma.

\begin{lemma}\label[lemma]{lem:CPAIL-results} Let $M\colon D\rightarrow\Real$ be the
Maxwellian associated to a spring obeying the CPAIL force model
\eqref{CPAIL-model} with parameter $b \geq 3$. Then,
the compact embedding $\HH^1_M(D) \compEmb \LL^2_M(D)$ holds; and
the set $\CIC(D)$ is dense in $\HH(D;M)$.
\end{lemma}

\section{Some results on distributions}
\label{sec:distrResults}

Throughout this section $\Omega$ will denote an open subset of $\Real^d$.

\begin{lemma}\label[lemma]{lem:LeibnizFormula}
Let $\alpha \in \natural_0^d$ and let $f \in \LL^1_{\loc}(\Omega)$ and $g \in
\CC(\Omega)$ be such that $\partial_\beta f \in \LL^1_{\loc}(\Omega)$ and
$\partial_\beta g \in \CC(\Omega)$ for all $\beta \leq \alpha$. Then,
$\partial_\alpha(f g) \in \LL^1_{\loc}(\Omega)$ and
\begin{equation}\label{LeibnizFormula}
\partial_\alpha(f g) = \sum_{\beta \leq \alpha} \frac{\alpha!}{\beta! (\alpha - \beta)!} \partial_\beta f \,\partial_{\alpha-\beta} g,
\end{equation}
where the convention $\gamma! = \prod_{i\in\ii{d}} \gamma_i!$ for all $\gamma
\in \natural_0^d$ has been followed.
\end{lemma}
\begin{proof}
The result is obviously true for $\alpha = (0,\dotsc,0)$. Then, in the
$\abs{\alpha} = 1$ case, the result is stated under the assumption of $f$, $g$,
$\partial_\alpha f $, $\partial_\alpha g$, $f g$ and $f\partial_\alpha g +
g\partial_\alpha f$ being members of $\LL^1_{\loc}(\Omega)$ (which is clearly
implied by our hypotheses) in the discussion that follows Theorem~7.4 of
\cite{GT}. The final result follows from standard combinatorial arguments and
an induction procedure.

\end{proof}

The purpose of the following lemma is to formulate a result analogous to
Theorem~3.41 of \cite{AF:2003} for weighted Sobolev spaces without resorting to
density arguments, which may be unavailable for one or both
of the weighted Sobolev spaces being connected.

\begin{lemma}\label[lemma]{lem:distrComposition}
Let $T$ be an invertible $\CC^\infty(\overline{\Omega})$ transformation with
codomain $\tilde \Omega$ and let $f \in \LL^1_{\loc}(\Omega)$ be such that its
distributional derivatives are in $\LL^1_{\loc}(\Omega)$ up to the order
$\alpha \in \natural^d$. Then,
\begin{equation}\label{FaaDiBruno}
\partial_\alpha (f \circ T^{-1})
= \sum_{1 \leq \abs{\beta} \leq \abs{\alpha}} M_{\alpha,\beta} (\partial_\beta f \circ T^{-1})\in \LL^1_{\loc}(\tilde\Omega),
\end{equation}
where $M_{\alpha,\beta}$ is a polynomial of degree not exceeding $\abs{\beta}$
in derivatives of orders not exceeding $\abs{\alpha}$ of the various components
of $T^{-1}$.
\end{lemma}
\begin{proof}
Let $S$ denote the inverse of $T$ and let $S_k$ denote the its $k$-th
component. From Theorem~6.1.2 in \cite{Hormander:1983-I} and the remark that
follows it we know, first, that there exists a unique continuous linear map $S^*
\colon \mathcal{D}'(\Omega) \rightarrow \mathcal{D}'(\tilde \Omega)$ whose
restriction to $\CC(\Omega)$ is $u \mapsto u \circ S$ and, second, that the
chain rule,
\[ \partial_j S^* u = \sum_{k=1}^d \partial_j S_k\, S^* \partial_k u \]
holds in $\mathcal{D}'(\tilde \Omega)$. It is easy to see (either directly or
from the proof of Theorem~6.1.2 of \cite{Hormander:1983-I}) that $S^* u$ has
the explicit form
\[ S^*u (\varphi) = u\big( (\varphi \circ T) \abs{\det(\grad T)} \big) \qquad \forall\,\varphi \in \CIC(\tilde \Omega). \]

For a regular distribution such as $f$ the above characterization and the change
of variable formula for integrable functions (see, e.g.,
\cite[Theorem 3.7.1]{Bogachev}) makes $S^* f$ precisely the regular
distribution associated with the $\LL^1_{\loc}(\tilde \Omega)$ function $f
\circ S$. Similarly, $S^* \partial_k f$ will be the regular distribution
associated with the $\LL^1_{\loc}(\tilde \Omega)$ function $\partial_k f \circ
S$. Hence, $\partial_j(f \circ T^{-1}) = \sum_{k=1}^d \partial_j S_k \partial_k
f \circ S$ and \eqref{FaaDiBruno} is proved for $\abs{\alpha} = 1$. An
induction argument then establishes \eqref{FaaDiBruno} in the general case.
\end{proof}

\begin{lemma}\label[lemma]{lem:weightedSpacesMap}
Let $\tilde\Omega$ and $T$ be as in \cref{lem:distrComposition} and let $w$
be a weight function defined on $\Omega$. Then, $f \in \HH^m_w(\Omega)$ if, and only if, $f \circ
T^{-1} \in \HH^m_{\tilde w}(\tilde \Omega)$ and there exist positive constants
$c_1(m)$ and $c_2(m)$ such that
\begin{equation*}
c_1 \norm[n]{f \circ T^{-1}}_{\HH^m_{\tilde w}(\tilde \Omega)}
\leq \norm{f}_{\HH^m_w(\Omega)}
\leq c_2 \norm[n]{f \circ T^{-1}}_{\HH^m_{\tilde w}(\tilde \Omega)},
\end{equation*}
where $\tilde w = w \circ T^{-1}$.
\end{lemma}
\begin{proof} We use \cref{lem:distrComposition} to replace the first part
of the proof of Theorem~3.41 of \cite{AF:2003}. Then, the rest of that proof,
\emph{mutatis mutandis}, carries over to our case.
\end{proof}


\section{Eigenvalue asymptotics for Ornstein--Uhlenbeck operators with FENE and CPAIL potentials via the Liouville transformation}
\label{sec:ev-asymptotics}

\begin{lemma}\label[lemma]{lem:easyAsymptotics}
Let $\Omega \subset \Real^d$ be a bounded and convex domain of class $\CC^3$
and let $w \in \CC^2(\Omega)$ be a positive function such that $\CC^2_0(\Omega)$
is dense in $\HH^1_w(\Omega)$ and $\HH^1_w(\Omega) \compEmb \LL^2_w(\Omega)$.
We further assume that
\begin{enumerate}
\item\label{QBounded} $\displaystyle\inf_{\ponf \in \Omega} Q_1(\ponf) > -\infty$, or
\item\label{dsqQBoundedByThat} there exists a $\Theta > 0$ such that $\displaystyle\gamma_\Theta :=
\inf_{\ponf \in \Omega} \mathfrak{d}(\ponf)^2 Q_\Theta(\ponf) \in (-1/4,0]$,
\end{enumerate}
where $Q_\Theta := \Theta-w^{-1/2}\div(w\grad w^{-1/2})$
and $\mathfrak{d}$ is the distance-to-the-boundary function in $\Omega$.

Let $\seq{\lambda_n}{n \in \natural}$ be the (ordered, with repetitions
according to multiplicity) sequence of eigenvalues of the problem: Find
$\lambda \in \Real$ and $u \in \HH^1_w(\Omega) \setminus \{0\}$ such that
\begin{equation}\label{unshifted-ev}
\langle u, v\rangle_{\HH^1_w(\Omega)} = \lambda \langle u, v\rangle_{\LL^2_w(\Omega)} \qquad \forall\,v \in \HH^1_w(\Omega).
\end{equation}
Then, there exist positive numbers $c_1$ and $c_2$ and a natural number $n_0$
such that
\begin{equation}\label{twoSidedBound}
n \geq n_0 \implies c_1 n^{2/d} \leq \lambda_n \leq c_2 n^{2/d}.
\end{equation}
\end{lemma}
\begin{proof}
Let, for $\Theta > 0$, $\seq{\lambda_{\Theta,n}}{n \in \natural}$ be the
(ordered, with repetitions according to multiplicity) sequence of of
eigenvalues of the shifted problem: Find $\lambda^\Theta \in \Real$ and $u \in
\HH^1_w(\Omega) \setminus \{0\}$ such that
\begin{equation}\label{shifted-ev}
\langle u, v \rangle_{\HH^1_w(\Omega),\Theta} := \langle \grad u, \grad v\rangle_{[\LL^2_w(\Omega)]^d} + \Theta \langle u, v\rangle_{\LL^2_w(\Omega)} = \lambda^\Theta \langle u, v\rangle_{\LL^2_w(\Omega)} \quad\forall\,v\in\HH^1_w(\Omega).
\end{equation}
By the hypotheses of the lemma the existence and the accumulation at $\infty$
only of the $\lambda_{\Theta,n}$ is guaranteed via \cref{lem:abstractEV}. It
further follows from the spectral theory of self-adjoint compact operators that
$\lambda_{\Theta,n}$ can be characterized by the Courant--Fischer--Weyl min-max principle:
\begin{equation}\label{EV-infsup}
\lambda_{\Theta,n}
= \min_{\substack{\dim(S) = n\\S \subset \HH^1_w(\Omega)}} \max_{z \in S\setminus\{0\}}
\frac{\left\langle z,z\right\rangle_{\HH^1_w(\Omega),\Theta}}{\left\langle z,z\right\rangle_{\LL^2_w(\Omega)}}
= \inf_{\substack{\dim(S) = n\\S \subset \CC^2_0(\Omega)}} \sup_{z \in S\setminus\{0\}}
\frac{\left\langle z,z\right\rangle_{\HH^1_w(\Omega),\Theta}}{\left\langle z,z\right\rangle_{\LL^2_w(\Omega)}},
\end{equation}
the second equality being a consequence of the density of $\CC^2_0(\Omega)$ in
$\HH^1_w(\Omega)$ (cf.\ \cite[Theorem 4.5.3]{Davies}). Note that when $\Theta =
1$ the problem \eqref{shifted-ev} and the problem \eqref{unshifted-ev} coincide
(and so do the sequences $\seq{ \lambda_{\Theta,n} }{n \in \natural}$ and
$\seq{ \lambda_n }{n \in \natural}$).

Let $L := w^{-1/2} \in \CC^2(\Omega)$, let $z$ be an arbitrary $\CC^2_0(\Omega)$
function and let $y := L^{-1} z$. Then,
\begin{equation*}
\begin{split}
\norm{z}_{\HH^1_w(\Omega),\Theta}^2 
& = \int_{\Omega} \left(\abs{\grad(L y)}^2 + \Theta (L y)^2\right) L^{-2}\\
& = \int_{\Omega} \abs{\grad y}^2 + \int_{\Omega} \left(\Theta + L^{-2}\abs{\grad L}^2\right) y^2 + \int_{\Omega} L^{-1} \grad L \cdot \grad (y^2)\\
& = \int_{\Omega} \abs{\grad y}^2 + \int_{\Omega} \left(\Theta + L^{-2}\abs{\grad L}^2\right) y^2 - \int_{\Omega} \div\left(L^{-1} \grad L\right) y^2\\
& = \int_{\Omega} \abs{\grad y}^2 + \int_{\Omega} \left[\Theta - L\, \div(L^{-2}\grad L)\right] y^2\\
& = \int_{\Omega} \abs{\grad y}^2 + \int_{\Omega} Q_\Theta \,y^2.
\end{split}
\end{equation*}
Similarly, $\norm{z}_{\LL^2_w(\Omega)}^2 = \norm{y}_{\LL^2(\Omega)}^2$. As $z
\in \CC^2_0(\Omega)$ is arbitrary and $z \mapsto L^{-1} z$ is a bijection of
$\CC^2_0(\Omega)$ into itself, \eqref{EV-infsup} begets
\begin{equation*}
\lambda_{\Theta,n}
= \inf_{\substack{\dim(S) = n\\ S \subset\CC^2_0(\Omega)}} \sup_{y \in S\setminus\{0\}} \frac{ \norm{\grad y}_{[\LL^2(\Omega)]^d}^2 + \int_{\Omega} Q_\Theta\, y^2 }{ \norm{y}_{\LL^2(\Omega)}^2 }.
\end{equation*}
If condition (\ref{QBounded}) holds, there must exist a $\Theta > 0$ such
that $Q_\Theta \geq 0$ in $\Omega$. For such a $\Theta$, of course, $\int_\Omega
Q_\Theta y^2 \geq 0$. On the other hand, if condition (\ref{dsqQBoundedByThat})
is met, then with the particular $\Theta$ given in the condition we have that
\begin{equation*}
\int_{\Omega} Q_\Theta\, y^2
\geq \gamma_\Theta \int_{\Omega} \frac{y^2}{\mathfrak{d}^2}
\geq \frac{\gamma_\Theta}{4} \norm{\grad y}_{[\LL^2_w(\Omega)]^d}^2,
\end{equation*}
the last inequality being a multi-dimensional Hardy inequality (see, e.g.,
\cite[Theorem 11]{MMP-1998}, bearing in mind that $\gamma_\Theta$ has been assumed to be
nonpositive). In either case, we can write
\begin{equation}\label{lowerComparison}
\lambda_{\Theta,n} \geq \inf_{\substack{\dim(S) = n\\ S \subset\CC^2_0(\Omega)}} \sup_{y \in S\setminus\{0\}} \frac{ \alpha \norm{\grad y}_{[\LL^2(\Omega)]^d}^2 }{ \norm{y}_{\LL^2(\Omega)}^2 },
\end{equation}
where
\begin{equation*}
0 < \alpha := \begin{cases}
1 & \text{if condition (\ref{QBounded}) holds},\\
(1+\gamma_\Theta/4) & \text{if condition (\ref{dsqQBoundedByThat}) holds}.
\end{cases}
\end{equation*}
The $\CC^3$ regularity of $\partial\Omega$ implies the
existence of an $\varepsilon_0 \in (0,1)$ such that for each $\varepsilon \in
(0,\varepsilon_0)$ there exists a subdomain $\Omega_\varepsilon \compEmb
\Omega$ that is also of class $\CC^3$ and has measure
$(1-\varepsilon)\abs{\Omega}$. Fixing $\varepsilon \in (0,\varepsilon_0)$, the
fact that the extensions by zero of functions in $\CC^2_0(\Omega_\varepsilon)$
form a subspace of $\CC^2_0(\Omega)$ and \eqref{EV-infsup} imply that the eigenvalues
of the unshifted problem \eqref{unshifted-ev} can be bounded from above
according to
\begin{equation}\label{upperComparison}
\lambda_n \leq \inf_{\substack{\dim(S) = n\\ S \subset \CC^2_0(\Omega_\varepsilon)}} \sup_{z \in S\setminus\{0\}}
\frac{ \langle z, z\rangle_{\HH^1_w(\Omega_\varepsilon)} }{ \langle z, z\rangle_{\LL^2_w(\Omega_\varepsilon)} }.
\end{equation}

Now, the right-hand side of \eqref{lowerComparison} and the right-hand side of
\eqref{upperComparison} are precisely the $n$-th eigenvalue associated with the
(variational form of the) problem
\[ -\alpha \lapl y = \mu y \ \text{ in } \Omega, \qquad y = 0 \text{ on } \partial \Omega \]
and the problem
\begin{equation*}
-\div(w \grad y) + w y = \nu w y \ \text{ in } \Omega_\varepsilon, \qquad y = 0
\text{ on } \partial\Omega_\varepsilon,
\end{equation*}
respectively. These standard eigenvalue problems obey Weyl's law (this results
from the fairly general Theorem~2.4 of \cite{Clark} with input from the
regularity result in \cite[Theorem 2.4]{Browder:1961}---alternatively, see
\cite[\S VI.4.4]{CH1}); that is,
\begin{subequations}\label{countingLimits}
\begin{gather}\lim_{\mu\to\infty} \frac{\# \{ n \in \natural\colon \mu_n \leq \mu
\}}{\mu^{d/2}} = \frac{\alpha^{-d/2}\abs{\Omega}}{(2\sqrt{\pi})^d \, \Gamma(1+d/2)} = \alpha^{-d/2} C > 0,\\
\lim_{\nu\to\infty} \frac{\# \{ n \in \natural\colon \nu_n \leq \nu
\}}{\nu^{d/2}} = \frac{\abs{\Omega_\varepsilon}}{(2\sqrt{\pi})^d \, \Gamma(1+d/2)} = (1-\varepsilon) C > 0,
\end{gather}
\end{subequations}
where $C := \abs{\Omega}((2\sqrt{\pi})^d \Gamma(1+d/2))^{-1}$. Particularizing
these limits to $\mu = \mu_n$ and $\nu = \nu_n$ they turn into statements about
the rate of growth of the eigenvalues themselves, as opposed to the counting
functions. That is,
\[ \lim_{n\to\infty} \mu_n / n^{2/d} = \alpha C^{-2/d} \qquad \text{and} \qquad \lim_{n\to\infty} \nu_n / n^{2/d} = (1-\varepsilon)^{-2/d} C^{-2/d}. \]
From the definition of the shifted eigenvalue problem \eqref{shifted-ev}, for
any $\Theta$, it is immediate that
$\lambda_{\Theta,n} = \lambda_n + \Theta - 1$ for all $n\in\natural$.
We then deduce, via the inequalities \eqref{lowerComparison} and
\eqref{upperComparison}, that the asymptotic bounds \eqref{twoSidedBound} hold.
\end{proof}

\begin{remark}\label[remark]{rem:notTheBest}\noindent
\begin{enumerate}
\item It follows from the proof of \cref{lem:easyAsymptotics} that, if
condition (\ref{QBounded}) holds, the constants $c_1$ and $c_2$ of
\eqref{twoSidedBound} can be taken arbitrarily close to $C^{-2/d}$ and,
consequently, to each other.
\item One might relax the condition of convexity of the domain in
\cref{lem:easyAsymptotics} at the possible cost of having a stricter lower
bound for $\gamma_\Theta$ in condition (\ref{dsqQBoundedByThat}), as the
constant for the Hardy inequality might deteriorate. The $\CC^3$ regularity
condition on the domain can be drastically relaxed (see, for example
\cite{BirmanSolomjak:1970}); however, the literature tends to force one to
choose at most two among readability, the size of the class of problems covered,
and frugality in terms of hypotheses. For our purposes, the statement in
\cref{lem:easyAsymptotics} suffices.
\end{enumerate}
\end{remark}

\begin{corollary}\label[corollary]{cor:easyAsymptotics}
The eigenvalues of the eigenvalue problem \eqref{partial-ev} associated with
both the FENE model \eqref{FENE-model} and the CPAIL model \eqref{CPAIL-model}
obey \eqref{twoSidedBound} if their parameter $b_i$ is greater than $2$ and
$3$, respectively.
\end{corollary}
\begin{proof}
We shall apply \cref{lem:easyAsymptotics}. For both the FENE and CPAIL
models the domains (being balls) and their associated Maxwellian weights are
regular enough. The compact embedding and density hypotheses are satisfied in
the parameter ranges under consideration (cf.\ \cref{hyp:partialCompEmb},
\cref{rem:hypotheses}, \cref{rem:densities} and \eqref{iso-iso}). It only
remains to prove condition (\ref{QBounded}) or condition
(\ref{dsqQBoundedByThat}).

From \eqref{FENE-model} and \eqref{partial-maxw} it follows that the Maxwellian
associated to the FENE potential is
\begin{equation}\label{M-FENE}
M_i(\ponf) = Z_i^{-1} \left(1-\nicefrac{\abs{\ponf}^2}{b_i}\right)^{b_i/2}, \qquad \ponf \in B(0,\sqrt{b_i}),
\end{equation}
where $Z_i$ is a positive constant. A direct calculation returns that with this
weight $Q_\Theta$ is
\[ Q_\Theta(\ponf) = \Theta + \left(\frac{1}{4}-\frac{1}{b_i}\right) \abs{\ponf}^2 \left(1-\frac{\abs{\ponf}^2}{b_i}\right)^{-2}-\frac{d}{2}\left(1-\frac{\abs{\ponf}^2}{b_i}\right)^{-1}. \]
In this form, it is readily apparent that $Q_1$ is bounded from below in its
domain $B(0,\sqrt{b_i})$ (i.e., (\ref{QBounded}) holds) if $b_i>4$. From the fact
that $\mathfrak{d}(\ponf) = \sqrt{b} - \abs{\ponf}$ for all $\ponf$ in the
domain under consideration it is easy to see that $\mathfrak{d}^2 Q_\Theta$ is
always bounded from below and uniformly continuous up to the boundary. If $b_i
\in (2,4]$, $Q_\Theta$ is never bounded from below, so it takes negative values
and thus the infimum of $\mathfrak{d}^2 Q_\Theta$ is strictly less than zero.
As $\mathfrak{d}^2$ is continuous and positive within the domain yet zero at
its boundary, the existence of a $\Theta$ that makes case (\ref{dsqQBoundedByThat}) hold
is equivalent to demanding that
\[ \lim_{\abs{\ponf}\to\sqrt{b_i}_-} \mathfrak{d}(\ponf)^2 Q_1(\ponf) \in (-1/4,0]. \]
As in the range $b_i \in (2,4]$ that limit is $b_i(b_i/4-1)/4$ we see that the
condition (\ref{dsqQBoundedByThat}) holds there.

Analogously, \eqref{CPAIL-model} and \eqref{partial-maxw} imply that the
Maxwellian associated to the CPAIL potential is
\begin{equation}\label{M-CPAIL}
M_i(\ponf) = Z_i^{-1} \exp\left(-\nicefrac{\abs{\ponf}^2}{6}\right) \left(1-\nicefrac{\abs{\ponf}^2}{b_i}\right)^{b_i/3}, \qquad \ponf \in B(0,\sqrt{b_i}),
\end{equation}
with $Z_i$ a positive constant. Again, a direct calculation yields
\begin{equation*}
Q_\Theta(\ponf) = \Theta - \frac{d}{6} + \frac{\abs{\ponf}^2}{36} + \left(\frac{1}{9}-\frac{2}{3 b_i}\right)
   \abs{\ponf}^2 \left(1-\frac{\abs{\ponf}^2}{b_i}\right)^{-2} - \left(\frac{d}{3} - \frac{\abs{\ponf}^2}{9}\right) \left(1-\frac{\abs{\ponf}^2}{b_i}\right)^{-1}.
\end{equation*}
By arguments similar to those given when considering the FENE potential, we
have that condition (\ref{QBounded}) holds if $b_i > 6$ or if $b_i = 6$ and $d
= 2$; and that condition (\ref{dsqQBoundedByThat}) holds if $b_i \in (3,6]$.
\end{proof}

If two weights $w$ and $\tilde w$ defined on a domain $\Omega$ are
comparable---that is, there exist two positive constants $c_1$ and $c_2$ such
that $c_1 w \leq \tilde w \leq c_2 w$---a number of consequences follow
immediately. As discussed elsewhere, $\LL^2_w(\Omega)$ and $\LL^2_{\tilde
w}(\Omega)$ on the one hand and $\HH^1_w(\Omega)$ and $\HH^1_{\tilde
w}(\Omega)$ on the other will be one and the same algebraically and
topologically. In particular, the hypotheses of \cref{lem:abstractEV} will
be met by the eigenvalue problem
\[ \langle e, v \rangle_{\HH^1_w(\Omega)} = \lambda \langle e, v\rangle_{\LL^2_w(\Omega)} \qquad \forall\,v\in\HH^1_w(\Omega) \]
if, and only if, they are met by the eigenvalue problem
\[ \langle e, v \rangle_{\HH^1_{\tilde w}(\Omega)} = \tilde \lambda \langle e, v\rangle_{\LL^2_{\tilde w}(\Omega)} \qquad \forall\,v\in\HH^1_{\tilde w(\Omega)}. \]
The inf-sup characterization (cf.\ \eqref{EV-infsup}) of the
successive eigenvalues of both problems allow for the bounds
\[ \frac{c_1}{c_2} \lambda_n \leq \tilde\lambda_n \leq \frac{c_2}{c_1} \lambda_n. \]
That is, the bounds \eqref{twoSidedBound} will hold for one set of eigenvalues
if, and only if, they hold for the other. This allows for establishing the
following sufficiency condition for weights defined on two- or
three-dimensional balls, which is in most cases much easier to test than the
conditions of \cref{lem:easyAsymptotics}.

\begin{lemma}\label[lemma]{lem:conditionsForWeyl}
Let $\Omega$ be an open ball in two or three dimensions and let $w$ be a
positive and continuous weight defined on $\Omega$ with the property
\begin{equation*}
\sigma_1 \mathfrak{d}(\ponf)^\alpha \leq w(\ponf) \leq \sigma_2 \mathfrak{d}(\ponf)^\alpha
\end{equation*}
for all $\ponf \in \Omega$ such that $\mathfrak{d}(\ponf) < \delta$, for some
exponent $\alpha > 1$, for some margin $\delta > 0$ and some positive constants
$\sigma_1$ and $\sigma_2$.

Then, the eigenvalues of the problem
\[ \langle e, v \rangle_{\HH^1_w(\Omega)} = \lambda \langle e, v \rangle_{\LL^2_w(\Omega)} \qquad\forall\,v\in\HH^1_w(\Omega) \]
obey the two-sided bounds \eqref{twoSidedBound}.
\end{lemma}
\begin{proof}
If the radius of the ball happens to be $\sqrt{2 \alpha}$ the conditions on $w$
force it to be comparable to the FENE Maxwellian \eqref{M-FENE} and so the
result follows from the above discussion. Otherwise, one just needs to rescale
the domain; this will effect a fixed linear transformation on the eigenvalues,
but will not affect the validity of the bounds \eqref{twoSidedBound} (the
constants involved will change, though).
\end{proof}

\begin{remark}\label[remark]{rem:RussianLiterature}
The eigenvalue problem \eqref{partial-ev} associated with either the FENE or
the CPAIL model falls within what is called \emph{weak degeneracy} case in the
Russian spectral theory literature; i.e., problems of the form: Given $\Omega \subset
\Real^d$, find $(\lambda,u) \in \Real \times (\HH^1_{\mathfrak{d}^\alpha}(\Omega)
\setminus \{0\})$ such that
\begin{equation}\label{ev-VulisSolomjak}
\int_\Omega \left(A \grad u \cdot \grad v + h\,u\,v\right) \mathfrak{d}^\alpha
= \lambda \int_\Omega b\,u\,v\,\mathfrak{d}^\beta \qquad \forall\,v\in\HH^1_{\mathfrak{d}^\alpha}(\Omega),
\end{equation}
where $\alpha-\beta < 2/d$ (see \cite[\S 1]{VS:1974} for the precise statement,
which includes additional conditions on $\Omega$, $A$, $h$, $b$, $\alpha$ and
$\beta$). As, in the FENE and CPAIL versions of \eqref{partial-ev}, the same
weight (the associated Maxwellian) appears in both the left- and right-hand
side bilinear forms, and, in both cases, that weight is bounded from above and
below by powers of $\mathfrak{d}$ (cf.\ \eqref{M-FENE}, \eqref{M-CPAIL}), it
turns out that our problem is equivalent to a problem of the form
\eqref{ev-VulisSolomjak} with $\alpha - \beta = 0$.

The result, according to \cite[Theorem 1.1]{VS:1974} and assuming that $b\geq0$
is that
\begin{equation}\label{fromVulisSolomjak}
\lim_{\lambda \to \infty} \lambda^{-d/2}\#\{n \in \natural\colon \lambda_n < \lambda\} = \frac{1}{(2\sqrt{\pi})^d \Gamma(1+d/2)}\int_{\Omega}\frac{\mathfrak{d}^{-(\alpha-\beta)d/2} b^{d/2}}{\sqrt{\det(A)}}
\end{equation}
(compare this with \eqref{countingLimits}; note also that in \cite{VS:1974} the
statement is made in terms of what in our notation is $1/\lambda$). The problem
with this particular source is that, for a proof, it remits the reader to
either one of two publications. The first, \cite{BirmanSolomjak:1972} proves
related yet not directly applicable results---there is a gap that needs to be
bridged by means, perhaps elementary, that are unknown to us. We have not been
able to get hold of the second, \cite{Tashchiyan:1975} by
G.~M.~Ta\v{s}\v{c}ijan (also romanized as Tashchiyan). However, the latter is
also cited in \cite[Theorem 1]{Tashchiyan:1981}, where a generalization of
\eqref{fromVulisSolomjak} is proved, under the condition (in our notation) $d >
2$.
\end{remark}

\textbf{Acknowledgement.} We are grateful to Professor Marco Marletta (Cardiff University) for helpful
suggestions regarding the Liouville transformation.

\bibliographystyle{pillito}
\bibliography{references}

\end{document}